\documentclass{amsart}

\usepackage[T1]{fontenc}
\usepackage[utf8]{inputenc}
\usepackage[english]{babel}

\usepackage{hyperref}
\usepackage{amsthm, amssymb}
\usepackage{amsmath}
\usepackage{xifthen}
\usepackage{enumitem}
\usepackage{url}
\usepackage{xcolor}
\usepackage{dsfont}

\numberwithin{equation}{section}

\newtheorem{theorem}{Theorem}[section]
\newtheorem{proposition}[theorem]{Proposition}
\newtheorem{lemma}[theorem]{Lemma}
\newtheorem{corollary}[theorem]{Corollary}

\theoremstyle{definition}
\newtheorem{definition}[theorem]{Definition}

\theoremstyle{remark}
\newtheorem{remark}[theorem]{Remark}

\setenumerate{label=\roman*)}

\newenvironment{bigcases}{\left\{\begin{aligned}}{\end{aligned}\right.}
\newcommand{\edintertext}[1]{%
  \noalign{%
    \vskip\belowdisplayshortskip
    \vtop{\hsize=\linewidth#1\par
    \expandafter}%
    \expandafter\prevdepth\the\prevdepth
  }%
}

\newcommand{\wto}{\rightharpoonup}

\newcommand{\sub}{\subset}
\newcommand{\subeq}{\subseteq}

\newcommand{\bfrac}[2]{\frac{\displaystyle#1}{\displaystyle#2}}
\newcommand{\vp}{\varphi}
\newcommand{\ve}{\varepsilon}
\newcommand{\ua}{u_\alpha}
\newcommand{\ben}{\begin{equation}}
\newcommand{\een}{\end{equation}}
\newcommand{\bal}{\begin{aligned}}
\newcommand{\eal}{\end{aligned}}

\newcommand\R{\mathbb{R}}

\newcommand\Sob[1][k]{H^{#1}}
\newcommand{\hSob}[1][k,2]{D^{#1}}

\newcommand\abs[1]{\left\lvert #1 \right\rvert}
\newcommand{\norm}[1]{\left\| #1 \right\|}
\newcommand\Snorm[3][\Omega]{\norm{#3}_{\Sob[#2](#1)}}
\newcommand\Hnorm[3][\R^n]{\norm{#3}_{\hSob[#2](#1)}}
\newcommand{\Lnorm}[3][\R^n]{\norm{#3}_{L^{#2}(#1)}}

\newcommand{\dist}[1]{\operatorname{dist}(#1)}

\DeclareMathOperator{\bigO}{O}
\DeclareMathOperator{\smallo}{o}
\newcommand{\intM}[3][x]{\int_{#2} #3\, d#1}
\newcommand{\inty}[2][\O]{\int_{#1} #2\, dy}
\newcommand{\spanned}[1]{\operatorname{span}\left\{#1\right\}}

\newcommand{\D}{\Delta}

\newcommand{\eps}{\varepsilon}
\renewcommand{\emptyset}{{\text{\Large\o}}}

\renewcommand{\phi}{\varphi}
\renewcommand{\a}{\alpha}
\newcommand{\crit}{{2^\sharp}}
\newcommand{\exc}{\frac{n-2k}{2}}
\newcommand{\sumi}[1][1]{\sum_{i=#1}^N}
\newcommand{\sumij}{\sum_{\substack{i=0,\ldots N\\ j=1,\ldots d_i}}}
\newcommand{\cupj}[1][m]{\bigcup\limits_{j\in \sB_{#1}}}
\newcommand{\sumj}[1][m]{\sum_{j\in \sB_{#1}}}

\renewcommand{\O}{\Omega}
\newcommand{\dO}{\partial\Omega}
\newcommand{\oO}{\overline{\Omega}}
\newcommand{\Dd}{(-\Delta)}
\newcommand{\Va}[1][i]{V^{#1}_\a}
\newcommand{\Ba}[1][i]{B^{#1}_\a}
\newcommand{\xa}[1][i]{x^{#1}_\a}
\newcommand{\ma}[1][i]{\mu^{#1}_\a}
\newcommand{\ta}[1][i]{\theta^{#1}_\a}
\newcommand{\ra}[1][i]{r^{#1}_\a}
\newcommand{\rha}[1][ji]{\rho^{#1}_\a}
\newcommand{\Za}[1][ij]{Z^{#1}_\a}
\newcommand{\la}[1][ij]{\lambda^{#1}_\a}
\newcommand{\sa}[1][ij]{s^{#1}_\a}

\newcommand{\sB}{\mathcal{B}}
\newcommand{\sA}{\mathcal{A}}
\newcommand{\sS}{\mathcal{S}}
\newcommand{\sC}{\mathcal{C}}
\newcommand{\Ker}{\mathcal{K}}
\newcommand{\pBa}[1][i]{\big(\Ba[{#1}](y)\big)}
\newcommand{\Wa}{\mathcal{W}_\a}
\newcommand{\ea}{\eta_\a}
\newcommand{\Oa}[1][m]{\O^{#1}_\a}

\newcommand{\hOa}[1][m]{\hat{\O}^{#1}_\a}

\newcommand{\Ilo}[2][\cap]{I^l_{\O#1 #2}}
\newcommand{\Pa}[1][m]{\Phi^{#1}_\a}
\newcommand{\xmm}{\xa[m]+\ma[m]}

\newcommand{\nss}[2][\eta_\a]{\|#2\|_{\ast\ast, #1}}
\newcommand{\ns}[1]{\|#1\|_{\ast}}

\newcommand{\firstTerm}[1][\a]{(\ea \nss{R_\a} + \eps_{#1} \ns{\phi_\a})\Big(1+\sumi \ta(x)^{-l}\Ba(x) \Big)}

\title[]{A priori bounds for energy-bounded solutions of critical polyharmonic equations}

\date{\today}

\author{Lorenzo Carletti}

\author{Bruno Premoselli}

\address{Universit\'e Libre de Bruxelles, Service d'analyse, CP 218, Boulevard du Triomphe, B-1050 Bruxelles, Belgique.}

\email{bruno.premoselli@ulb.be}

\email{lorenzo.carletti@ulb.be}

\thanks{The first author is supported by the French Community of Belgium as part of the funding of a FRIA grant. The second author was supported by an ARC Avanc\'e 2020 grant and a MIS FNRS grant.}

\begin{document}
\maketitle

\begin{abstract}
We investigate critical polyharmonic equations of the type:
$$ Lu = |u|^{2^\sharp-2} u \quad \text{ in } \O $$ 
with Dirichlet boundary conditions, in a smooth bounded domain $\O$ of $\R^n$. Here $L$ is an elliptic differential operator of integer order $2 \le 2k < n$ whose leading order term is $(-\Delta)^k$ and $2^\sharp = \frac{2n}{n-2k}$ is the critical Sobolev exponent. Our main result establishes, in large dimensions, uniform \emph{a priori} bounds in $C^{2k}(\overline{\Omega})$ for bounded-energy solutions of this problem, that only depend on an upper bound on the energy. We prove this under a coercivity assumption of sorts on the lower-order terms of $L$. Our results are sharp, at least when $k=1$. Our approach uses asymptotic analysis techniques and in the course of the proof we obtain in particular a new global pointwise description of bounded-energy blowing-up solutions for this problem, which is of independent interest. 
\end{abstract}

\section{Introduction}

\subsection{Statement of the results}

Let $\O \sub \R^n, n \ge 3$ be a smooth bounded domain, and let $k$ be an integer satisfying $1 \le k< \frac{n}{2}$. Throughout this paper $L$ will denote an elliptic differential operator of the following form: 
\ben \label{def:L}
Lu= \Dd^k u+ \sum_{l=0}^{p} (-1)^l \nabla^l\big(A_{l}(\nabla^l  u \,)\big)
\een
for a smooth function $u$ in $\Omega$, where $p \in \{0, \cdots, k-1 \}$ denotes a fixed integer and, for $0 \le l \le p$, $A_{l}\in C^0(\overline{\O})$ is a symmetric $(2l,0)$-tensor. In \eqref{def:L} we write classically $\nabla^l u = \big( \partial^{i_1\,\ldots \, i_l}u\big)_{i_1,\ldots, i_l \in \{1,\ldots, n\}}$ for the $(0,l)$-tensor of the partial derivatives of $u$ of order $l\geq 1$; for the precise definition of $L$ by duality we refer to Definition \ref{def:weaksolution} below. In the particular case $k=1$, $L$ as in \eqref{def:L} simply writes as $L = - \Delta + A_0$ for some $A_0 \in C^0(\overline{\Omega})$. We investigate in this paper critical polyharmonic equations for $L$ with Dirichlet boundary conditions in $\O$: 
  \begin{equation}\label{eq:lmmain} 
   \begin{bigcases}
    &Lu = |u|^{\crit-2} u & &\text{in } \O\\
    &|\nabla^l u|=0 & &\text{for } 0 \le l \le k-1 \quad \text{on } \dO,
  \end{bigcases}
  \end{equation}
where $\crit = \frac{2n}{n-2k}$ is the critical exponent for Sobolev embeddings of $H^k(\O)$ into Lebesgue spaces. Problem \eqref{eq:lmmain} has attracted uninterrupted attention since the seminal paper of Br\'ezis--Nirenberg \cite{bn} in the case $k=1$ and has originated a vast literature. For a suitable choice of constant coefficients $A_l$, \eqref{eq:lmmain} includes the case of the Br\'ezis-Nirenberg problem $L = - \Delta- \lambda$ for $\lambda \ge 0$ (when $k=1$) and its higher-order counterparts. For operators $L$ with general non-constant coefficients, existence and multiplicity results for \eqref{eq:lmmain} are in \cite{HebeyVaugon94} ($k=1$) and \cite{GeWeiZhou11} ($k \ge 2$). Existence and multiplicity results for the Br\'ezis-Nirenberg problem when $k=1$ are in \cite{CAP2, CeramiFortunatoStruwe, CeramiSoliminiStruwe, ChenZou, ClappWeth2, DevillanovaSolimini, RoselliWillem, SchechterZou, TavaresYouZou} and the references therein, and for its higher-order counterparts they have been obtained in \cite{BernisGaAzoPeral96, DengWang99, EdFoJa} ($k=2$) and \cite{Grunau95, PucciSerrin90} ($k \ge 3$).

We investigate in this paper compactness and \emph{a priori} bounds results for energy-bounded solutions of \eqref{eq:lmmain}. We let 
\begin{equation} \label{def:A}
 \mathcal{A}_p = \Big \{ \bold{A} = (A_0, \cdots, A_p), A_l \in C^l(\overline{\O}) \text{ if } l \ge 1, A_0 \in C^1(\overline{\O}) \Big \},
 \end{equation}
be the space of coefficients of $L$, where each $A_l$ is a symmetric $(2l,0)$ tensor, that we endow with the norm 
$$ \Vert \bold{A} \Vert_{\mathcal{A}_p} = \sum_{l=1}^p \Vert A_l \Vert_{C^{l}(\overline{\O})} + \Vert A_0\Vert_{C^{1}(\overline{\O})} .$$ 
If $\bold{A} \in \mathcal{A}_p$, $L$ is as in \eqref{def:L}, and $\Lambda >0$ is fixed, we define 
$$\mathcal{S}_{\bold{A}, \Lambda} = \Big \{ u \in H^k_0(\O) \text{ solution of } \eqref{eq:lmmain} \text{ with }  \intM{\O}{ |\nabla^k u|^2} \le \Lambda \Big\}.$$
 Our main result states that, under a non-vanishing assumption on the tensor $A_p$ of highest order, energy-bounded solutions of \eqref{eq:lmmain} possess \emph{a priori} bounds in $C^{2k}(\overline{\Omega})$ that are uniform under perturbations of the coefficients of $L$: 
 
   \begin{theorem} \label{theo:compactness}
  Let $\O$ be a bounded smooth domain of $\R^n, n \ge 3$, and let $1 \le k < \frac{n}{2}$ and $0 \le p \le k-1$ be integers. Let $\bold{A}_* \in \mathcal{A}_p$ be fixed, $\bold{A}_* = (A_0, \cdots, A_p)$, and let $L$ be as in \eqref{def:L}. We assume that $n > 6k - 4p$ and that $\bold{A}_*$ satisfies condition \eqref{cond:signchanging} below. Let $\Lambda >0$. Then there exist positive real numbers $\delta = \delta (n, \Lambda, \bold{A}_*)$ and $C = C(n,\Lambda, \bold{A}_*)$ such that 
$$ \sup_{\Vert \bold{A}_* - \bold{A} \Vert_{\mathcal{A}_p}< \delta} \, \, \sup_{u \in \mathcal{S}_{\bold{A}, \Lambda}} \Vert u \Vert_{C^{2k}(\overline{\O})} \le C. $$ 
  \end{theorem}
 An important consequence of Theorem~\ref{theo:compactness}, which follows from standard elliptic theory, is that the set of energy-bounded solutions of \eqref{eq:lmmain} is compact in strong topologies: 
    \begin{corollary} \label{corol:compactness}
 Let $\O$ be a bounded smooth domain of $\R^n, n \ge 3$, and let $1 \le k < \frac{n}{2}$ and $0 \le p \le k-1$ be integers. Let $\bold{A}_* \in \mathcal{A}_p$ be fixed, $\bold{A}_* = (A_0, \cdots, A_p)$, and let $L$ be as in \eqref{def:L}. We assume that $n > 6k - 4p$ and that $\bold{A}_*$ satisfies condition \eqref{cond:signchanging} below. Let $\Lambda >0$. Then $\mathcal{S}_{\bold{A}_*, \Lambda}$ is a compact subset of $C^{2k}(\overline{\O})$. 
\end{corollary}
We point out that we do not assume $\ker L = \{0\}$ in Theorem \ref{theo:compactness}, but solely that  \eqref{cond:signchanging} is satisfied, and that we assume nothing on $\O$ except that it is regular. Our main assumption in  Theorem \ref{theo:compactness}, condition \eqref{cond:signchanging}, is defined and discussed in detail in subsection \ref{subsec:condition} below. It states that, up to a fixed sign, at each point $x_0 \in \overline{\O}$ the quadratic form associated to $A_p(x_0)$ defines a norm on the space of limiting bubbling profiles for \eqref{eq:lmmain}. This condition should then be understood as a weak coercivity of sorts for $A_p$, up to a sign. A notable fact is that, when $k \ge 2$, condition \eqref{cond:signchanging} allows us to apply Theorem \ref{theo:compactness} to operators $L$ which are \emph{not isotropic}. We refer to Remark \ref{rem:anisotropy} below for a detailed illustration of these facts.

\subsection{Application to Br\'ezis-Nirenberg type problems}

Condition \eqref{cond:signchanging} is easy to verify and is satisfied for a large class of differential operators, including polyharmonic Br\'ezis--Nirenberg operators such as  
\begin{equation} \label{L:BN}
 L = (-\Delta)^k + \lambda (-\Delta)^p, 
 \end{equation}
where $0 \le p \le k-1$ and $\lambda \neq 0$ (see subsection \ref{subsec:condition} for an easy proof of this fact). When $k=1$  it suffices to take $A_0 \equiv-\lambda$ in the definition of $L$ to recover the celebrated Br\'ezis-Nirenberg problem. For a general $k \ge 1$ and when $p=0$, \eqref{L:BN} is the so-called Pucci--Serrin problem \cite{PucciSerrin90}, for which nonexistence of solutions when $\lambda > 0$ and $\O$ is starshaped was proven in \cite{PucciSerrin86}. We do not assume that our domains are starshaped in the following so we will not restrict to the case $\lambda \le 0$. For polyharmonic Br\'ezis--Nirenberg problems \eqref{L:BN}, Theorem~\ref{theo:compactness} implies the following: 
  \begin{theorem} \label{corol:BN}
  Let $\O$ be a bounded smooth domain of $\R^n, n \ge 3$, and let $1 \le k < \frac{n}{2}$ and $0 \le p \le k-1$ be integers. Let $\lambda \neq 0$ and $E >0$ be fixed. We assume that $n > 6k - 4p$. There exist positive real numbers $\delta = \delta(\lambda, E)$ and $C = C(\lambda, E)$ such that for any $|\mu-\lambda|< \delta$ and for any solution $u \in H^k_0(\O)$ of 
 \begin{equation*} 
  (-\Delta)^k u + \mu (-\Delta)^p u = |u|^{\crit-2} u \quad \text{ in } \O 
  \end{equation*}
satisfying $\int_{\Omega} |\nabla^k u|^2 dx \le E$,  we have 
$$  \Vert u \Vert_{C^{2k}(\overline{\O})} \le C. $$ 
  \end{theorem}

  \medskip

Theorem \ref{corol:BN} only requires $\lambda \neq 0$ and does not assume a sign on $\lambda$. With the help of Theorem \ref{corol:BN} we are able to answer several open questions related to the existence of blowing-up solutions of bounded energy for polyharmonic Br\'ezis-Nirenberg problems such as \eqref{L:BN} and to classify the possible concentration values in large dimensions. We say that $\bar{\lambda} \in \mathbb{R}$ is a \emph{concentration value} for \eqref{L:BN} if there exists a family $(u_\lambda)_{\{|\lambda-\bar{\lambda}| \ll 1\}}$ of solutions of \eqref{L:BN} satisfying 
$$ \limsup_{ \lambda  \to \bar{\lambda}}\int_{\Omega} |\nabla^k u_\lambda|^2 dx < + \infty \quad \text{ and } \quad \limsup_{ \lambda \to \bar{\lambda}} \Vert u_\lambda \Vert_{C^0(\overline{\O})} = + \infty$$
(we do not restrict to $\bar{\lambda} \ge 0$). Theorem \ref{corol:BN} then immediately implies the following: 

\begin{corollary} \label{corol:BN2}
Let $\O$ be a bounded smooth domain of $\R^n, n \ge 3$, and let $1 \le k < \frac{n}{2}$ and $0 \le p \le k-1$ be integers. Assume that $n > 6k - 4p$. Then the only possible concentration value for \eqref{L:BN} is $\lambda = 0$. 
\end{corollary}
Corollary \ref{corol:BN2} provides the first general classification result of possibly sign-changing blowing-up solutions for \eqref{L:BN}. In the special case $k=1$, where $p=0$, it shows that, when $n \ge 7$, the only concentration value for the Br\'ezis-Nirenberg problem is $\lambda=0$: it implies therefore that bubble-tower solutions and boundary concentration for the Br\'ezis-Nirenberg problem may only occur as $\lambda \to 0$. Corollary \ref{corol:BN2} in particular answers many questions raised in \cite{LiVairaWeiWu2}: it negatively answer \cite[Open Question 6]{LiVairaWeiWu2}, proves \cite[Conjecture 1]{LiVairaWeiWu2} and partially answers \cite[Open Question 8]{LiVairaWeiWu2}. 

\subsection{Discussion of our results}

When $k=1$, $L$ is of the form $L= - \Delta + A_0$: the condition $n > 6k - 4p$ simply becomes $n \ge 7$ and \eqref{cond:signchanging} amounts to require that $A_0$ never vanishes in $\overline{\Omega}$ (see subsection \ref{subsec:condition} below). The only previous compactness result on domains for \eqref{eq:lmmain}, analogous to our Theorem \ref{theo:compactness}, had been obtained when $k=1$ in \cite{DevillanovaSolimini} (still when $n \ge 7$), and only covered the case where $A_0$ is a negative constant. When $k=1$, therefore, Theorem \ref{theo:compactness} is new when $A_0$ is either a positive constant or a non-constant function in $\overline{\O}$ --- in particular when $A_0$ is a positive function in $\overline{\Omega}$. As such, it has to be understood as the euclidean counterpart of similar compactness results on closed manifolds \cite{DruetJDG, PremoselliVetois2}. When $k \ge 2$ our results are entirely new and provide the first general compactness result for equations such as \eqref{eq:lmmain}. Prior to this, a blow-up analysis for sign-changing radial solutions in the unit ball had been achieved in \cite{RobertPS} and partial compactness results on closed manifolds had been obtained in \cite{Robert:localisation:bubbling}. Similar results when $k=1$ and $L$ contains additional nonlocal terms very recently appeared in \cite{BBD2026}.

\medskip


When $k=1$, the assumptions of Theorems~\ref{theo:compactness} and~\ref{corol:BN} are sharp. First, \eqref{cond:signchanging} is necessary: families of solutions of \eqref{eq:lmmain} with $L_\lambda = - \Delta - \lambda$ which blow-up as $\lambda \to 0_+$ have for instance been constructed when $n \ge 7$ in \cite{MussoPistoia2, PistoiaRocci, PistoiaVaira2, Premoselli12, Rey}. Second, the dimensional restriction $n \ge 7$ is also necessary: when $4 \le n \le 6$ sign-changing families of solutions of \eqref{eq:lmmain} with $L_\lambda = - \Delta - \lambda$ which blow-up as $\lambda$ approaches some $\lambda_0 >0$ (which is an eigenvalue of $-\Delta$ when $n=4,5$) have been constructed in \cite{LiVairaWeiWu,LiVairaWeiWu2}. As a remark, the solutions in  \cite{LiVairaWeiWu,LiVairaWeiWu2} do not have minimal energy, since sign-changing solutions of \eqref{eq:lmmain} with minimal energy have been recently shown to be unconditionally compact when $n=4,5$ in \cite{CAP1}. When $n=3$, and again for $L_\lambda = - \Delta- \lambda$, it is known that when $\lambda$ approaches some $\lambda_*>0$ defined by a vanishing mass condition blowing-up solutions of \eqref{eq:lmmain} exist  \cite{bn, Druetdim3, MussoSalazar}.  Finally, the energy bound is also necessary. One of the main novelties of Theorem~\ref{theo:compactness}, indeed, is that it covers the general case of sign-changing solutions. If one considers \emph{positive} solutions of \eqref{eq:lmmain} one may remove the energy bound in Theorem \ref{theo:compactness}, at least when $n \ge 4$, provided $A_0$ and all its perturbations are nonpositive, see \cite{LaurainKonig1, LaurainKonig2}. For sign-changing solutions, however, no blow-up analysis is possible to this day without assuming a priori energy bounds, as the set of sign-changing bubbling profiles (i.e.  solutions of \eqref{eq:critRn} below) is very large. We strongly believe that the assumptions of Theorem \ref{theo:compactness} are also sharp when $k \ge 2$, but the results in \cite{LiVairaWeiWu,LiVairaWeiWu2, MussoPistoia2, PistoiaRocci, PistoiaVaira2, Premoselli12, Rey} have not been adapted to the polyharmonic case yet.

\medskip

\subsection{Strategy of proof and organisation of the paper}

Theorem \ref{theo:compactness} was previously proven for $k=1$ and when $L = - \Delta +A_0$ and $A_0$ is a negative constant in \cite{DevillanovaSolimini}. The proof of \cite{DevillanovaSolimini} relied on the observation that, if $u$ solves \eqref{eq:lmmain}, $|u|$ is a weak subsolution to a close variant of \eqref{eq:lmmain}, and used the latter to prove integral bounds for $u$ on annuli corresponding to scales where the solution may concentrate. When $k \ge 2$ this approach trivially fails since the absolute value $|u|$ of a function $u \in H^k_0(\O)$ is not in $H^k_0(\O)$ in general, and $(-\Delta)^k |u|$ does not easily relate to $(-\Delta)^k u$ in a weak sense. In this paper we thus adopt a new and radically different approach than \cite{DevillanovaSolimini}, which allows us to treat simultaneously the case of every $k \ge 1$. Our proof of Theorem \ref{theo:compactness} relies on asymptotic analysis techniques. Equation \eqref{eq:lmmain} has a critical nonlinearity on its right-hand side, which implies that its solutions are likely to \emph{blow-up}. 
We prove by a contradiction argument, that no blow-up may occur when $n > 6k-4p$ and \eqref{cond:signchanging} holds. 

This approach requires us to first prove a very accurate description of the behavior of energy-bounded blowing-up sequences of (perturbations of) \eqref{eq:lmmain}, which is of independent interest. This description is the subject of Theorem \ref{prop:main} below: it focuses most of our efforts in this work and is the main novelty of this paper. Theorem \ref{prop:main} should be understood as a quantitative version of classical $H^k_0(\O)$ compactness results \emph{\`a la} Struwe \cite{Struwe} (see \cite{Maz17} for the polyharmonic case): it improves classical energy decompositions into global pointwise estimates, where the difference between blowing-up solutions and the bubbling profiles that they develop --- which we call a \emph{bubble-tree} (see Section \ref{sec:bulles2} for precise definitions) --- is precisely estimated in strong spaces. The estimates that we prove are valid in $\overline{\O}$ and thus include the boundary, and an important feature of Theorem \ref{prop:main} is that it holds true \emph{regardless of the bubbling configuration}: we do not assume anything on the bubbling profiles that may appear\footnote{There is an infinite number of them since our solutions may change sign, and so may the bubbling profiles} and we allow for bubbles at the boundary, whose existence is not ruled out yet when $k \ge 2$. Theorem \ref{prop:main} is the first instance of pointwise estimates up to the boundary for sign-changing solutions of equations like \eqref{eq:lmmain} in a general bounded domain, and it builds upon very recent developments in the field \cite{CarJLMS, CarRob25, Pre24, Robert:localisation:bubbling, Robert:C0:ordre:k}. The final part of the proof of Theorem \ref{theo:compactness} relies on a polyharmonic Pohozaev identity, and the contradiction is obtained by computing the contribution of lower-order terms in $L$ thanks to assumption \eqref{cond:signchanging}. Estimating each term in the Pohozaev identity requires precise \emph{pointwise} estimates on the solutions, and this is the reason we prove Theorem \ref{prop:main}.

\medskip

The organisation of the paper is as follows. In Section \ref{sec:bulles} we investigate some properties of the bubbling profiles that arise in $\R^n$ or $\R^n_+$. We prove sharp pointwise decay  for these bubbles as well as for solutions of their linearised equations. Section \ref{sec:bulles2} introduces the notion of bubble-tree and analyses the interactions between the different bubbling profiles arising during the blow-up. Sections \ref{sec:lin} and \ref{sec:nonlinear} are devoted to the proof of Theorem \ref{prop:main}. Section \ref{sec:lin} is the core of the analysis of this paper and we prove in it a quantitative invertibility result for the linearisation of \eqref{eq:lmmain} around a fixed bubble-tree configuration. In Section \ref{sec:nonlinear} we apply this result to prove Theorem \ref{prop:main} by a nonlinear perturbative argument. Section \ref{sec:pohozaev} is devoted to an analysis of the Pohozaev identity and concludes the proof of our main theorem. Finally the Appendix contains several technical results that are used throughout this paper, in particular a folklore result for Struwe-type decompositions whose proof for sign-changing solutions we could not find in the literature.

\bigskip

\textbf{Acknowledgements:} the second author is indebted to Saikat Mazumdar for useful discussions concerning Appendix A.

\section{Preliminary results} \label{sec:bulles}

\par Let $U \subseteq \R^n$ be any smooth domain of $\R^n, n \ge 3$, not necessarily bounded. We define $D^{k,2}(U)$ as the closure of $C^\infty_c(U)$ under the following norm
\begin{equation} \label{norm:Dk2}
  \Hnorm[U]{k,2}{u}^2 := \intM{U}{|\nabla^k u|^2}.
\end{equation}
When $U$ is bounded and smooth, \eqref{norm:Dk2} is the classical $H^k_0(U)$ norm and we will denote it by $\Vert \cdot \Vert_{H^k_0(U)}$. We endow $D^{k,2}(U)$ (and $\Sob_0(U)$ when $U$ is bounded) with the scalar product induced by \eqref{norm:Dk2}. We recall the following definition:
\begin{definition} \label{def:weaksolution}
 Let $U \subseteq \R^n$ be a domain with smooth boundary, $0 \le p \le k-1$ be an integer and let $(A_{l})_{0 \le l \le p}$ be continuous tensors in $\overline{U}$. Let $L$ be defined as in \eqref{def:L} and let $u\in D^{k,2}(U)$. We say that $u$ is a weak solution of \eqref{eq:lmmain} if 
  \[  \intM{U}{\big \langle \Dd^{k/2}u, \Dd^{k/2}\phi \big \rangle} + \sum_{l=0}^{p} \intM{U}{A_{l}(\nabla^l u,\nabla^l \phi)} = \intM{U}{|u|^{\crit-2}u\phi}
  \]
  for all $\phi \in C^\infty_c(U)$, where we have let $\Dd^{k/2} = \nabla\Dd^m$ when $k=2m+1$ is odd and where $\big \langle \Dd^{k/2}u, \Dd^{k/2}\phi \big \rangle$ denotes the (pointwise) scalar product in $\R$ or $\R^n$ depending on the parity of $k$.
\end{definition}
Throughout the rest of this paper, the only unbounded example of of open sets that we will investigate will be $\R^n$ and the half-space $\R^n_+$, that we define as  
$$\R^n_+ = \{x_1>0\}.$$

\subsection{Limiting equations and the compactness condition \eqref{cond:signchanging}} \label{subsec:condition}

We study in this subsection two equations that will arise as limiting problems in the analysis of the blow-up behavior of \eqref{eq:lmmain}. They are: 
\begin{equation}\label{eq:critRn}
  \Dd^k u = |u|^{\crit-2}u \quad \text{in }\R^n, \qquad u \in \hSob(\R^n),
\end{equation}
which will correspond to \emph{interior} bubbling solutions, and
\begin{equation}\label{eq:crithf}
\left \{  \begin{aligned}
    \Dd^k u &= |u|^{\crit-2}u \quad  \text{in }\R^n_+\\
    |\nabla^l u|& = 0  \quad \text{on }\partial\R^n_+,  \text{ for } l =0,\ldots, k-1\\
  \end{aligned} \right. ,  \quad u \in \hSob(\R^n_+)
\end{equation}
which will correspond to \emph{boundary} bubbling solutions. We recall that $2^\sharp = \frac{2n}{n-2k}$. Solutions to \eqref{eq:critRn} and \eqref{eq:crithf} are defined in a weak sense as in Definition \ref{def:weaksolution} and regularity theory for critical polyharmonic equations (see \cite{Maz16}) shows that solutions to \eqref{eq:critRn} and \eqref{eq:crithf} are of class $C^{2k}$ in $\R^n$ (resp. $\overline{\R^n_+}$). Concerning \eqref{eq:critRn}, positive solutions are classified and are, up to scaling and translations, the \emph{standard bubble}
  \begin{equation}\label{def:eucBub}
    B(x) = \big(1+ a_{n,k}|x|^2\big)^{-\exc},
  \end{equation}
  where $a_{n,k} = \Big(\Pi_{l=-k}^{k-1}(n+2l)\Big)^{-1/k}$, as proven in \cite{WeiXu99}. The function $B$ achieves the optimal constant for the critical Sobolev embedding $\hSob(\R^n) \subset L^\crit(\R^n)$. By contrast, any sign-changing solution $u$ of \eqref{eq:critRn} satisfies $\intM{\R^n}{|u|^\crit} > 2\intM{\R^n}{|B|^\crit}$ (see \cite[Lemma 5]{GeWeZh11} for the large inequality and \cite[Lemma A.1]{HumbertPetridesPremoselli} for the strict one) and there is a plethora of sign-changing solutions with arbitrarily large energy: see e.g. \cite{Ding, DelPinoMussoPacardPistoia1, DelPinoMussoPacardPistoia2, MedinaMusso, MedinaMussoWei} for the case $k=1$ \cite{BartschSchneiderWeth, GuoLiWei13} when $k \ge 2$. Concerning \eqref{eq:crithf}, if follows from the classical Pohozaev identity \cite{poho} that no non-zero solutions exist when $k=1$, and it was proven in \cite[Theorem 4]{ReiWet09} that when $k\ge 2$ no non-zero bounded non-negative solutions exist. The question of whether non-trivial sign-changing solutions of \eqref{eq:crithf} exist for $k \ge 2$ remains open to this day.   

  \smallskip

 In the next sub-section we will prove the following decay for solutions of \eqref{eq:critRn} and \eqref{eq:crithf}: 
  \begin{proposition}\label{prop:ctlsol}
    Let $u$ be a solution of \eqref{eq:critRn} (resp. \eqref{eq:crithf}). There exists $C>0$ such that for all $x\in \R^n$ (resp. $x \in \overline{\R}^n_+$) and for any $0 \le l \le 2k$,
  \begin{equation}\label{eq:Idcrit}
    |\nabla^l u(x)| \leq C (1+ |x|)^{2k-n-l}.
  \end{equation}
  \end{proposition} 
Proposition~\ref{prop:ctlsol} is sharp for \eqref{eq:critRn} since positive solutions given by \eqref{def:eucBub} decay exactly as $|x|^{2k-n}$ at infinity. With Proposition \ref{prop:ctlsol} we may now introduce our main assumption in the statement of Theorem~\ref{theo:compactness}.  Let $\Omega$ be a smooth bounded domain, let $0 \le p \le k-1$ be fixed and $\bold{A} = (A_0, \cdots, A_p) \in \mathcal{A}_p$ be as in \eqref{def:A}. For $x_0 \in \overline{\O}$ fixed we define the following integral: 
  \begin{equation*} 
 I_{\bold{A}}(x_0, u) =
\left \{
\begin{aligned}
 & \intM{\R^n}{A_p(x_0) \big( \nabla^p u, \nabla^p u \big)}  \quad &  \text{ if } u \text{ solves } \eqref{eq:critRn}  \\ 
 & \intM{\R^n_+}{A_p(x_0) \big( \nabla^p u, \nabla^p u \big)} \quad &  \text{ if } u \text{ solves } \eqref{eq:crithf}.   
\end{aligned}
\right. 
\end{equation*}
Here we see $A_p(x_0)$ as a $(2p,0)$-tensor with constant coefficients in $\R^n$. Estimate \eqref{eq:Idcrit} shows that when $n > 4k-2p$ any solution $u$ of \eqref{eq:critRn} (resp. \eqref{eq:crithf}) satisfies $\nabla^p u \in L^2(\R^n)$ (resp. $\nabla^p u \in L^2(\R^n_+)$): in particular, if $n \ge 6k-4p$, $I_A(x_0,u)$ is finite. We consider the set of all possible values for $I_{\bold{A}}(x_0,u)$ as $x_0$ runs through $\overline{\O}$ and $u$ runs over the set of all non-zero finite-energy solutions of \eqref{eq:critRn} and \eqref{eq:crithf}: 
$$ \mathcal{I}_{\bold{A}} = \Big \{ I_{\bold{A}}(x_0, u), \, x_0 \in \overline{\O} \text{ and } u \neq 0 \text{ solves } \eqref{eq:critRn} \text{ or } \eqref{eq:crithf}\Big \} .$$
Our main assumption in Theorem \ref{theo:compactness} that will ensure uniform bounds is the following one: 
 \begin{equation}\label{cond:signchanging}
\begin{aligned}
\text{ Either} \quad  \mathcal{I}_{\bold{A}}  \subseteq (0, + \infty) \quad \text{ or }  \quad \mathcal{I}_{\bold{A}}  \subseteq (- \infty, 0 ).
\end{aligned}
\end{equation}
Thus we assume that, for any $x_0 \in \overline{\O}$ and any non-zero finite-energy solution $u$ of \eqref{eq:critRn} and \eqref{eq:crithf}, all the integrals $I_{\bold{A}}(x_0,u)$ are non-zero and have the same sign. As a first result we observe that \eqref{cond:signchanging} is satisfied for equations of Br\'ezis-Nirenberg type and deduce Theorem \ref{corol:BN}: 

\begin{proof}[Proof of Theorem \ref{corol:BN} assuming Theorem \ref{theo:compactness}]
Let $0 \le p \le k-1$ and $\lambda \neq 0$ be fixed. For $\mu \neq 0$ we let $A_{p,\mu} = \mu (\cdot, \cdot) $, where $(\cdot, \cdot)$ denotes the pointwise scalar product on $p$-symmetric tensors, and we let $\bold{A}_\mu = (0, \cdots, 0, A_{p,\mu}) \in \mathcal{A}_p$. For any $x_0 \in \overline{\O}$ and any non-zero $u$ solving \eqref{eq:critRn} or \eqref{eq:crithf} we have $I_{\bold{A}_\lambda}(x_0,u) = \lambda  \intM{U}{| \nabla^p u|^2}$, where $U= \R^n, \R^n_+$, and thus $\bold{A}_\lambda$ satisfies \eqref{cond:signchanging}. Theorem \ref{theo:compactness} thus applies to $L$ associated to $\bold{A}_\lambda$ as in \eqref{def:L}. Integration by parts in $H^k_0(\O)$ shows that 
$$L = (-\Delta)^k + \lambda (-\Delta)^p, $$
and Theorem \ref{corol:BN} now follows from Theorem \ref{theo:compactness}. 
\end{proof}

\begin{remark} \label{rem:anisotropy}
Condition  \eqref{cond:signchanging} allows $L$ to have anisotropic differential terms of lower order, and Theorem \ref{theo:compactness} does not only apply to operators like \eqref{L:BN}. For simplicity we illustrate this when $k=2$ and $p=1$, but this example easily generalises to any order $k \ge 2$ and $0 \le p \le k-1$. Let $1 \le d <n$ be fixed and let $a_1, \cdots, a_d$ be positive functions in $\overline{\O}$. Let, for $x \in \overline{\O}$,
$$ A_1(x) = \text{Diag}(a_1(x), \cdots, a_d(x), 0, \cdots, 0), $$
let $A_0 \in C^1(\overline{\O})$ be any function and let $\bold{A} = (A_0, A_1) \in \mathcal{A}_1$. Then $L = (-\Delta)^2 - \sum_{i=1}^d \partial_i \big( a_i \partial_i\big) + A_0$ and, for any $x_0 \in \overline{\O}$ we have 
$$ I_{\bold{A}}(x_0,u) = \sum_{i=1}^d a_i(x_0)  \intM{U}{ (\partial_i u)^2} $$
where $U= \R^n$ or $\R^n_+$ depending on whether $u$ solves \eqref{eq:critRn} or \eqref{eq:crithf}. The latter is positive and $\bold{A}$ thus satisfies \eqref{cond:signchanging}. The same argument applies to the case where all the $a_i$ are negative in $\overline{\O}$. 
\end{remark}

\subsection{Proof of Proposition~\ref{prop:ctlsol}}

We prove Proposition~\ref{prop:ctlsol} in what follows. The strategy of proof for the full space $\R^n$ or the half-space $\R^n_+$ is the same: we first transform \eqref{eq:critRn} and \eqref{eq:crithf} into critical equations on spaces of compact closure via a well-chosen conformal map, we then obtain global bounds for solutions on the compact space, and we translate back these bounds into decay on the original space. We will need some definitions in the half-space case $\R^n_+$. Following the arguments in \cite{ReiWet09}, we define a variant of the Kelvin transform that sends the unit ball to the upper half space. In the following we will let $B = B(0,1)$ be the unit ball in $\R^n$ and $e_1=(1,0,\ldots,0)\in \R^n$. We let  
$$   \Phi : \left \{  \begin{aligned}  
   B  & \to \R^n_+\\
    y &\mapsto \frac{y + e_1}{|y+e_1|^2}-\frac{e_1}{2}.
  \end{aligned} \right. $$
  It is easily seen that $\Phi$ is smooth and bijective and sends $\partial B \backslash \{-e_1\}$ to $\{ x_1 = 0 \}  = \partial \R^n_+$. For $y \in B$, straightforward computations show that  $ D \Phi(y) = \frac{1}{|y+e_1|^2} \big( Id - 2 M)$
  where $M = \frac{1}{|y+e_1|^2} (y+e_1) {}^{t}(y+e_1)$ is a rank-one matrix. Thus 
  $$ |\det D \Phi(y)| = |y+e_1|^{-2n} $$ 
  and, if $\xi$ denotes the euclidean metric both in $\R^n_+$ and $B$, and for $y \in B$, we have $ \Phi^* \xi(y) = \frac{1}{|y+e_1|^4} \xi$. For $u\in L^\crit(\R^n_+)$ we define its Cayley transform as follows: 
  \begin{equation}\label{def:Cayley}
    u^*(y) = |\det D \Phi(y)|^{\frac{n-2k}{2n}}\, u\big(\Phi(y)\big) =   |y+e_1|^{2k-n}\, u\big(\Phi(y)\big) \quad \text{ a.e. } y\in B.
  \end{equation}
A simple change of variables shows that  $\Phi$ leaves the critical norm invariant:
\begin{equation}\label{eq:invdeu}
  \int_{B}{|u^*(y)|^\crit}dy =   \intM[x]{\R^n_+}{|u(x)|^\crit}.
\end{equation}
If $U$ denotes either the unit ball $B$ or the half-space $\R^n_+$ we write $G_{U} \in L^1_{loc}(U \times U)$ for the Green's function of the Dirichlet problem 
\[  \begin{cases}
  (-\D)^k u = f & \text{in }U\\
  |\nabla^l u|=0 & \text{on } \partial U, \quad \text{for } l=0,\ldots, k-1 
\end{cases}.
\] 
The arguments in \cite[Section 3]{ReiWet09} show that $G_B$ and $G_{\R^n_+}$ are positive and are respectively given by: 
$$ \begin{aligned}
  G_B(x,y) & = \kappa_{k,n} |x-y|^{2k-n} \int_1^{\big( 1 + \psi(x,y) \big)^{\frac12}} (t^2-1)^{k-1} t^{1-n}dt  \quad  \text{ for }~x\neq y \text{ in } B,\\
 G_{\R^n_+}(x,y) & = \kappa_{k,n} |x-y|^{2k-n} \int_1^{\big( 1 + \psi_\infty(x,y) \big)^{\frac12}} (t^2-1)^{k-1}t^{1-n}dt  \quad  \text{ for }~x\neq y \text{ in }  \R^n_+,
\end{aligned} $$ 
where $\kappa_{k,n}>0$ is some numerical constant depending only on $n$ and $k$ and where 
$$ \psi(x,y) = \frac{(1-|x|^2)(1-|y|^2)}{|x-y|^2} \quad \text{ and } \quad \psi_\infty(x,y) = \frac{4 x_1 y_1}{|x-y|^2} . $$
These formulas can also be derived from the formulas in \cite[Formula (4.7)]{GazGruSw10} or \cite{Bog}. Direct computations using the expression of $\Phi$ show that 
\begin{equation} \label{eq:exprPhi}
 \big| \Phi(x) - \Phi(y) \big| = \frac{|x-y|}{|x+e_1||y+e_1|} \quad \text{ for } x,y \in B ,
 \end{equation}
and hence that 
$$ \psi_\infty \big( \Phi(x), \Phi(y) \big) = \psi (x,y) \quad \text{ for } x,y \in B. $$
The latter, together with the expression of $G_B$ and $G_{\R^n_+}$ and again with \eqref{eq:exprPhi}, shows that for any $x \neq y \in B$ the following holds: 
\begin{equation}\label{eq:GBGR}
  G_B(x,y) = |x+e_1|^{2k-n}|y+e_1|^{2k-n}G_{\R^n_+}(\Phi(x),\Phi(y)). 
  \end{equation}
Let $ v \in C^\infty_c(\R^n_+)$. Since $v$ is supported away from $\partial \R^n_+$, it is easily seen that its Cayley transform satisfies $v^* \in C^\infty_c(B)$. A representation formula for $v^*$ then shows that for any $x \in B$, 
$$ \begin{aligned}
  v^*(x)  & =  \int_{B} G_B(x,y) (-\Delta)^k v^*(y) dy \\
 & =  |x+e_1|^{2k-n}  \int_{\R^n_+} G_{\R^n_+} \big( \Phi(x), z \big) \big| \Phi^{-1}(z) +e_1 \big|^{n+2k} \big[(-\Delta)^k v^* \big]\big( \Phi^{-1}(z)\big) dz, 
 \end{aligned}$$ 
where the last equality follows from \eqref{eq:GBGR} and a change of variables. Identifying with the definition of $v^*$ then yields 
$$ \begin{aligned} 
v \big( \Phi(x) \big) & = \int_{\R^n_+} G_{\R^n_+} \big( \Phi(x), z \big) \big| \Phi^{-1}(z) +e_1 \big|^{n+2k} \big[(-\Delta)^k v^* \big]\big( \Phi^{-1}(z)\big) dz \\
& = \int_{\R^n_+} G_{\R^n_+} \big( \Phi(x), z \big) (-\Delta)^k v(z) dz, 
\end{aligned}  $$ 
where the last equality is simply a representation formula for $v$. Since this equality is true for every $x \in B$ we obtain that for any $y \in B$ and any $ v \in C^\infty_c(\R^n_+)$,
\begin{equation} \label{eq:laplaciens}
(-\Delta)^k v^*(y) = \frac{1}{\big| y +e_1 \big|^{n+2k} }(-\Delta)^k v \big( \Phi(y) \big). 
\end{equation}
A consequence of \eqref{eq:laplaciens} is that the $D^{k,2}$ norm is also preserved via $\Phi$: integrating \eqref{eq:laplaciens} by parts against $v^*$ and using the expression of $v^*$ indeed yields 
\begin{equation}\label{eq:invdeuderivees}
\int_B \big| (- \Delta)^{\frac{k}{2}} v^*(y)\big|^2 dy = \int_{\R^n_+} \big| (- \Delta)^{\frac{k}{2}} v(x)\big|^2 dx. 
\end{equation}
We are now in position to prove Proposition~\ref{prop:ctlsol}.

\begin{proof}[Proof of Proposition~\ref{prop:ctlsol}]
We first consider the entire space case. We adapt the proof in \cite{Pre24} for the case $k=1$. Let  $u \in D^{k,2}(\R^n)$ be a solution of \eqref{eq:critRn}. Let, for $x \in \R^n$, $U(x)^{\frac{4}{n-2k}} = 4 (1+|x|^2)^{-2}$. Let $\pi: \mathbb{S}^n \backslash \{N\} \to \R^n$ be the stereographic projection from the North pole and let, for $y \in \mathbb{S}^n$, $\tilde{u}(y) = \frac{u}{U}(\pi(y))$. If $g_0$ denotes the round metric on the sphere it is well-known that $(\pi^{-1})^{*} g_0 (x) = U(x)^{\frac{4}{n-2k}} \xi$ for every $x \in \R^n$. We recall that $(-\Delta)^k$ is the Graham--Jenne--Mason--Sparling operator \cite{GJMS} for the euclidean metric in $\R^n$ and we denote by $P_{g_0}$ the GJMS operator of order $2k$ on the round sphere $(\mathbb{S}^n, g_0)$.  The conformal invariance of $P_{g_0}$ shows that, for any $y \in \mathbb{S}^n$, $\tilde{u}$ satisfies
\begin{equation} \label{confinvRn}
 P_{g_0} \tilde{u}(y) = U(\pi(y))^{1 - \crit} (-\Delta)^k u (\pi(y)) = |\tilde u(y)|^{\crit -2}  \tilde u(y) \quad \text{ in } \mathbb{S}^n.  
 \end{equation}
$P_{g_0}$ is explicit and writes as $ P_{g_0} = \prod_{j=1}^k \Big( \Delta_{g_0} + \frac{(n+2j-2)(n-2j)}{4}  \Big)$
(see e.g. \cite[Proposition 7.9]{FeffermanGraham}) and is thus coercive in $H^k(\mathbb{S}^n)$. A consequence of the first equality in \eqref{confinvRn} and of the assumption $u \in D^{k,2}(\R^n)$ is that $\tilde u \in H^k(\mathbb{S}^n)$. With \eqref{confinvRn} we may thus apply the regularity theory of \cite[Proposition 8.3]{Maz16} and we obtain that $\tilde u  \in C^{2k}(\mathbb{S}^n)$. Since, for all $x \in \R^n$, $u(x) = U(x)\tilde{u}(\pi^{-1}(x))$, we obtain that $|u(x)| \le C (1+|x|)^{2k-n}$ for some numerical constant $C$ which does not depend on $x$. This proves Proposition~\ref{prop:ctlsol} in the case $l = 0$. The proof for $1 \le l \le 2k$ now follows from local elliptic estimates in $\R^n$. 

\smallskip

We now consider the half-space case.   Let $u\in \hSob(\R^n_+)$ be a weak solution of \eqref{eq:crithf} and let $u^*$ be its Cayley transform. By Sobolev's inequality, $u \in L^{\crit}(\R^n_+)$ and hence \eqref{eq:invdeu} and  \eqref{eq:invdeuderivees} show that $u^* \in H^k(B)$. It is a simple observation that $u^* \in H^k_0(B)$: this follows from the definition of $H^k_0(B)$ as the closure of $C^\infty_c(B)$, the fact that if $ \vp \in C^\infty_c(\R^n_+)$ then $\vp^* \in C^\infty_c(B)$, and \eqref{eq:invdeuderivees}. A change of variable using the weak formulation of \eqref{eq:crithf} and the properties of $\Phi$ shows, by using \eqref{eq:laplaciens}, that $u^*$ satisfies weakly 
  \[  \begin{cases}
  (-\D)^k u^* = |u^*|^{\crit-2}u^* & \text{in }B\\
  |\nabla^l u^*|=0 & \text{for } l=0,\ldots, k-1, \quad \text{on } \partial B
\end{cases}
\] 
in the sense of Definition \ref{def:weaksolution}. Standard elliptic regularity theory then shows that $u^* \in C^{2k}(\overline{B})$ and that there exists $C>0$ such that
\begin{equation} \label{contderustar}
  \Vert \nabla^l u^* \Vert_{L^\infty(B)}\leq C 
\end{equation} 
for every $l=0,\ldots 2k$. Going back to the definition of $u^*$ we have, for any $x\in \R^n_+$, 
\begin{equation}\label{exprustar}
u(x) = \big| \Phi^{-1}(x) + e_1\big|^{n-2k} u^* \big( \Phi^{-1}(x) \big). 
\end{equation}
Using the definition of $\Phi$ it is easily seen that for any $x \in \R^n_+$ we have 
$$ \Phi^{-1}(x) = \frac{x+ \frac{e_1}{2}}{\big| x+ \frac{e_1}{2}\big|^2} - e_1 \quad \text{ and } \quad  \big| \Phi^{-1}(x) + e_1\big| = \frac{1}{\big|x + \frac{e_1}{2}\big|} \le \frac{C}{1+|x|} $$
for some numerical constant $C$ that does not depend on $x$. With \eqref{contderustar} we thus get 
$$|u(x)| \le \frac{C}{(1+|x|)^{n-2k}},$$
which proves Proposition~\ref{prop:ctlsol} in the case $l = 0$. The proof for $1 \le l \le 2k$ now follows from  \eqref{contderustar}, \eqref{exprustar} and the chain rule, since the explicit expression of $\Phi^{-1}$ shows that $|\nabla^l \Phi^{-1}(x)| \le \frac{C_l}{(1+|x|)^{l+1}}$ for every $x \in \R^n_+$. 
\end{proof}

  \subsection{Estimates on solutions of linearised equations}

   In our analysis we will also investigate the linearisations of \eqref{eq:critRn} and \eqref{eq:crithf} at a non-zero solution. If $v \in D^{k,2}(\R^n)$ is a solution of \eqref{eq:critRn} we let 
 \begin{equation} \label{ker:interior}
  \Ker_{v}= \{h\in \hSob(\R^n)\,:\,\Dd^k h = (\crit-1)|v|^{\crit-2}h \text{ in } \R^n\}.
      \end{equation} 
  Similarly, if $v \in D^{k,2}(\R^n_+)$ is a solution of \eqref{eq:crithf} we let  
  \begin{equation} \label{ker:boundary}
\Ker_{v}= \{h\in \hSob(\R^n_+)\,:\,\Dd^k h = (\crit-1)|v|^{\crit-2}h \text{ in } \R^n_+\}.
    \end{equation} 
We prove the following optimal decay for elements in $\Ker_v$: 
  \begin{lemma}\label{prop:ctlsollin}
    Let  $v$ be a solution of \eqref{eq:critRn} (resp. \eqref{eq:crithf}) and let $Z \in \Ker_{v}$, where $\Ker_v$ is as in \eqref{ker:interior} (resp. \eqref{ker:boundary}). There exists $C>0$ such that for all $x\in \R^n$ (resp. $x \in \overline{\R}^n_+$) and for any $0 \le l \le 2k-1$,
  \begin{equation*} 
    |\nabla^l Z(x)| \leq C (1+ |x|)^{2k-n-l}.
  \end{equation*}
  \end{lemma} 
  
  \begin{proof}
  The proof follows the same lines and uses the same notations than the proof of Proposition~\ref{prop:ctlsol} and we only sketch the main steps. 
  
Consider first consider the entire space case. Let  $Z \in \Ker_v$ and let, for $y \in \mathbb{S}^n$, $\tilde{Z}(y) = \frac{Z}{U}(\pi(y))$. Then $\tilde Z \in H^k(\mathbb{S}^n)$ and the conformal invariance of $P_{g_0}$ shows that, for any $y \in \mathbb{S}^n$, $\tilde{Z}$ satisfies
\begin{equation*} 
 P_{g_0} \tilde{Z}(y) = a(y) \tilde{Z}(y) \quad \text{ in } \mathbb{S}^n, \quad \text{ where } \quad a(y) = (2^\sharp-1) \Big( \frac{ v(\pi(y))}{U(\pi(y))} \Big)^{2^\sharp -2} .
  \end{equation*}
 By Proposition~\ref{prop:ctlsol}, $a \in L^\infty(\mathbb{S}^n)$, so that by standard elliptic theory $\tilde{Z} \in C^{2k-1}(\mathbb{S}^n)$. We conclude as in the proof of Proposition~\ref{prop:ctlsol}.

We now consider the half-space case. Let $v\in \hSob(\R^n_+)$ be a weak solution of \eqref{eq:crithf}, let $v^*$ be its Cayley transform and let $Z \in \Ker_v$. As before, $Z^* \in H^k_0(B)$, and a change of variable using the weak formulation of \eqref{eq:crithf} and the properties of $\Phi$ shows, by using \eqref{eq:laplaciens}, that $Z^*$ satisfies weakly 
  \[  \begin{cases}
  (-\D)^k Z^* = (2^\sharp-1)|u^*|^{\crit-2}Z^* & \text{in }B\\
  |\nabla^l Z^*|=0 & \text{for } l=0,\ldots, k-1, \quad \text{on } \partial B
\end{cases}.
\] 
As proven in \eqref{contderustar} we have $u^* \in L^\infty(B)$ and hence standard elliptic theory shows that $Z^* \in C^{2k-1}(\overline{B})$. We again conclude as in the proof of Proposition~\ref{prop:ctlsol}.
  \end{proof}

\section{Bubble-tree configurations} \label{sec:bulles2}

In this section we define the bubbling profiles that we will use in this work and we introduce the notion of bubble-tree.

  \subsection{Bubbling profiles and weak $H^k_0(\O)$ compactness}

\par Until the end of this paper $\O$ will denote a smooth bounded domain of $\R^n, n \ge 3$. Smoothness ensures the existence of $r_0 >0$ and, for any $x \in \dO$, of an open set $U_x$ containing $x$ and of a diffeomorphism $\sigma_x : B(0,r_0) \to U_x \subseteq \R^n$ such that $\sigma_x(0) = x$, $(x,z) \mapsto \sigma_x(z)$ is smooth from $\dO \times B(0,r_0)$ to $\oO$, and 
\begin{align*}
  &\sigma_x\big(B(0,r_0) \cap \R^n_+\big) = U_x \cap \O, \quad  &\sigma_x\big(B(0,r_0) \cap \partial\R^n_+) = U_x \cap \dO, \\
  &d\sigma_x(0) : \R^n \to \R^n \text{ is an isometry,  }  \quad & \text{ and } d\sigma_x(0)[-e_1] = \nu(x)
\end{align*}
where $\nu(x)$ is the outer unit normal vector to $\dO$ at $x$. Throughout this paper we will let $\chi \in C^\infty_c(\R^n)$ be a cut-off function, such that $\chi \equiv 1$ in $B(0,\frac12)$ and $\chi\equiv 0$ in $\R^n \setminus B(0,1)$. Following \cite[Definition 2.1]{Maz17} we define

\begin{definition}[Interior bubble]\label{def:bub:int}
We call \emph{interior bubble} a triple  $[(x_\a)_\a, (\mu_\a)_\a, v]$, where $(\mu_\a)_\a$ is a sequence of positive numbers converging to $0$ as $\alpha \to + \infty$,  $(x_\a)_\a$ is a convergent sequence of points in $\O$ such that $\lim_{\a\to \infty} \frac{\dist{x_\a,\dO}}{\mu_\a} = +\infty$, and $v\in \hSob(\R^n)$ is a non-zero solution of \eqref{eq:critRn}. If $[(x_\a)_\a, (\mu_\a)_\a, v]$ is an interior bubble we define, for all $x \in \oO$,
    \[  V_\a(x) = \chi\Bigg(\frac{x-x_\a}{\dist{x_\a,\dO}}\Bigg) \mu_\a^{-\exc} v\Big(\frac{x-x_\a}{\mu_\a}\Big).  \]
 We will sometimes refer to $V_\alpha$ as the bubble. 
\end{definition} 
 
 \begin{definition}[boundary bubble]\label{def:bub:bdr}
We call \emph{boundary bubble} a triple  $[(x_\a)_\a, (\mu_\a)_\a, v]$, where $(\mu_\a)_\a$ is a sequence of positive numbers converging to $0$ as $\alpha \to + \infty$, $(x_\a)_\a$ is a convergent sequence of points in $\partial \O$ and $v\in \hSob(\R^n_+)$ is a non-zero solution of \eqref{eq:crithf}. If $[(x_\a)_\a, (\mu_\a)_\a, v]$ is a boundary bubble we define, for all $x \in \oO$,
    \[  V_\a(x) = \chi\big(\sigma_{x_\a}^{-1}(x)\big) \mu_\a^{-\exc} v\Big(\frac{\sigma_{x_\a}^{-1}(x)}{\mu_\a}\Big).    \]
 We will also sometimes refer to $V_\alpha$ as the bubble.   
        \end{definition}

\noindent We could equivalently define a boundary bubble as a triple $[(x_\a)_\a, (\mu_\a)_\a, v]$, where $\dist{x_\alpha, \partial \O} = O(\mu_\a)$ and $v\in \hSob(\R^n_+)$ solves  \eqref{eq:crithf}. Up to replacing $x_\alpha$ by its projection onto $\partial \O$ (which is uniquely defined since $\mu_\a \to 0$), however, we may assume with no loss of generality that $x_\a \in \partial \O$ for all $\a$, which justifies Definition \ref{def:bub:bdr}. Remark that for an interior bubble as in Definition \ref{def:bub:int} it is perfectly possible that $\lim_{\alpha \to + \infty} x_\alpha \in \partial \O$. Direct computations show that in either case we have $V_\a \in \Sob_0(\O)$, $V_\a \wto 0$ in  $\Sob_0(\O)$ and, as $\alpha \to + \infty$,
\[  \Vert V_\a \Vert_{H^k_0(\Omega)} = \begin{cases}
  \Hnorm{k,2}{v} + \smallo(1) & \text{for an interior bubble,}\\
  \Hnorm[\R^n_+]{k,2}{v} + \smallo(1) & \text{for a boundary bubble. }\\
\end{cases}  
\]  
\noindent Throughout this paper we will investigate sequences of solutions of perturbations of \eqref{eq:lmmain}, where the operator $L$ given by \eqref{def:L} is replaced by a sequence of differential operators $(L_\a)_\a$. The perturbations of $L$ that we will consider are as follows: if $0 \le p \le k-1$ is as in \eqref{def:L} we will let
\begin{equation}\label{def:La}  
  L_\a = \Dd^k + \sum_{l=0}^{p} \nabla^l (-1)^l\big(A_{l,\a}(\nabla^l \,\cdot\,)\big),
\end{equation}
where, for $l=0,\ldots, p$, $A_{l,\a}$ is a smooth symmetric $(2l,0)$-tensor such that $A_{l,\a} \to A_l$ in $C^{l}(\oO)$ as $\a \to \infty$. We will investigate the blow-up behavior of energy-bounded sequences of solutions $(u_\a)_\a$ that satisfy the following critical polyharmonic equations with Dirichlet boundary conditions:
\begin{equation}\label{eq:main}
  \begin{bigcases}
    &L_\a u_\a= |u_\a|^{\crit-2} u_\a & &\text{in }\O\\
    &|\nabla^{l} u_\a| = 0 & &\text{for } l=0,\ldots, k-1, \quad \text{on } \dO.
  \end{bigcases}
\end{equation}
Let $(u_\a)_{\a}$ be a sequence of solutions of \eqref{eq:main}. We assume that $(u_\a)_\a$ is bounded in $H^k_0(\O)$, that is $\limsup_{\a \to + \infty} \Vert u_\a \Vert_{H^k_0(\O)} < + \infty$, and that it blows-up as $\alpha \to + \infty$, that is $\Vert u_\a \Vert_{L^\infty(\O)} = + \infty$. By  \eqref{def:La} such a sequence is a Palais-Smale sequence for the functional over $H^k_0(\O)$ defined by
\begin{equation}\label{def:funcI}  
  I(u) = \frac{1}{2}\intM{\O}{|\Dd^{k/2}u|^2}+ \frac{1}{2} \sum_{l=0}^{p}\intM{\O}{A_l(\nabla^l u,\nabla^l u)} - \frac{1}{\crit}\intM{\O}{|u|^{\crit}}
\end{equation}
for which \eqref{eq:lmmain} is the associated Euler-Lagrange equation. We may thus apply the global compactness result of \cite[Theorem 1.2]{Maz17}, that generalizes the celebrated result of \cite{Struwe} for $k=1$: it shows that there exist $u_0 \in \Sob_0(\O)$ a solution to \eqref{eq:lmmain}, an integer $N\geq 1$ and $N$ bubbles $\Va$ defined as in Definitions \ref{def:bub:int} and \ref{def:bub:bdr} by their triples $[(\xa)_\a,(\ma)_\a,v^i]$, $i=1,\ldots, N$ such that up to a subsequence we have 
  \begin{equation}\label{eq:decHk} 
   u_\a = u_0 + \sumi \Va + \smallo(1) \qquad \text{in } \Sob_0(\O)
  \end{equation} 
  and 
  \begin{equation}\label{eq:decHk:bis} 
  I(u_\a) = I(u_0) + \sumi I_0(v^i) + \smallo(1),
  \end{equation} 
  where we have let 
  \begin{equation} \label{def:I0}
   I_0(v^i) = \left \{ \begin{aligned}
  & \frac{1}{2}\intM{\R^n}{|\Dd^{k/2}v^i|^2} - \frac{1}{\crit}\intM{\R^n}{|v^i |^{\crit}} \\ 
    & \frac{1}{2}\intM{\R^n_+}{|\Dd^{k/2}v^i|^2} - \frac{1}{\crit}\intM{\R^n_+}{|v^i |^{\crit}} 
  \end{aligned} \right. 
  \end{equation}
depending on whether $v^i$ is an interior or a boundary bubble. Moreover, the sequences $(\xa)_\a,(\ma)_\a$ satisfy the fundamental structure relation
  \begin{equation}\label{eq:struc}
    \eps_\a^{ij} = \frac{|\xa-\xa[j]|^2}{\ma\ma[j]} + \frac{\ma}{\ma[j]} + \frac{\ma[j]}{\ma}\to \infty \quad \text{as } \a \to \infty,
  \end{equation} 
  for all $i\neq j$ in $\{1,\ldots,N\}$. Relation \eqref{eq:struc} implies that the bubbles $[(\xa)_\a,(\ma)_\a,v^i]$ do not interact at first-order in $H^k_0(\O)$. Relation \eqref{eq:struc} has been known for positive solutions of equations like \eqref{eq:main} for a long time, where it appeared e.g. \cite{BahriCoron} (see also \cite[Chapter 3]{Heb14} for a proof). For \emph{sign-changing} solutions of equations like \eqref{eq:main}, relation \eqref{eq:struc} has also been known to hold true for a long time (it is e.g. mentioned in \cite[Chapter III, Remark 3.2]{StruweVariationalMethods} in the case $k=1$) and is a folklore-type result but, to the best of our knowledge, no proof seems to have appeared in the literature, not even in the case $k=1$. For the sake of completeness we provide a simple proof of \eqref{eq:struc} in Appendix \ref{structurerelation} below. 

\subsection{Bubble-tree configurations}

To prove Theorem~\ref{theo:compactness} we will show that, for sequences $(u_\a)_\a$ of energy-bounded solutions of \eqref{eq:main}, the decomposition \eqref{eq:decHk} can be improved into strong global pointwise bounds in $\oO$. In order to do this we will analyse the pointwise behavior of any configuration that may arise in \eqref{eq:decHk}. Let $u_0\in C^{2k}(\oO)\cap\Sob_0(\O)$ be a solution to \eqref{eq:lmmain}, $N\geq 1$ be an integer and let $N$ bubbles $\Va$, defined as in Definitions \ref{def:bub:int} and \ref{def:bub:bdr}. We allow indistinctly interior and boundary bubbles, but we assume that \eqref{eq:struc} is satisfied. We let 
    \[ \Ker_0 = \{h\in \Sob_0(\O)\,:\, Lh = (\crit-1)|u_0|^{\crit-2}h \text{ in } \O\},
    \] 
  where $L$ is as in \eqref{def:L}. Since $L$ is elliptic $\Ker_0$ is finite-dimensional: we let  $d_0 = \dim\Ker_0$ and we let $Z^{01},\ldots, Z^{0d_0}$ be an orthonormal basis of $\Ker_0$ in $\Sob_0(\O)$ with respect to the $H^k_0(\O)$ scalar product \eqref{norm:Dk2}. For any $i\in\{1,\ldots,N\}$ we represent the bubble $\Va$ by its triple $[(\xa)_\a,(\ma)_\a,v^i]$. 
  \begin{itemize}
  \item If $\Va$ is an interior bubble we let $\Ker_{v^i}$ be given by \eqref{ker:interior},  
    $d_i = \dim\Ker_{v^i}$, we let $Z^{i1},\ldots,Z^{id_i}$ be an orthonormal basis of $\Ker_{v^i}$ in $\hSob(\R^n)$ and we define, for $j=1,\ldots d_i$ and for $x\in \oO$,
    \[  \Za(x) = \chi\Big(\frac{x-\xa}{ \dist{\xa, \dO}}\Big)(\ma)^{-\exc}Z^{ij}\Big(\frac{x-\xa}{\ma}\Big),
    \]
    where $\chi$ is as in definition \ref{def:bub:int}. 
      
\item If $\Va$ is a boundary bubble we let $\Ker_{v^i}$ be given by \eqref{ker:boundary},
    $d_i = \dim\Ker_{v^i}$, we let $Z^{i1},\ldots,Z^{id_i}$ be an orthonormal basis of $\Ker_{v^i}$ in $\hSob(\R^n_+)$ and we define, for $j=1,\ldots d_i$ and for $x\in \oO$,
    \[  \Za(x) = \chi\big(\sigma_{\xa}^{-1}(x)\big)(\ma)^{-\exc}Z^{ij}\Big(\frac{\sigma_{\xa}^{-1}(x)}{\ma}\Big),
    \]
    where $\chi$ and $\sigma_{\xa}$ are as in definition \ref{def:bub:bdr}.
      \end{itemize}
Simple computations show that $\Za[i j]\in H^k_0(\O)$ and 
\begin{equation}\label{eq:ortZa}
  \intM{\O}{(-\D)^k \Za[i_1j_1]\Za[i_2j_2]} = \delta_{i_1 i_2}\delta_{j_1 j_2} + \smallo(1)
\end{equation}
as $\a \to \infty$, for $i_1,i_2\in \{0,\ldots, N\}$, $j_1 \in \{1,\ldots, d_{i_1}\},\,j_2\in \{1,\ldots, d_{i_2}\}$. For $1 \le i \le N$ we define $\Ker^i_\a = \spanned{\Za,\, j=1,\ldots,d_i} \sub \Sob_0(\O)$. For simplicity we will write $\Za[0j]=Z^{0j}$ even if $Z^{0j}$ does not depend on $\a$. We define 
\begin{equation}\label{def:Ka}
  \Ker_\a = \Ker_0 \oplus \bigoplus_{i=1}^{N} \Ker^i_\a.
\end{equation}
This space will act as an approximate kernel for the linearised equation of \eqref{eq:main} at a bubbling solution. We may now introduce the definition of a bubble-tree configuration that we will use in this paper: 

\begin{definition}[Bubble-tree configuration] \label{def:bubbletree}
Let $u_0\in C^{2k}(\oO)\cap\Sob_0(\O)$ be a solution of \eqref{eq:lmmain}, let $N\geq 1$ be an integer and let $[(\xa)_\a,(\ma)_\a,v^i]$, $i=1,\ldots, N$ be $N$ bubbles as in Definitions \ref{def:bub:int} and \ref{def:bub:bdr}, satisfying \eqref{eq:struc}. For $0 \le i \le N$ and $1 \le j \le d_i$, where $d_i = \dim \Ker^i_\a$, let $(\nu^{ij}_\a)_\a$ be any sequence of real numbers such that $\nu^{ij}_\a \to 0$. For $\alpha \ge 1$ we define 
\begin{equation}\label{def:Wa}
  \Wa = u_0 + \sumi \Va + \sumij \nu^{ij}_\a \Za.
\end{equation}
We say that $\Wa$ is a \emph{bubble-tree} defined by $u_0$ and $[(\xa)_\a,(\ma)_\a,v^i]$, $i=1,\ldots, N$, and $\nu_{\alpha}^{ij}$, $0 \le i \le N$, $1 \le j \le d_i$. 
\end{definition}
Clearly, $\Wa \in C^{2k}(\oO) \cap \Sob_0(\O)$ and $\limsup_{\alpha \to + \infty} \Vert \Wa \Vert_{H^k_0(\O)} < + \infty$. Up to a small perturbation in $H^k_0(\O)$ along $\Ker_\a$, $\Wa$ is modeled on the first-order terms that arise in the Struwe decomposition \eqref{eq:decHk}. We will often compare each bubble $\Va$ with its associated positive bubble, that we define as follows. We let, for $x \in \oO$,
\begin{equation} \label{def:Ba}
 \begin{aligned}
 \Ba[0](x)  & = 1 \quad \text{ and, for every  } 1 \le i \le N,  \\
  \Ba(x)  & =  \left(\frac{\ma}{(\ma)^2 + a_{n,k} |x-\xa|^2}\right)^{\exc} \quad \text{ and } \\ 
   \ta(x)  & = \ma + |x-\xa|
\end{aligned} 
\end{equation}
  where $a_{n,k}$ is as in \eqref{def:eucBub}. We define $ \Ba[0](x) \equiv 1$ in \eqref{def:Ba} for simplicity but also to highlight that we will consider, in our analysis, the weak limit $u_0$ as the zero-th bubbling profile. Note that $u_0$ may be everywhere zero in Definition \ref{def:bubbletree}, but we still define $\Ba[0] \equiv 1$ in that case nonetheless. We introduce the following notation: until the end of the paper, the inequality 
  $$F_\a \lesssim G_\a,$$
  that we will use for real numbers or functions, will mean that there exists a positive constant $C$ independent of $\a$ such that $F_\a \le C G_\a$ for all $\a \ge 1$.  It is easily seen with \eqref{def:Ba} that 
  \[ (\ma)^{\exc} \ta(x)^{2k-n} \lesssim \Ba(x) \lesssim (\ma)^{\exc} \ta(x)^{2k-n} \quad \text{ for all } x \in \overline{\O}.
\] 
By the definition of $\Va$, and by \eqref{eq:Idcrit}, we get that for $0 \le l \le 2k$, for all $x\in \oO$, 
\ben \label{est:bubble}  |\nabla^l \Va(x)| \lesssim \ta(x)^{-l}\Ba(x).\een
Concerning the kernel elements, it follows from elliptic theory that for $j=1,\ldots, d_0$, $Z^{0j} \in C^{2k}(\oO)$. Similarly, it follows from Lemma~\ref{prop:ctlsollin} and from direct computations using the expression of $\Za$ that for $i=1,\ldots,N$, $j=1,\ldots, d_i$, and for $0 \le l \le 2k$ and $x\in \oO$, we have
\begin{equation}\label{eq:IZa}
  |\nabla^l \Za(x)|  \lesssim \ta(x)^{-l}\Ba(x). \\
\end{equation}
Combining the previous estimates gives, for $0 \le l \le 2k-1$ and all $x \in \oO$, 
\begin{equation}\label{eq:IWa}
  |\nabla^l \Wa(x)|\lesssim 1 + \sumi \ta(x)^{-l} \Ba(x) .
\end{equation}
The constant in \eqref{eq:IWa} depends on $n,k,u_0, (v^i)_{1 \le i \le N}$.

\subsection{Radii of influence and bubble-tree computations}\label{sec:bubtre} 
Throughout this subsection we let $\Wa$ be a bubble-tree configuration as given by Definition \ref{def:bubbletree}. The bubbles $\Va$, $1 \le i \le N$, that appear in $\Wa$ may concentrate at different speeds and may take very different pointwise values in $\oO$. In this subsection we decompose $\oO$ in regions where each bubble $\Va$ is dominant at a pointwise level. This information will be crucial to prove sharp global pointwise estimates in Theorem~\ref{prop:main} below by an inductive approach. We rely for this on the notion of \emph{radius of influence}, which was originally introduced in \cite{DruetJDG} for positive bubbles (see also \cite[Chapter 8.2]{Heb14}). As observed in \cite{Pre24}, the notion of radius of influence remains just as useful to obtain pointwise estimates for sign-changing bubbles.

Up to passing to a subsequence for $(\xa)_\a,(\ma)_\a$, and up to renumbering the bubbles in $\Wa$, we can assume that 
\begin{equation}\label{eq:renum}
    \ma[N] \leq \ldots\leq \ma[1]< 1
\end{equation}
for all $\a \ge 1$. Let $i \in \{1,\ldots N\}$ be fixed. We define 
\begin{equation} \label{def:Ai}
  \sA_i = \{j\neq i\,:\, \ma = \bigO(\ma[j])\}  \quad \text{ and } \quad 
  \sA_i^c = \{j \neq i\,:\, \ma[j] = \smallo(\ma)\}. 
\end{equation}
Indices in $\sA_i $ refer to bubbles $V_\a^j$, $j \neq i$, which concentrate at a speed slower than or comparable to $\Va$, while indices in $\sA_i $ refer to bubbles $V_\a^j$, $j \neq i$, which concentrate strictly faster than $\Va$. We will often refer to these bubble, respectively, as to \emph{lower} and \emph{higher} bubbles than $\Va$. Note that $\O$ is a bounded open set here, so the speed of concentration for bubbles is solely measured in terms of the size of $\mu_\alpha^i$, and not by $\eps_\a^{ij}$ in \eqref{eq:struc} (as was e.g. the case in \cite{DengSunWei} in $\R^n$). Remark that by \eqref{eq:renum}, we have $\{1,\ldots, i-1\} \subeq \sA_i$ for $i>1$, and $\sA_i^c\subeq \{i+1, \ldots,N\}$ for $i<N$.

\smallskip

Let $j\in \sA_i$, so that $\ma = \bigO(\ma[j])$ and $\Va[j]$ is lower than $\Va$. We define $\sa$ as follows: 
  \begin{itemize}
      \item If $\ma = o(\ma[j])$, we define
    \[  \big(\sa\big)^2 = \frac{1}{a_{n,k}}\frac{\ma}{\ma[j]}\Big((\ma[j])^2 + a_{n,k} |\xa-\xa[j] |^2 \Big)
    \]
    \item If $\ma[j] = \bigO(\ma)$, that is if we have in addition $i\in \sA_j$, we also have $\frac{|\xa-\xa[j]|}{\ma}\to \infty$ by \eqref{eq:struc}. In this case we let $m_{ij}= \limsup\limits_{\a\to \infty} \Big(\frac{\ma}{\ma[j]}+ \frac{\ma[j]}{\ma}\Big) \in (0, + \infty)$, and we define
    \[  \big(\sa \big)^2 =   \frac{1}{4 a_{n,k}} \frac{1}{ m_{ij}}\frac{\ma}{\ma[j]}\Big((\ma[j])^2 + a_{n,k} |\xa-\xa[j] |^2 \Big),
    \]
  \end{itemize}
   where $a_{n,k}$ is as in \eqref{def:eucBub}. It follows from the definition of $\Ba$ in \eqref{def:Ba} that $ \Ba(x) \ge \Ba[j](\xa)$ for every $x \in B(\xa, \sa)$, whenever $j \in \mathcal{A}_i$, and this is the reason why we include $\frac{1}{a_{n,k}}$ in our definition of $\sa$. Moreover, a simple computation shows that if $j\in \sA_i$, for any $0 \le l \le 2k$ and $x \in B(\xa, \sa)$ we have 
\begin{equation}\label{eq:ctlBj}
  \ta[j](x)^{-l} \Ba[j](x) \lesssim \frac{(\ma)^{\frac{n-2k}{2}}}{(\sa)^{n-2k+l}} \lesssim  \ta(x)^{-l}\Ba(x). 
\end{equation}
Specialising \eqref{eq:ctlBj} to $l = 0$ shows in particular that $\sa$ measures the radius around $\xa$ up to which $\Ba$ is pointwise larger than $\Ba[j]$. In the special case of \emph{comparable} bubbles where $j\in \sA_i$ and $i\in\sA_j$, ie when $\ma \lesssim \ma[j] \lesssim \ma$ for all $\alpha \ge 1$, the definition of $\sa$ shows that $\sa + \sa[ji] \leq \frac{|\xa-\xa[j]|}{2}$ for $\a$ big enough, so that 
\[  B(\xa, \sa) \cap B(\xa[j],\sa[ji]) = \emptyset.
\]
This is the motivation for the additional normalisation factor $\frac{1}{4m_{ij}}$ when $i \in \mathcal{A}_j$.
\begin{definition}
 Let $1 \le i \le N$.  We define the radius of influence of $\Va$ as
  \begin{equation} \label{def:ra}
    \ra = \min\left\{\min_{j\in \sA_i} \sa, \sqrt{\ma}\right\}.
  \end{equation}
\end{definition}
\noindent In view of \eqref{eq:ctlBj} $\ra$ should be understood as the largest radius of a ball centered at $\xa$ inside which the positive bubble $\Ba$ is \emph{pointwise} larger than all the other lower positive bubbles. This also includes the non-concentrated zero-th bubble $\Ba[0]\equiv 1$ since, as can be easily checked, $\Ba(x) \ge (1+ \smallo(1)) a_{n,k}^{k-\frac{n}{2}}$ when $|x-\xa| \le \sqrt{\ma}$. It follows from \eqref{eq:struc} and the definition of $\sa$ that $\frac{\sa}{\ma} \gtrsim \ve_\a^{ij}$ as $\a \to + \infty$ for every $j \in \mathcal{A}_i$. As a consequence, and for any $1 \le i \le N$,
\begin{equation} \label{rasurma}
 \frac{\ra}{\ma} \gtrsim \min\Bigg( \min_{1 \le i \le N} (\ma)^{-\frac12}, \min_{\substack{i=1,\ldots N\\ j\neq i \,:\, \ma \lesssim \ma[j]}} \eps_\a^{ij}\Bigg) \to + \infty 
\end{equation}
as $\a \to + \infty$. We showed that in $B(\xa[i], \ra)$, $\Ba$ pointwise dominates the positive lower bubbles. But inside $B(\xa, \ra)$ there can still be higher bubbles interacting with $\Ba[i]$. They are described by the following set of indices: 
\begin{equation} \label{def:Bi}
  \sB_i = \{j\neq i\,:\, \ma[j]=o(\ma) \text{ and } |\xa -\xa[j]| \leq 2 \ra\}.
\end{equation}
The following result shows that, as long as pointwise estimates on derivatives of the bubble-tree $\Wa$ are considered in $B(\xa, \ra)$, one only has to take into account the $i$-th bubble and higher bubbles belonging to $\sB_i$:
\begin{lemma}
Let $1 \le i \le N$ and $0 \le l \le 2k-1$ be fixed. For any  $x \in B(\xa, \ra)$ we have 
\begin{equation}\label{meaning:ra:0}
\begin{aligned}
& 1+ \sum_{j=1}^N  \ta[j](x)^{-l}\Ba[j](x)\lesssim  \ta(x)^{-l} \Ba(x)  + \sum_{j \in \mathcal{B}_i} \ta[j](x)^{-l}\Ba[j](x) . 
\end{aligned}
 \end{equation}
\end{lemma}
\begin{proof}
Let $j \neq i$. If $j \in \mathcal{A}_i$ and since $\ra \le \sa$ by definition we have $ \ta[j](x)^{-l} \Ba[j](x) \lesssim \ta(x)^{-l}\Ba(x)  $ for any $x \in B(\xa \ra)$ by \eqref{eq:ctlBj}.  If $j \not \in \mathcal{A}_i$ then $\ma[j] = \smallo(\ma)$ as $\alpha \to + \infty$ and two cases may occur: either $j \in \sB_i$ as in \eqref{def:Bi} or $j\in \sA_i^c\setminus \sB_i$. In the latter case we have $\ma[j] = o(\ma)$ and $|\xa-\xa[j]|>2\ra$. Therefore, we have $|x-\xa[j]| \geq \ra$ for all $x\in B(\xa,\ra)$, so that 
  \[  \ta[j](x)^{-l}\Ba[j](x) \leq C \left(\frac{\ma[j]}{\ma}\right)^{\exc}\frac{(\ma)^{\exc}}{(\ra)^{n-2k+l}} = \smallo\big(\ta(x)^{-l}\Ba(x)\big)
\]
as $\a \to + \infty$, where we again used \eqref{eq:ctlBj}. Finally, since $\ra \lesssim \sqrt{\ma}$ we have $ 1 \lesssim  \ta(x)^{-l}\Ba(x)$ for all $x\in B(\xa,\ra)$. This proves \eqref{meaning:ra:0}.
\end{proof}
A consequence of \eqref{meaning:ra:0} and of the definition of $\Wa$ in Definition \ref{def:bubbletree} is that, for any $0 \le l \le 2k-1$,  we have
\begin{equation}\label{meaning:ra:1}
 |\nabla^{l} \Wa(x)| \lesssim   \ta(x)^{-l} \Ba(x) + \sum_{j \in \mathcal{B}_i} \ta[j](x)^{-l}\Ba[j](x)  \quad \text{ for } x \in B(\xa, \ra).
 \end{equation}
Let now $j \in \mathcal{B}_i$ be fixed. It follows easily from the definition of $s_\a^{ji}$ and from the property $\ma[j] = \smallo(\ma)$ that there exists $C >0$ such that if $\Ba[j](x) \ge \frac{1}{C} \Ba(x)$ then $|x-\xa[j]| \le C s_\a^{ji}$. As a consequence, $\Ba[j](x) \le C \Ba(x)$ when $|x-\xa[j]| \ge C s_\a^{ji}$. Since $\ma[j] = \smallo(\ma)$ we have $s_\a^{ji} = \smallo (|\xa - \xa[j]|)$ by definition of $s_\a^{ji}$, and thus direct computations give
$$ \Ba[j](x)  \le C \Ba(\xa[j]) = (1+ \smallo(1))C \Ba(x)  \quad \text{ for } s_\a^{ji} \le |x-\xa[j]| \le C s_\a^{ji}  .$$ 
Specialising \eqref{meaning:ra:1} to $l=0$ with the latter shows in particular that 
\begin{equation}\label{meaning:ra:2}
 |\Wa(x)| \lesssim \Ba(x)  \quad \text{ for } x \in B(\xa, \ra) \backslash \bigcup_{j \in \mathcal{B}_i} B(\xa[j], s_\a^{ji}).
 \end{equation}
 Estimate \eqref{meaning:ra:2} shows that in the region $B(\xa, \ra) \backslash \bigcup_{j \in \mathcal{B}_i} B(\xa[j], s_\a^{ji})$ the $i$-th positive bubble is \emph{pointwise} dominant in the bubble-tree $\Wa$. Partitioning $\O$ in regions where each positive bubble $\Ba, 1 \le i \le N$, is dominant at the pointwise level will provide us with a top-bottom inductive scheme to prove the pointwise estimates of Theorem~\ref{prop:main} below. As it turns out, though,  \eqref{meaning:ra:2} is not enough for our purposes, since we will need precise estimates on all the derivatives of $\Wa$ up to order $2k-1$. In the following we modify the definition of $\ra$ in order to control the sum over $\mathcal{B}_i$ in \eqref{meaning:ra:1} solely by the $i$-th bubble term: 
\begin{definition}
  Let $1 \le i \le N$ and let $j\in \sB_i$. We define 
  \begin{equation}\label{def:rha}
    \rha = 2\left(\frac{\ma[j]}{\ma}\right)^{\frac{n-2k}{2(n-1)}}\big(|\xa-\xa[i]|+\ma\big).
  \end{equation}
\end{definition}
It is simple to check that for any $0 \le l \le 2k-1$ we have 
\begin{equation} \label{forgotten:label}
(\ta[j](x))^{-l}\Ba[j](x) \gtrsim (\ta(x))^{-l}\Ba(x) \quad \text{  for every } \quad x \in B(\xa, \rha),
\end{equation}
that there exists $C >0$ such that 
$$ \text{ if }  \quad (\ta[j](x))^{-l}\Ba[j](x) \ge \frac{1}{C}(\ta(x))^{-l}\Ba(x) \quad \text{ then }  \quad |x-\xa[j]| \le C \rha $$
and that 
\begin{equation} \label{def:rha:plus}
(\ta[j](x))^{-l}\Ba[j](x) \lesssim (\ta(x))^{-l}\Ba(x)\quad \text{ when } |x-\xa[j]| \ge \rha.
\end{equation} 
By analogy with $\sa$ above, $\rha$ is thus defined as the maximum distance from $\xa[j]$ up to which $(\ta[j])^{-l}\Ba[j] $ is pointwise larger than  $(\ta)^{-l}\Ba$. Since $|\nabla^l \Ba[j]|$ is of the order of  $(\ta[j])^{-l}\Ba[j]$ this should be understood as the maximum distance from $\xa[j]$ up to which $|\nabla^l \Ba[j]|$ dominates $|\nabla^l \Ba|$ pointwise.  

By definition of $\mathcal{B}_i$, for every $j \in \mathcal{B}_i$ we have $|\xa -\xa[j]| \leq 2 \ra$ and hence, since $\ma[j] = \smallo(\ma)$, $\rha = o(\ra)$. Since $\frac{n-2k}{2(n-1)}<\frac12$ we also easily have $\sa[ji]=o(\rha)$. Moreover, when $j\in \sB_i$ is such that $\frac{|\xa[j]-\xa|}{\ma}=\bigO(1)$, we get that $\rha = o(\ma)$. For a fixed $1 \le i \le N$ we now define 
  \begin{equation}\label{def:Oa}
    \Oa[i] = \O \cap \Big(B(\xa,\ra)\setminus \bigcup_{j\in \sB_i} B(\xa[j],\rha)\Big).
  \end{equation}
We have the following Lemma: 
\begin{lemma}\label{prop:ctltre}
  Let $i\in \{1,\ldots, N\}$. For any $0\le l \le 2k-1$ and for any $x\in \overline{ \Oa[i]}$ we have 
 $$  1+ \sum_{i=1}^N  \ta[j](x)^{-l}\Ba[j](x) \lesssim  \ta(x)^{-l}\Ba(x). $$
\end{lemma}
A consequence of Lemma \ref{prop:ctltre} and of \eqref{meaning:ra:1} is that, for any $0 \le l \le 2k-1$ and any $x\in \overline{ \Oa[i]}$, we have 
$$  |\nabla^{l} \Wa(x)| \lesssim   \ta(x)^{-l} \Ba(x) .$$ 
Lemma \ref{prop:ctltre} shows that $\Oa[i]$ can be understood as the region in $\O$ where $\Va$, and all its derivatives up to order $2k-1$, are pointwise dominant in the bubble-tree $\Wa$. We will sometimes refer to it as the \emph{region of influence of the $i$-th bubble $\Va$ in $\O$.}
\begin{proof}
  The proof follows from the definitions of $\ra$ and $\Oa[i]$. If $j\in \sB_i$, \eqref{def:rha:plus} shows that 
  \[  \ta[j](x)^{-l}\Ba[j](x) \lesssim \ta(x)^{-l}\Ba(x)
  \] 
  for all $x\in \oO \setminus B(\xa[j],\rha)$. The result then follows from \eqref{meaning:ra:0}.
\end{proof}

\begin{remark}
It is possible to give a subtler definition of $\ra$ in \eqref{def:ra} to take into account e.g. a possible zero weak limit $u_0$. This has been done in \cite{Pre24} for the case $k=1$, but we will not need the same precision than in \cite{Pre24} here. 
\end{remark}

We conclude this subsection by illustrating how the notion of radius of influence allows us to precisely estimate the pointwise interactions between the different bubbling profiles:

\begin{lemma} \label{lemme:interactions:arbre:bulles}
Let $1 \le i \le N$ be fixed, let $\Ba$ be as in \eqref{def:Ba} and let $\Oa[i]$ be as in \eqref{def:Oa}. We have:
$$ \begin{aligned} 
& (\ma)^{\frac{n+2k}{2}} \Bigg| \Bigg|   \sum_{\substack{r,s = 0,\ldots, N\\ r\neq s}} (\Ba[r])^{\crit-2}\Ba[s] \Bigg| \Bigg|_{L^\infty(\overline{\Oa[i]})} \\
& \lesssim \Big(  \max_{\substack{i,j=1,\ldots N\\ j\neq i \,:\, \ma \lesssim \ma[j]}} (\eps_\a^{ij})^{-\frac{1}{2}} \Big)^{\min(n-2k,4k)} + \Big(  \max_{\substack{i=1,\ldots N\\ j\neq i \,:\, \ma[j] = o(\ma)}} \Big(\frac{\ma[j]}{\ma[i]}\Big)^{\frac{2k-1}{2(n-1)}} \Big)^{\min(n-2k,4k)},
\end{aligned} 
 $$
as $\alpha \to + \infty$.  
\end{lemma}
This lemma will be crucially used in the next section: the key point here is that the right-hand side goes to $0$ as $\alpha \to + \infty$ and is solely estimated (albeit not in a sharp way) in terms of the parameters that define the bubble-tree $\Wa$.

\begin{proof}
Let $y \in \overline{\Oa[i]}$ be fixed. Using Lemma \ref{prop:ctltre}, we have
  \begin{equation} \label{eq:borne:globale:3}
  \begin{aligned}
     \sum_{\substack{r,s = 0,\ldots, N\\ r\neq s}} \Ba[r](y)^{\crit-2}\Ba[s](y)& \lesssim \sum_{\substack{r=0,\ldots,N\\ r\neq i}}\Ba(y)^{\crit-2}\Ba[r](y)\\
    &+ \sum_{\substack{r=0,\ldots,N\\ r\neq i}}\Ba(y)\Ba[r](y)^{\crit-2}.  
  \end{aligned}
  \end{equation}
 Let $r \in \{1,\ldots,N\} \backslash \{i\}$ be fixed. We estimate the r.h.s of \eqref{eq:borne:globale:3} by distinguishing several cases that depend on the position of the $r$-th bubble with respect to the $i$-th bubble.

Assume first that  $r\in \sA_i$. Then, by \eqref{eq:ctlBj} we have $\Ba[r](y) \lesssim \frac{(\ma)^{\exc}}{(\sa[ir])^{n-2k}}$ for all $y\in \Oa[i]$. Using \eqref{def:Ba} we then obtain 
    \[ \begin{aligned}
     & (\ma)^{\frac{n+2k}{2}}\Big[\Ba(y)^{\crit-2}\Ba[r](y)+\Ba[r](y)^{\crit-2}\Ba(y)\Big] \\
     & \lesssim \left( \frac{\ma}{\sa[ir]} \right)^{\min(n-2k,4k)}  \lesssim \Big(  \max_{\substack{i,j=1,\ldots N\\ j\neq i \,:\, \ma \lesssim \ma[j]}} (\eps_\a^{ij})^{-\frac{1}{2}} \Big)^{\min(n-2k,4k)},
     \end{aligned} \]
where we used the definition of $\sa[ir]$ and \eqref{eq:struc} to write that $\frac{\ma}{\sa[ir]}\lesssim (\eps^{ir}_\a)^{-\frac{1}{2}}$ for $r \in \sA_i$. In particular, the right-hand goes to $0$ as $\a \to + \infty$ by \eqref{eq:struc}.

     Assume now that $r \in \sB_i$. Then, by definition of $\Ba[r]$, we have $\Ba[r](y) \lesssim \frac{(\ma[r])^{\exc}}{(\rha[ri])^{n-2k}}$ for all $y\in \Oa[i]$, so that 
    \[  \begin{aligned}
&      (\ma)^{\frac{n+2k}{2}}\Big[\Ba(y)^{\crit-2}\Ba[r](y)+\Ba[r](y)^{\crit-2}\Ba(y)\Big]  \\
  & \lesssim \left( \frac{\ma\ma[r]}{\rha[ri]}\right)^{\min(2k, \frac{n-2k}{2})}  \lesssim  \Big(  \max_{\substack{i=1,\ldots N\\ j\neq i \,:\, \ma[j] = o(\ma)}} \Big(\frac{\ma[j]}{\ma[i]}\Big)^{\frac{2k-1}{2(n-1)}} \Big)^{\min(4k, n-2k)},
    \end{aligned} \] 
where we used \eqref{def:rha} to write that $\frac{\ma\ma[r]}{\rha[ri]} \lesssim \left(\frac{\ma[r]}{\ma}\right)^{\frac{2k-1}{n-1}}$ for $r \in \sB_i$. Here again, the right-hand side goes to $0$ as $\a \to + \infty$ by definition of $\sB_i$.

 Assume now that $r\not \in \sA_i$ and $r\not \in \sB_i$: this means that $\ma[r] = \smallo(\ma)$ as $\a \to + \infty$ and  $|\xa[r]-\xa|>2\ra$. Then $\Ba[r](y) \lesssim \frac{(\ma[r])^{\exc}}{(\ra)^{n-2k}}$ for all $y\in \Oa[i]$, so that 
$$ \begin{aligned}
     & (\ma)^{\frac{n+2k}{2}}\Big[\Ba(y)^{\crit-2}\Ba[r](y)+\Ba[r](y)^{\crit-2}\Ba(y)\Big] \\
      & \lesssim  \left( \frac{\ma}{\ra} \right)^{\min(n-2k,4k)} \lesssim \Big(  \max_{\substack{i=1,\ldots N\\ j\neq i \,:\, \ma \lesssim \ma[j]}} (\eps_\a^{ij})^{-\frac12} \Big)^{\min(n-2k,4k)},
\end{aligned}
$$
where we again used \eqref{eq:struc} and \eqref{def:ra} for the last line.

Assume finally that $r=0$, that is $\Ba[r] \equiv 1$. Then, for any $y \in \Oa[i]$,
    \begin{align*}
      &(\ma)^{\frac{n+2k}{2}}\Big[\Ba(y)^{\crit-2}\Ba[r](y)+\Ba[r](y)^{\crit-2}\Ba(y)\Big] \\
      & = (\ma)^{\frac{n+2k}{2}}\big(\Ba(y)\big)^{\crit -2} + (\ma)^{\frac{n+2k}{2}}\Ba(y) \lesssim (\ma)^{\exc} + (\ma)^{2k} 
    \end{align*}
as $\alpha \to + \infty$. This concludes the proof. 
\end{proof}

\section{Sharp weighted estimates for linearised equations at a bubble-tree}\label{sec:lin}

This section contains the core of the analysis of the paper. We prove a uniform invertibility result in appropriate weighted spaces for a linearisation of \eqref{eq:main} at any fixed bubble-tree configuration. We keep the notations of Sections \ref{sec:bulles} and \ref{sec:bulles2}. Throughout this section we fix a bubble-tree $(\Wa)_\a$ given by 
\begin{equation*} 
  \Wa = u_0 + \sumi \Va + \sumij \nu^{ij}_\a \Za,
\end{equation*}
where $u_0\in C^{2k}(\oO)\cap\Sob_0(\O)$ solves \eqref{eq:lmmain}, $[(\xa)_\a,(\ma)_\a,v^i]$, $i=1,\ldots, N$ are $N$ bubbles $\Va$ as in Definitions \ref{def:bub:int} and \ref{def:bub:bdr} and satisfying \eqref{eq:struc}, $(\nu^{ij}_\a)_\a$, $0 \le i \le N$, $1 \le j \le d_i$, are sequences of real numbers converging to $0$ and $d_i = \dim \Ker^i_\a$ is as in \eqref{def:Ka}. For $\alpha \ge 1$ and $y\in \oO$ we define 
\begin{equation}\label{def:PSI}
  \Psi_{\a}(y) = \sumi \ta(y)^{2-2k}\Ba(y)+ \sum_{\substack{i,j = 0,\ldots, N\\i\neq j}} \pBa[j]^{\crit-2} \Ba(y),
\end{equation}
where $\Ba$ and $\ta$ are given by \eqref{def:Ba}. Let $\eta >0$ be a positive real number. We introduce the following weighted norms:
  \begin{equation}\label{def:star}
    \begin{aligned}
      \ns{\phi} & = \max_{y \in \oO} \sum_{l=0}^{2k-1}  \frac{|\nabla^l \phi(y)|}{1+ \sumi \ta(y)^{-l} \Ba(y)} \quad \text{ for } \vp \in C^{2k-1}(\oO), \\
      \nss[\eta]{R} & = \max_{y\in \oO} \frac{|R(y)|}{\Psi_{\a}(y) +\eta \sumi[0] \Ba(y)^{\crit-1}} \quad \text{ for } R \in C^{0}(\oO).\\
      \end{aligned}
  \end{equation}
These norms depend on $\a$, but for simplicity we simply denote them by $\ns{\cdot}, \Vert \cdot \Vert_{**, \eta}$. We let $L_\a$ be as in \eqref{def:La}. The main result of this section is as follows:
\begin{theorem}\label{prop:ptinv}
  Let $(\Wa)_\a$ be defined as in \eqref{def:Wa}. There exist a constant $C_0>0$ and a sequence $(\ea)_\a$ of positive numbers converging to zero as $\alpha \to + \infty$ 
  such that the following holds: for any $R \in C^0(\oO)$ and for any $\alpha$ large enough, there exists a unique function $\phi_\a \in \Ker_\a^\perp $ that solves 
  \begin{equation} \label{eq:lincrit}
    L_\a \phi_\a - (\crit-1)|\Wa|^{\crit-2}\phi_\a = R + \sumij \la (-\D)^k \Za \quad \text{ in } \O,
  \end{equation}
for some real numbers $(\la)_\a$, $i=0,\ldots, N$, $j=1,\ldots, d_i$. Moreover, $\phi_\a \in C^{2k-1}(\overline{\Omega})$ and satisfies
  \begin{equation}\label{eq:Iptinv}
    \ns{\phi_\a} \leq C_0 \ea \nss{R}.
  \end{equation}
\end{theorem}
Recall that $\Za$ and $\Ker_\a$ are given by \eqref{def:Ka}, and that the orthogonal is taken with respect to the $H^{k}_0(\O)$ scalar product given by \eqref{norm:Dk2}. 

Some remarks are in order. The operator $\tilde{L}_\a  = L_\a  - (\crit-1)|\Wa|^{\crit-2}$ is, formally, the linearisation of \eqref{eq:main} at $\Wa$ -- even though $\Wa$ is in general only an approximate solution of \eqref{eq:main}. Theorem \ref{prop:ptinv} should thus be understood as a uniform invertibility result for the linearised version of \eqref{eq:main} at a general bubble-tree $\Wa$: it provides sharp pointwise bounds, in weighted norms adapted to the configuration of the bubble-tree (i.e. norms that capture the leading order of size of $\Wa$ at each scale), for solutions of the linearised equation \eqref{eq:lincrit}. Theorem~\ref{prop:ptinv} does not assume that $\Wa$ originates from a Struwe-type decomposition like \eqref{eq:decHk}: $\Wa$ could be any bubble-tree and Theorem~\ref{prop:ptinv} is formally decorrelated from \eqref{eq:main}. Theorem~\ref{prop:ptinv} shows that the only obstruction to the uniform (in $\a$) invertibility of $\tilde{L}_\a$ is the kernel $\Ker_\a$, which simply consists of the contributions of the kernels of each element $u_0, \Va[1], \ldots, \Va[N]$ appearing in the bubble-tree. This fact has long been known at the energy level (see for instance Proposition~\ref{prop:lininv} below), but the real novelty -- and the main analytical difficulty --  in Theorem~\ref{prop:ptinv} is to prove that uniform invertibility \eqref{eq:Iptinv} holds true in the weighted norms given by \eqref{def:star}. In particular, Theorem~\ref{prop:ptinv} remains true for \emph{any} bubble-tree configuration $\Wa$, regardless of the choice of $u_0$ and the bubbles $[(\xa)_\a,(\ma)_\a,v^i]$, provided they satisfy \eqref{eq:struc}. The constant $C_0$ and the sequence $(\eta_\a)_\a$ in \eqref{eq:Iptinv} only depend on $n,k$, $\O$, $\|A_{l,\a}-A_l\|_{C^l(\oO)}$, and $\Wa$ (see \eqref{def:eta} below).  The weighted norms \eqref{def:star} that appear in \eqref{eq:Iptinv} are dictated by the bubble-tree structure. Simple computations show indeed that $\Wa$ satisfies 
$$ \big| L_\a \Wa - |\Wa|^{2^\sharp - 2} \Wa \big| \lesssim \Psi_{\a} +\eta_\a \sumi[0] (\Ba)^{\crit-1} \quad \text{ in } \O$$
for some sequence $\eta_\a \to 0$ of positive numbers (we will prove this in \eqref{eq:eta4} below). The choice of the $\Vert \cdot \Vert_{**, \eta_\a}$ norm is therefore motivated by the setting of Section~\ref{sec:nonlinear}, where we will crucially rely on Theorem~\ref{prop:ptinv} to prove global pointwise estimates for solutions of \eqref{eq:main}. The $\Vert \cdot \Vert_*$ norm is, in turn, imposed by $\Vert \cdot \Vert_{**, \eta_\a}$: for a fixed $0 \le l \le 2k-1$, the weight $1+ \sumi \ta(y)^{-l} \Ba(y)$ canonically appears when integrating a Green's function against $\Psi_\a + \eta_\a \sumi[0] (\Ba)^{\crit-1} $ (see Lemma \ref{prop:lem2} below). Estimate \eqref{eq:Iptinv} therefore shows that, insofar as global pointwise estimates are concerned, solutions of \eqref{eq:lincrit} behave as if they formally satisfied the equation
\[   L_\a \phi_\a = R + \sumij \la (-\D)^k \Za \quad \text{ in } \O. \]
This had already been observed in \cite{Pre24} in the case $k=1$. A similar principle was used by \cite{DengSunWei} to prove the optimal quantitative stability of Struwe's decomposition for positive functions in $\R^n$. To prove Theorem \ref{prop:ptinv} we adapt the strategy of proof of \cite[Theorem 3.2]{Pre24}. The new approach that we develop here and which uses the weighted norms \eqref{def:star} allows us to tackle simultaneously the polyharmonicity of any order and the occurrence of possible boundary bubbles for a general bubble-tree $\Wa$, all the while simplifying the original proof of \cite[Theorem 3.2]{Pre24} even when $k=1$.

\subsection{Preliminary results}

We first state an $H^k_0(\O)$ version of Theorem~\ref{prop:ptinv}: 
\begin{proposition}\label{prop:lininv}
  Let $R \in H^{-k}(\O)$. For any $\a$ large enough there exists a unique $\phi_\a \in \Ker_\a^\perp \sub \Sob_0(\O)$ and unique real numbers  $(\la)_\a$ for $i=0,\ldots, N$, $j=1,\ldots, d_i$ such that 
  \begin{equation*}
    L_\a \phi_\a - (\crit-1)|\Wa|^{\crit-2}\phi_\a = R + \sumij \la (-\D)^k \Za.
  \end{equation*}
  This $\vp_\a$ satisfies in addition
  \begin{equation}\label{eq:Ilininv}
   \Vert \phi_\a \Vert_{H^k_0(\Omega)} \leq C \Snorm{-k}{R}
  \end{equation}
  for some $C>0$ independent of $\alpha$, $R,\phi_\a$. In particular, the map $\Ker_\a^\perp \to \Ker_\a^\perp : \Pi_{\Ker_\a^\perp}\big((-\D)^{-k}R\big) \mapsto \phi_\a$ is a uniformly bicontinuous isomorphism.
\end{proposition}
\begin{remark}\label{rk:Dinv}
  Here and in the following we denote by $(-\D)^{-k} : \Sob[-k](\O) \to \Sob_0(\O)$ the linear mapping that sends $f\in \Sob[-k](\O)$ to the unique weak solution $h\in \Sob_0(\O)$ of $(-\Delta)^k h = f$ in $\O$. 
\end{remark}
\noindent The proof of Proposition \ref{prop:lininv} follows from standard arguments. We refer for instance to \cite{Pre24} and \cite{RobertVetois} for $k=1$ and to \cite{CarJLMS} for $k \ge 1$ with a single positive bubble. In the rest of this section we show how the weak control \eqref{eq:Ilininv} can be improved into the weighted pointwise control \eqref{eq:Iptinv}. We first state two technical results: 

\begin{lemma}\label{prop:lem1}
 Let $\Psi_\a$ be as in \eqref{def:PSI}. There exists a sequence $\eta_{1,\a} \to 0$ depending only on $[(\xa)_\a,(\ma)_\a,v^i]$  such that 
  \[  \Lnorm[\O]{\frac{2n}{n+2k}}{\Psi_\a} \leq \eta_{1,\a}.
  \]
 \end{lemma}
\noindent The proof follows from direct computations using the explicit form of $\Psi_\a$, and relies on \eqref{eq:struc}. Similar arguments can be for instance found in \cite[Proposition 3.4]{RobertVetois}.
\begin{lemma}\label{prop:lem2}
  There exists a sequence $\eta_{2,\a} \to 0$  depending only on $[(\xa)_\a,(\ma)_\a,v^i]$ such that, for $0 \le l \le 2k-1$ and for all $x\in \oO$,
  \[  \inty{|x-y|^{2k-n-l}\Psi_\a(y)} \leq \eta_{2,\a} \Big(1 + \sumi \ta(x)^{-l} \Ba(x)\Big).
  \]
\end{lemma}
\noindent For completeness we prove Lemma \ref{prop:lem2} in Appendix \ref{app:technicalresults} below. In addition to $(\eta_{1,\a})_\a$ and $(\eta_{2,\a})_\a$ given by Lemmas \ref{prop:lem1} and \ref{prop:lem2}, we also define the following sequences that depend on $[(\xa)_\a,(\ma)_\a,v^i]_{\{1 \le i \le N \}}$, $\|A_{l,\a}-A_l\|_{C^l(\oO)}$ and $(\nu^{ij}_\a)_\a$:
\begin{align*}
  \eta_{3,\a}^{(1)} &:= \max_{\substack{i=1,\ldots N\\ j\neq i \,:\, \ma \lesssim \ma[j]}} (\eps_\a^{ij})^{-\frac{1}{2}} & \eta_{3,\a}^{(2)} &:= \max_{\substack{i=1,\ldots N\\ j\neq i \,:\, \ma[j] = o(\ma)}} \Big(\frac{\ma[j]}{\ma[i]}\Big)^{\frac{2k-1}{2(n-1)}},
\end{align*}
where $\eps_\a^{ij}$ is as defined in \eqref{eq:struc}. We may now define two sequences $(\eta_{3,\a})_\a$ and $(\eta_{4,\a})_\a$ by 
\begin{equation}\label{eq:lem3}
  \eta_{3,\a} := (\eta_{3,\a}^{(1)})^{\min(n-2k, 4k)} + (\eta_{3,\a}^{(2)})^{\min(n-2k, 4k)} + \max_{i=1,\ldots, N} (\ma)^{\min(\frac{n-2k}{2},2k,1)},
\end{equation}
and
\begin{equation}\label{eq:lem4}
  \eta_{4,\a} := \max_{\substack{i=0,\ldots,N\\j=1,\ldots,d_i}} |\nu_\a^{ij}| + \sum_{l=0}^{p} \|A_{l,\a}-A_l\|_{C^{l}(\oO)}.
\end{equation}
As we mentioned in the discussion following Theorem~\ref{prop:ptinv} the sequence $\eta_\a$ that we will obtain in \eqref{eq:Iptinv} explicitly depends on $n,k$, $\O$, $\|A_{l,\a}-A_l\|_{C^l(\oO)}$, and $\Wa$. We define it as follows: 
\begin{equation}\label{def:eta}
  \ea = \max\{\eta_{1,\a}, \eta_{2,\a},\eta_{3,\a},\eta_{4,\a}\}.
\end{equation}
Until the end of this section the notation $(\eta_\a)_\a$ will always refer to the sequence defined by \eqref{def:eta}. We define $\eta_\a$ in this way \emph{a posteriori}, so that every sequence that arises in the course of the proof of Theorem \ref{prop:ptinv} can be simply bounded from above by $\eta_\a$. (One can check, precisely, that $\eta_{1,\a}$ appears in \eqref{eq:I1phi}, $\eta_{2,\a}$ in \eqref{eq:eta2}, and $\eta_{3,\a}$ in \eqref{eq:estRa}, while $\eta_{4,\a}$ is used in \eqref{eq:eta4}). We introduce a final piece of notation: we will denote by $(\ve_\a)_\a$ any other sequence of positive numbers, which may change from one line to the other,  such that 
  \begin{equation}\label{def:ea}
    \begin{bigcases}
        &\ve_\a \to 0 \qquad \text{as }\a \to \infty, \\
        &(\ve_\a)_\a \text{ only depends on $n,k$, $\O$, $\|A_{l,\a}-A_l\|_{C^l(\oO)}$, and $\Wa$}.
    \end{bigcases}
  \end{equation} 
Several distinct such sequences will be numbered as $\ve_{1,\a},\ve_{2,\a},\ldots$. 

\subsection{Proof of Theorem \ref{prop:ptinv} assuming a key estimate}\label{sec:ptinv1} In this subsection we show that the proof of Theorem \ref{prop:ptinv} reduces to the proof of a key estimate, given in Proposition \ref{prop:estIl} below. Throughout the rest of this section we let $(R_\a)_\a$ be a sequence of functions in $C^0(\oO)$  and $\phi_\a \in \Ker_\a^\perp$ be the unique solution to \eqref{eq:lincrit} for some $(\la)_\a$, $i=0,\ldots, N,j=1,\ldots,d_i$ given by Proposition \ref{prop:lininv} (which applies since by the usual duality pairing we have $C^0(\overline{\O}) \subset H^{-k}(\O)$). By standard elliptic theory $\vp_\a \in C^{2k-1}(\overline{\O})$, so that  $ \ns{\phi_\a}$ is well-defined. A first simple observation is that, by \eqref{eq:Ilininv} and Lemma \ref{prop:lem1}, we have
\begin{equation}\label{eq:I1phi}
  \Vert \phi_\a \Vert_{H^k_0(\Omega)} \lesssim \Snorm{-k}{R_\a} \lesssim \ea \nss{R_\a},
\end{equation}
where we used the continuous dual Sobolev embedding of $L^{\frac{2n}{n+2k}}(\O)$ into $\Sob[-k](\O)$. As a consequence we claim that for $i=0,\ldots,N$, $j=1,\ldots, d_i$, we have
\begin{equation}\label{eq:Ilama}
  |\la| \lesssim \ea\nss{R_\a}.
\end{equation}
To prove \eqref{eq:Ilama} we integrate \eqref{eq:lincrit} against $\Za[i_0j_0]$ for some $i_0\in \{0,\ldots,N\}$ and $j_0\in \{1,\ldots,d_i\}$: using \eqref{eq:ortZa} and the second inequality in \eqref{eq:I1phi} we have that
\begin{multline*}  
  \intM{\O}{\big[L_\a \Za[i_0j_0] -(\crit-1)|\Wa|^{\crit-2}\Za[i_0j_0]\big]\phi_\a}\\ = \bigO(\ea\nss{R_\a})+ \la[i_0j_0]+ \smallo\Big(\sumij |\la|\Big).
\end{multline*}
Straightforward computations using \eqref{eq:IWa} and the definition of  $\Za[i_0j_0]$ show that
  \begin{equation}\label{eq:IDDZa}
   \begin{aligned} 
       \big| (-\D)^k \Za[0j] \big| & \lesssim 1 \quad \text{ for }1 \le j \le d_0 \text{ and } \\ 
   \big| (-\D)^k \Za \big| &  \lesssim (\Ba)^{\crit-1} + \sum_{l=0}^{2k-1} (\ta)^{-l}\Ba \quad \text{ for }1 \le i \le N \text{ and } 1 \le j \le d_i. \\
\end{aligned} 
  \end{equation}
Using \eqref{eq:IDDZa} it is easily seen that $L_\a \Za[i_0j_0] -(\crit-1)|\Wa|^{\crit-2}\Za[i_0j_0]$ is uniformly bounded in $\Sob[-k](\O)$ as $\a \to + \infty$. Thus, using \eqref{eq:I1phi} and summing over all $i_0,j_0$ proves \eqref{eq:Ilama}. For $0 \le l \le 2k-1$ and $x\in \oO$ we define in what follows
\begin{equation}\label{def:Il}
  I^l_\O(x) = \inty{|x-y|^{2k-n-l}\Big(\sumi \Ba(y)\Big)^{\crit-2}|\phi_\a(y)|}.
\end{equation}
We have the following pointwise estimates on $\phi_\a$ and its derivatives:
\begin{lemma}
 For every $0 \le l \le 2k-1$  and for all $x\in \oO$, 
  \begin{equation}\label{eq:I2phi}
    |\nabla^l \phi_\a(x)| \lesssim 
    \eta_\a(\nss{R_\a} +\ns{\phi_\a})\Big(1+\sumi \ta(x)^{-l}\Ba(x) \Big)+ I^{l}_\O(x). 
  \end{equation}
\end{lemma}
\begin{proof}
 We re-write the equation \eqref{eq:lincrit} satisfied by $\phi_\a$ as
  \begin{multline*}
    L\phi_\a -(\crit-1)|u_0|^{\crit-2}\phi_\a = R_\a + (L-L_\a)\phi_\a\\ + (\crit-1)\Big(|\Wa|^{\crit-2}-|u_0|^{\crit-2}\Big)\phi_\a + \sumij \la (-\D)^k \Za.
  \end{multline*}
  We let $G_0 \in L^1(\O\times\O)$ be the Green's function of the operator $L-(\crit-1)|u_0|^{\crit-2}$ in $\O$ with Dirichlet boundary conditions. Since $u_0$ is a solution of \eqref{eq:lmmain}, $u_0 \in C^{2k}(\oO)$ and it follows from \cite{GazGruSw10} and \cite{GrunauRobert} that there exists $C>0$ such that for $0 \le l \le 2k-1$, for all $x\neq y$ in $\oO$,
  \[  |\nabla^l G_0(x,y)|\leq C|x-y|^{2k-n-l}.
  \]
Let $0 \le l \le 2k-1$ and $\beta$ be a multi-index such that $|\beta| =l$. We write a representation formula for $\phi_\a$ that we differentiate. Since $\phi_\a \in \Ker_\a^\perp \sub \Ker_0^\perp$, it writes as follows: for any $x \in \O$,
\begin{equation} \label{eq:formulederep}
\begin{aligned}
\nabla^\beta \phi_\a(x) & = \int_{\O} \nabla^\beta G_0(x,y) \Bigg[R_\a + (L-L_\a)\phi_\a \\
&+ (\crit-1)\Big(|\Wa|^{\crit-2}-|u_0|^{\crit-2}\Big)\phi_\a  + \sumij \la (-\D)^k \Za  \Bigg] dy.
\end{aligned} 
\end{equation}
We now estimate each term in \eqref{eq:formulederep}. First, using Lemma \ref{prop:lem2} and \eqref{def:eta}, and with Lemma \ref{prop:giraud} in the Appendix below, we have
  \begin{multline}\label{eq:eta2}
    \abs{\inty{\nabla^\beta G_0(x,y) R_\a(y)}}\\
    \begin{aligned}
      &\lesssim  \nss{R_\a}\inty{|x-y|^{2k-n-l}\left[\Psi_\a(y)+ \eta_\a\Big(\sumi[0] \pBa^{\crit-1}\Big) \right]}\\
      &\lesssim  \ea \nss{R_\a} \Big(1+\sumi \ta(x)^{-l}\Ba(x) \Big).
    \end{aligned} 
  \end{multline}
By \eqref{def:La} and \eqref{def:eta} we have, for any $y \in \O$,
$$ |(L-L_\a)\phi_a(y)| \lesssim \left(\sum_{m=0}^{p}\|A_{m,\a}-A_m\|_{C^m(\oO)}\right) \sum_{m = 0}^{2k-2} |\nabla^{m} \phi_\a(y)| \lesssim \eta_\a \sum_{m = 0}^{2k-2} |\nabla^{m} \phi_\a(y)|  .$$
Using the definition of the weighted norm $\Vert \cdot \Vert_*$ in \eqref{def:star} we thus obtain 
 \begin{equation} \label{eq:eta2bis} 
 \begin{aligned} 
 &  \abs{\inty{\nabla^{\beta} G_0(x,y)(L-L_\a)\phi_a(y)}}\\
 &  \lesssim \eta_\a \ns{\phi_\a} \inty{|x-y|^{2k-n-l}\sum_{m=0}^{2k-2}\Big(1+\sumi \ta(y)^{-m}\Ba(y)\Big)}  \\
 & \lesssim \eta_\a \ns{\phi_\a}\Big(1+\sumi \ta(x)^{-l}\Ba(x) \Big).
  \end{aligned} 
  \end{equation}
To obtain the last inequality in \eqref{eq:eta2bis} we observed that for any $0 \le m \le 2k-2$ we have $\ta(y)^{-m}\Ba(y) \lesssim (\ma)^{\frac{n-2k}{2}} \ta(y)^{2-n} \lesssim (\ma)^{\frac{n-2k}{2}} \ta(y)^{1-n}$ and we again applied Lemma \ref{prop:giraud}. Using \eqref{def:Wa}, \eqref{est:bubble} (for $l=0$) and \eqref{def:eta} we now have 
  \[   \Big ||\Wa|^{\crit-2}-|u_0|^{\crit-2}\Big| \lesssim \Big(\sumi \Ba\Big)^{\crit-2} + \sumi \Ba + \eta_{4,\a},
  \]
  where $\eta_{4,\a}$ is as in \eqref{eq:lem4}. Therefore, using again Lemma \ref{prop:giraud} we have 
  \begin{equation} \label{eq:eta2ter}
  \begin{aligned}
    \Bigg| (2^\sharp-1)& \inty{\nabla^\beta G_0(x,y) \Big (|\Wa|^{\crit-2}-|u_0|^{\crit-2}\Big) \phi_\a(y)} \Bigg|\\
    &  \lesssim \eta_\a \ns{\phi_\a} \Big(1+\sumi \ta(x)^{-l}\Ba(x) \Big) + I^{l}_\O(x)
\end{aligned}
  \end{equation}
where $I^{l}_\O$ is defined in \eqref{def:Il}. The last integral in \eqref{eq:formulederep} is estimated using  \eqref{eq:Ilama}  and \eqref{eq:IDDZa}: straightforward computations show that
  \begin{equation}   \label{eq:eta24}
  \begin{aligned}
    \sumij & |\la| \abs{\inty{\nabla^\beta G_0(x,y)(-\D)^k \Za(y)}}\\ 
    &\lesssim \ea \nss{R_\a} \Big(1+\sumi \ta(x)^{-l}\Ba(x) \Big).
\end{aligned} 
  \end{equation}
Combining \eqref{eq:eta2},  \eqref{eq:eta2bis}, \eqref{eq:eta2ter}, \eqref{eq:eta24} into \eqref{eq:formulederep}  proves \eqref{eq:I2phi}.
\end{proof}
We let $\ra$ be the radius of influence of the $i$-th bubble, which is defined in \eqref{def:ra}. The next result estimates, for any $1 \le i \le N$, the integral $I^{l}_{ \O \cap B(\xa, \ra)}(x)$ that arises in \eqref{eq:I2phi} and which is defined in \eqref{def:Il}:

\begin{proposition}\label{prop:estIl}
  There exist a sequence $\eps_\a \to 0$  as in \eqref{def:ea} such that for any $1 \le i \le N$, $0 \le l \le 2k-1$ and  $x\in \oO$ we have 
  \begin{equation}\label{eq:estIl}
\begin{aligned}
 I^{l}_{\O \cap B(\xa, \ra)}(x)    \lesssim \firstTerm,
\end{aligned}
  \end{equation}
  where $\ea$ is as in \eqref{def:eta}, .
\end{proposition}
Proposition~\ref{prop:estIl} is the core of the analysis of this section and the main ingredient in the proof of Theorem~\ref{prop:ptinv}. For the sake of clarity we postpone its proof to subsection~\ref{key:analytical:argument} below. Assuming it we now conclude the proof of Theorem~\ref{prop:ptinv}:

\begin{proof}[Proof of Theorem \ref{prop:ptinv} assuming Proposition~\ref{prop:estIl}]
Let $x \in \overline{\O}$ and $0 \le l \le 2k-1$ be fixed. Using the definition of $I^{l}_{\O}$ in \eqref{def:Il} we easily have 
  \begin{equation*}
    I^{l}_\O(x)      \leq  \sumi I^{l}_{\O \cap B(\xa,\ra)} (x) +  I^{l}_{\O \backslash \cup_{i=1}^N B(\xa,\ra)} (x).
   \end{equation*}
  On the one hand, Proposition \ref{prop:estIl} shows that 
  $$ \sumi I^{l}_{\O \cap B(\xa,\ra)} (x) \lesssim \firstTerm,$$
while on the other hand, using the definition of the $\Vert \cdot \Vert_*$ norm in \eqref{def:star} we have 
 $$   \begin{aligned}
 & I^{l}_{\O \backslash \cup_{i=1}^N B(\xa,\ra)} (x) \\
 & \le \ns{\phi_\a}\inty[\O\setminus\bigcup\limits_{i=1}^N B(\xa,\ra)]{|x-y|^{2k-n-l}\sum_{j=1}^N\Big[\pBa[j]^{\crit-2}+\pBa[j]^{\crit-1}\Big]} \\
 & \le \ve_\a \ns{\phi_\a} \Big(1+\sumi \ta(x)^{-l}\Ba(x) \Big),
 \end{aligned} $$ 
 for some sequence $(\ve_\a)$ satisfying \eqref{def:ea}, where the last inequality follows from \eqref{eq:IIa2} and  Lemma \ref{prop:lemtrou} below, and since for any $1 \le i \le N$ we have $\frac{\ra}{\ma} \to + \infty$ as $\a \to + \infty$ by \eqref{rasurma}. Combining the latter two estimates shows that 
 \begin{equation} \label{est:I:last}
    I^{l}_\O(x) \lesssim \firstTerm
  \end{equation}
 where $\eta_\a$ is given by \eqref{def:eta}. Plugging \eqref{est:I:last} into \eqref{eq:I2phi} now shows that for any $0 \le l  \le 2k-1$ and any $x\in \oO$ we have 
  \begin{equation}\label{eq:I4phi}
    |\nabla^{l} \phi_\a(x)| \lesssim \firstTerm.
  \end{equation}
For any $\alpha \ge 1$ we may now let $y_\a \in \oO$ be a point where 
  \[  \ns{\phi_\a} = \sum_{0 \le l  \le 2k-1} \frac{|\nabla^{l} \phi(y_\a)|}{1+ \sumi \ta(y_\a)^{-l} \Ba(y_\a)},  \]
  whose existence simply follows from the definition of $\Vert \cdot \Vert_*$ in \eqref{def:star}. By evaluating \eqref{eq:I4phi} at $y_\a$ and summing over $l$ we get that there is a constant $C >0$ independent of $\a$ such that
  \[  \ns{\phi_\a} \leq C(\ea \nss{R_\a} + \eps_\a \ns{\phi_\a}).
  \]
Since $\ve_\a \to 0$ as $\alpha \to + \infty$ we conclude that there exists $C_0>0$ such that
  \[  \ns{\phi_\a} \leq C_0\ea\nss{R_\a}
  \]
  for $\alpha$ large enough. This proves \eqref{eq:Iptinv} and concludes the proof of Theorem~\ref{prop:ptinv}. 
\end{proof}

\subsection{Proof of Proposition~\ref{prop:estIl}} \label{key:analytical:argument}

In this subsection we prove Proposition~\ref{prop:estIl}. We crucially use the bubble-tree formalism that we introduced in subsection \ref{sec:bubtre}, and we use the notations from that section. We recall in particular that we assume \eqref{eq:renum}.  The main technical difficulty in the proof of Proposition~\ref{prop:estIl} is to obtain the \emph{weighted} norm $\Vert \phi_\a \Vert_*$ as the error term in the right-hand side of \eqref{eq:estIl}. Estimate \eqref{eq:estIl} with weaker norms would indeed be easily obtained: for instance using the definition of $I^{l}_{\O}$ in \eqref{def:Il} together with Lemma \ref{prop:lemtrou:0} below we have
\[  I^{l}_\O(x) = \smallo\left(\Lnorm[\O]{\infty}{\phi_\a}\Big(1+\sumi\ta(x)^{-l}\Ba(x)\Big)\right)
\]
for all $x\in \oO$ as $\a \to +\infty$. The latter estimate, however, does not yield the desired precision that we need in Theorem~\ref{prop:ptinv}. 

\medskip

For any $1 \le m \le N$ we introduce the following proposition:
  \begin{equation}\label{def:induc}
P_m: \eqref{eq:estIl} \text{ holds true for every } m \le i \le N, x \in \overline{\O} \text{ and } 0 \le l \le 2k-1.
  \end{equation} 
Proving Proposition~\ref{prop:estIl} amounts to proving that $P_1$ holds true. We will prove this by reverse induction in the bubble-counting index $m$, starting from $m=N$. The proof that $P_N$ holds true and that the induction property is satisfied follow the same lines, and we give a single proof that works in both cases. Throughout this subsection we thus fix $1 \le m\le N$. If $m=N$ we do not assume anything, but if $m \le N-1$ we assume that $P_{m+1}$ defined in \eqref{def:induc} is true. In both cases we want to prove that $P_m$ holds true, and it is easily seen that this reduces to proving that \eqref{eq:estIl} holds true when $i=m$, for any $x \in \overline{\O}$ and $0 \le l \le 2k-1$. 

\smallskip

The strategy of proof goes as follows. The main difficulty consists in obtaining sharp estimates on  $I^{l}_{\O \cap B(\xa, \ra)}(x)$. To estimate the latter we will improve the naive control on $\vp_\a$ given by the $\Vert \cdot \Vert_*$ norm and prove that, in $\overline{\Oa}$, $\vp_\a$ is quantitatively smaller than $\Ba[m]$. We measure this by the following local weighted norm: we define
  \begin{equation}\label{def:PHI}
    \Pa = \max_{y\in \overline{\Oa}} \abs{\frac{\phi_\a(y)}{\ta[m](y)\Ba[m](y)}}
  \end{equation}
  where we recall that $\Ba[m]$ and $\ta[m]$ are given by \eqref{def:Ba}.  We will estimate $\Pa$ in terms of $\eta_\a \Vert R_\a \Vert_{**, \eta_\a} + \ve_\a \Vert \vp_\a \Vert_{*}$ and deduce a sharp control on $I^{l}_{\O \cap B(\xa, \ra)}(x)$ by a self-improving argument. Our first result estimates $ I^{l}_{\O \cap B(\xa[m],\ra[m])}(x)$ in terms of $\Vert \phi_\a \Vert_{*}$ and $\Pa$:

\begin{lemma}
There exists a sequence $\eps_{1,\a}\to 0$ satisfying \eqref{def:ea} such that, for any  $0 \le l  \le 2k-1$ and for all $x\in \oO$, we have 
  \begin{multline}\label{eq:I1Il}
    I^{l}_{\O \cap B(\xa[m],\ra[m])}(x)
   \lesssim \firstTerm[1,\a]\\+\ma[m] \Pa \ta[m](x)^{-l}\Ba[m](x).
  \end{multline}
\end{lemma}
Note the additional factor $\ma[m]$ in the last term in \eqref{eq:I1Il}, which is due to the weight $\ta[m]$ in the definition of $\Pa$. 

\begin{proof}
Using  \eqref{meaning:ra:1}, \eqref{forgotten:label},  Lemma \ref{prop:ctltre} and \eqref{def:Il} we have
  \begin{multline}\label{tmp:decIlo}
  I^{l}_{\O \cap B(\xa[m],\ra[m])}(x) \lesssim \inty[\Oa]{|x-y|^{2k-n-l}\Ba[m](y)^{\crit-2}|\phi_\a(y)|}\\ +\inty[{\O\cap\cupj B(\xa[j],\rha[jm])}]{|x-y|^{2k-n-l}\Big(\sumj\Ba[j](y)^{\crit-2} \Big)|\phi_\a(y)|},
  \end{multline}
  where $\Oa$ is as  \eqref{def:Oa} and where $\rho_{\a}^{jm}$ are defined in \eqref{def:rha}. In the particular case $m=N$ we recall that $\sB_N = \emptyset$ by \eqref{def:Bi} and thus that $\Oa[N] = \O \cap B(\xa[N],\ra[N])$ by \eqref{def:Oa}.   Using the definition of $\ns{\cdot}$ in \eqref{def:star}, we get that
 $$ \begin{aligned}
   & \inty[{\O\cap\cupj B(\xa[j],\rha[jm])}]{|x-y|^{2k-n-l}\sumj\Ba[j](y)^{\crit-2}|\phi_\a(y)|} \\
   & \lesssim \sumj  I^{l}_{\O \cap B(\xa[j],\ra[j])}(x) 
    \\  &+ \ns{\phi_\a} \inty[{\cupj \big(B(\xa[j],\rha[jm])\setminus B(\xa[j],\ra[j])\big)}]{|x-y|^{2k-n-l}\sumj \Ba[j](y)^{\crit-1}} \\
     & \lesssim \sumj  I^{l}_{\O \cap B(\xa[j],\ra[j])}(x) 
    \\  &+ \ns{\phi_\a} \sum_{j \in \mathcal{B}_m} \inty[{\O \setminus B(\xa[j],\ra[j])}]{|x-y|^{2k-n-l} \Ba[j](y)^{\crit-1}}. \\
  \end{aligned} $$
By definition of  $\sB_m$ in \eqref{def:Bi} we have $\sB_m\subeq\{m+1,\ldots,N\}$, so that the induction assumption \eqref{def:induc} shows that 
\begin{equation}
\label{forgotten:label:2}
 \sumj  I^{l}_{\O \cap B(\xa[j],\ra[j])}(x)  \lesssim  \firstTerm.
  \end{equation}
Since $\frac{\ra[j]}{\ma[j]} \to + \infty$ for every $1 \le j \le N$ by \eqref{rasurma},  Lemma \ref{prop:lemtrou} below shows that 
$$ \ns{\phi_\a} \sum_{j \in \mathcal{B}_m} \inty[{\O \setminus B(\xa[j],\ra[j])}]{|x-y|^{2k-n-l} \Ba[j](y)^{\crit-1}} \le \ve_a  \ns{\phi_\a}  \Big(1+\sumi\ta(x)^{-l}\Ba(x)\Big) $$
for some sequence $\eps_\a \to 0$ satisfying \eqref{def:ea}, and combining the latter with \eqref{forgotten:label:2} shows that 
  \begin{multline}\label{tmp:onOac}
    \inty[{\O\cap\cupj B(\xa[j],\rha[jm])}]{|x-y|^{2k-n-l}\big(\Ba[m](y)+ \sumj\Ba[j](y)\big)^{\crit-2}|\phi_\a(y)|} \\
    \lesssim \firstTerm,
  \end{multline}
  for all $x\in \oO$ and $l=0,\ldots, 2k-1$, for some sequence $\eps_\a \to 0$ satisfying \eqref{def:ea}. Finally, direct computations using Lemma~\ref{lemme:ordre:deux} below and \eqref{def:PHI} show that for all $x\in \oO$ and $l=0,\ldots, 2k-1$ we have
  \begin{equation}\label{tmp:onOa}  
    \inty[\Oa]{|x-y|^{2k-n-l}\pBa[m]^{\crit-2}|\phi_\a(y)|} \lesssim \ma[m] \Pa \ta(x)^{-l}\Ba[m](x).
  \end{equation}
 Combining \eqref{tmp:onOac} and \eqref{tmp:onOa} in \eqref{tmp:decIlo} proves \eqref{eq:I1Il}. 
 \end{proof}
 
 We now prove sharp estimates on $\phi_\a$ and its derivatives in $\overline{\Oa}$, still depending on $\Vert \phi_\a \Vert_{*}$ and $\Pa$:

\begin{lemma}
There exists a sequence $\eps_{2,\a}\to 0$ as in \eqref{def:ea} such that, for $0 \le l \le 2k-1$ and for all $x \in \oO \cap B(\xa[m],\ra[m])$, we have
  \begin{multline}\label{eq:I3phi}
    |\nabla^{l} \phi_\a(x)| \lesssim \firstTerm[2,\a]\\
    + \ma[m] \Pa \ta[m](x)^{-l}\Ba[m](x) + \frac{(\ma[m])^{\exc}}{(\ra[m])^{n-2k+l}}\ns{\phi_\a}.
  \end{multline}
\end{lemma}
\begin{proof}
By definition of $I^{l}_\O$ in \eqref{def:Il} we have, for any $x \in \oO$,
\begin{equation} \label{eq:IplusI}
I^{l}_\O (x) \le   I^{l}_{\O \cap B(\xa[m],\ra[m])}(x)+ I^{l}_{\O \backslash B(\xa[m],\ra[m])}(x). 
\end{equation}
By definition of $\ns{\cdot}$ in \eqref{def:star}, we have 
  \begin{equation} \label{est:der:0}
  \begin{aligned}
  & I^{l}_{\O \backslash B(\xa[m],\ra[m])}(x)\\
  & \lesssim \sumi[m+1]  I^{l}_{\O \cap B(\xa,\ra)}(x)\\ 
  &+ \ns{\phi_\a} \sum_{j=1}^N\inty[\O\setminus \bigcup\limits_{i=m}^{N}B(\xa,\ra)]{|x-y|^{2k-n-l} \Big[ \Ba[j](y)^{\crit-2}+  \Ba[j](y)^{\crit-1}\Big]} \\
  & \lesssim \sumi[m+1]  I^{l}_{\O \cap B(\xa,\ra)}(x) +  \ns{\phi_\a}\sum_{j=1}^N\inty[\O]{|x-y|^{2k-n-l} \Ba[j](y)^{\crit-2}} \\ 
  &+ \ns{\phi_\a} \sum_{j=1}^{m-1}\inty[\O\setminus \bigcup\limits_{i=m}^{N}B(\xa,\ra)]{|x-y|^{2k-n-l} \Ba[j](y)^{\crit-1}} \\
  &+ \ns{\phi_\a} \sum_{j=m}^N\inty[\O\setminus B(x_\a^j,r_\a^j)]{|x-y|^{2k-n-l}   \Ba[j](y)^{\crit-1}}
\end{aligned}
  \end{equation}
    We estimate the terms involving $\Ba[j](y)^{\crit-1}$ that appear in \eqref{est:der:0}. First, if $j \in \{1,\ldots, m-1\}$ we have $j \in \sA_m$ by \eqref{eq:renum} and \eqref{def:Ai}. Thus Lemma \ref{prop:giraud} below and \eqref{eq:ctlBj} show that for $x\in \oO\cap B(\xa[m],\ra[m])$ we have 
  \begin{equation} \label{est:der:1}
     \sum_{j=1}^{m-1}\inty{|x-y|^{2k-n-l}\Ba[j](y)^{\crit-1}} \lesssim \ta[j](x)^{-l}\Ba[j](x) \lesssim \frac{(\ma[m])^{\exc}}{(\ra[m])^{n-2k+l}}.
  \end{equation}
 We now let $j \in \{m, \dots, N\}$. Since $\frac{\ra[j]}{\ma[j]} \to + \infty$  by \eqref{rasurma}, Lemma \ref{prop:lemtrou} below shows that there exists $(\ve_\a)_\a$ satisfying \eqref{def:ea} such that 
  \begin{equation}\label{eq:trouBj}
    \inty[{\O\setminus B(\xa[j],\ra[j])}]{|x-y|^{2k-n-l}\Ba[j](y)^{\crit-1}} \lesssim \ve_\a \ta[j](x)^{-l}\Ba[j](x) 
  \end{equation} 
holds for any $x\in \oO$ as $\a \to + \infty$. Combining \eqref{est:der:1}, \eqref{eq:trouBj} and Lemma~\ref{prop:lemtrou:0} below in \eqref{est:der:0} shows that 
$$ \begin{aligned}
 I^{l}_{\O \backslash B(\xa[m],\ra[m])}(x) &\lesssim \sumi[m+1]  I^{l}_{\O \cap B(\xa,\ra)}(x) + \ve_\a \ns{\phi_\a}\Big(1+\sumi \ta(x)^{-l}\Ba(x) \Big) \\
 &+  \frac{(\ma[m])^{\exc}}{(\ra[m])^{n-2k+l}}\ns{\phi_\a}   
 \end{aligned} $$
for some sequence $\eps_a\to 0$ satisfying \eqref{def:ea}. The proof of \eqref{eq:I3phi} now follows from combining the latter with \eqref{eq:I2phi},  \eqref{eq:I1Il}, \eqref{eq:IplusI} and with the induction assumption \eqref{def:induc}.
 \end{proof}
We now analyse the behavior of $\vp_\alpha$ at distances of order $\ma[m]$ from $\xa[m]$. We introduce for this the following quantities, defined for $\a \ge 1$: 

\begin{itemize}
\item If $\Va[m]$ is an interior bubble: 
\begin{equation}\label{def:psa}
\begin{aligned}
  \hOa & = \frac{1}{\ma[m]}\big(\Oa - \xa[m]\big) \quad  \text{ and, for }   y\in \overline{\hOa}, \\
  \psi_\a(y) & = \frac{(\ma[m])^{\exc}}{\ma[m]\Pa} \phi_\a(\xmm y) 
\end{aligned}
\end{equation}
\item If $\Va[m]$ is a boundary bubble: 
\begin{equation}\label{def:psa:bdr}
\begin{aligned}
  \hOa & = \frac{1}{\ma[m]}\sigma_{\xa[m]}^{-1}\big( \Oa \big) \quad  \text{ and, for } y\in \overline{\hOa}, \\
  \psi_\a(y) & = \frac{(\ma[m])^{\exc}}{\ma[m]\Pa} \phi_\a \big( \sigma_{\xa[m]}(\ma[m] y ) \big). 
\end{aligned}
\end{equation}
\end{itemize}
Recall that $\Oa$ and $\Pa$ are defined in \eqref{def:Oa} and \eqref{def:PHI} and that $\sigma_{\xa[m]}$ is the boundary chart as in Definition \ref{def:bub:bdr}. If $\Va[m]$ is a boundary bubble remark that $\hOa \subset \R^n_+$. The next result quantifies the approximate orthogonality of $\psi_\a$ to the linearised kernel at $v^m$. We recall that the $Z^{mj}$ are the a basis of the kernel $\Ker_{v^m}$ for the $m$-th bubble and are defined in \eqref{ker:interior} and \eqref{ker:boundary}, and that $d_m = \dim \Ker_{v^m}$. 

\begin{lemma}\label{prop:ortpsa}
  There exists a sequence $\eps_{3,\a}\to 0$ as in \eqref{def:ea} such that, for $j=1,\ldots, d_m$ and $\alpha \ge 1$, we have
  \begin{equation}\label{eq:ortpsa}
  \left|  \inty[\hOa]{\langle \nabla^{k}  \psi_\a, \nabla^{k}  Z^{mj}\rangle} \right| \leq  \eps_{3,\a} \frac{\ns{\phi_\a}}{\ma[m]\Pa}.
  \end{equation}
\end{lemma}
\begin{proof}
  By construction $\phi_\a$ is the unique solution of \eqref{eq:lincrit} in $\Ker_\a^\perp$, where orthogonality is taken with respect to the $H^k_0(\O)$ scalar product \eqref{norm:Dk2}. Since by definition of $\Ker_\a$ in \eqref{def:Ka} we have  $\Ker_\a^\perp \sub (\Ker_\a^m)^\perp$ we have in particular   $\phi_\a \in (\Ker_\a^m)^\perp$ so that 
  \begin{equation} \label{eq:ker:1}
  \inty{\langle \nabla^{k}  \phi_\a, \nabla^{k} \Za[mj] \big\rangle} = 0, \qquad j=1,\ldots, d_m
  \end{equation}
holds for all $\a$. Using the definition of $\Vert \cdot \Vert_*$ in \eqref{def:star} and \eqref{eq:IZa} we get 
  \begin{equation} \label{eq:ker:2}
  \begin{aligned}
    & \abs{\inty[\O\setminus \Oa]{\langle \nabla^{k}  \phi_\a, \nabla^{k}  \Za[mj] \rangle}} \\
    &\lesssim \ns{\phi_\a} \Bigg[\inty[\O\setminus\Oa]{\ta[m](y)^{-2k}\pBa[m]^2}\\
     &\quad+ \sum_{\substack{i=0,\ldots,N\\ i\neq m}} \inty{\ta(y)^{-k}\Ba(y)\ta[m](y)^{-k}\Ba[m](y)}\Bigg]\\
     &\leq \eps_\a \ns{\phi_\a}
  \end{aligned}
 \end{equation}
  for some sequence $\eps_\a \to 0$ satisfying \eqref{def:ea}, where by analogy with \eqref{def:Ba} we have let $\theta_\a^{0}(x) \equiv 1$. The last line of \eqref{eq:ker:2} follows from straightforward computations using \eqref{rasurma} and Lemma~\ref{prop:lemBiBj} below. Independently, by \eqref{def:psa} and the definition of $\Za[mj]$ a change of variable together with straightforward computations and the definition of the $\Vert \cdot \Vert_*$ norm gives 
  \begin{equation} \label{eq:ker:3} 
   \begin{aligned}  
    \inty[\Oa]{&\langle \nabla^{k}   \phi_\a, \nabla^{k}  \Za[mj] \rangle} \\
     & = \ma[m] \Pa \inty[\hOa]{\chi \big( \omega_\a \cdot \big) \langle \nabla^{k}   \psi_\a, \nabla^{k}  Z^{mj}  \rangle}     + \bigO(\omega_\a^{n-k} \Vert \vp_\a \Vert_*)\\
     & = \ma[m] \Pa \inty[\hOa]{ \langle \nabla^{k}   \psi_\a, \nabla^{k}  Z^{mj}  \rangle}     + \bigO(\omega_\a^{n-k} \Vert \vp_\a \Vert_*), \\
  \end{aligned}
  \end{equation}
  where we have let 
  $$ \omega_\alpha  = \left \{
  \begin{aligned}
  & \frac{\ma[m]}{\dist{\xa[m], \partial \O}} & \text{ if } \Va[m] \text{ is an interior bubble } \\
  & \ma[m] &  \text{ if } \Va[m] \text{ is a boundary bubble } 
  \end{aligned} 
  \right \} = \smallo(1) $$
as $\a \to + \infty$.  In particular $\omega_\a^{n-k}$ satisfies \eqref{def:ea}. Combining \eqref{eq:ker:1}, \eqref{eq:ker:2} and \eqref{eq:ker:3} proves the Lemma.
\end{proof}

\smallskip

The following result is the key argument in the proof of Proposition~\ref{prop:estIl}: 
  \begin{proposition}\label{prop:estPHI}
    There exists a sequence $\eps_{\a}\to 0$ as in \eqref{def:ea} such that, for any $\alpha \ge 1$ large enough,
    \begin{equation}\label{eq:estPHI}
      \ma[m]\Pa \lesssim \ea \nss{R_\a} +  \eps_{\a}\ns{\phi_\a}.
    \end{equation}
  \end{proposition}
It is easily seen that Proposition~\ref{prop:estIl} follows from this result:
\begin{proof}[End of the proof of Proposition~\ref{prop:estIl} assuming \eqref{eq:estPHI}]
Assume that \eqref{eq:estPHI} holds. Using \eqref{eq:estPHI} in \eqref{eq:I1Il} then shows that for any  $0 \le l \le 2k-1$ and for all $x\in \oO$, we have 
  \begin{equation*} 
    \Ilo{B(\xa[m],\ra[m])}(x) \lesssim \firstTerm 
  \end{equation*}
  for some sequence $(\ve_\a)_\a$ satisfying \eqref{def:ea}. If $m=N$, this simply proves \eqref{eq:estIl}, ie that $P_N$  defined by  \eqref{def:induc} holds true. If $m \leq N-1$, this proves that \eqref{eq:estIl} also holds true for $i=m$ and thus shows that $P_m$ is true. By induction this proves that \eqref{eq:estIl}  holds for every $1 \le i \le N$, which proves Proposition~\ref{prop:estIl}.
  \end{proof}
  
The end of this subsection is devoted to the proof of \eqref{eq:estPHI}. We adapt the original proof of \cite[Lemma 3.7]{Pre24} for $k=1$ (see also \cite{Pre18, Premoselli12}) to the polyharmonic setting and include arguments recently developed in \cite{CarRob25}. 

\begin{proof}[Proof of Lemma \ref{prop:estPHI}]
  We claim that \eqref{eq:estPHI} holds for 
  \begin{equation}\label{tmp:defe0}
    \eps_{\a} = \max\{\eps_{1,\a},\eps_{2,\a},\eps_{3,\a},\frac{\ma[m]}{\ra[m]}\},
  \end{equation}
  where $\eps_{1,\a}$ is as in \eqref{eq:I1Il}, $\eps_{2,\a}$ is as in \eqref{eq:I3phi}, and $\eps_{3,\a}$ is as in \eqref{eq:ortpsa}. It is clear that such a sequence $(\eps_{\a})_\a$ satisfies \eqref{def:ea}. We prove \eqref{eq:estPHI} by contradiction and we assume that, up to passing to a subsequence, we have
  \begin{equation}\label{eq:contpsa}
    \frac{\eta_\a\nss{R_\a}+\eps_{\a}\ns{\phi_\a}}{\ma[m]\Pa} \to 0 \qquad \text{as } \a \to \infty.
  \end{equation}
  Lemma \ref{prop:ctltre}, together with \eqref{eq:I3phi} and \eqref{eq:contpsa} shows that, for $0 \le l  \le 2k-1$ and $x\in \overline{\hOa}$, we have
  \begin{equation}\label{eq:Ipsa}
    \begin{aligned}
    |\nabla^l \psi_\a|(x) 
    &\lesssim  (1+|x|)^{2k-n-l} + \Big(\frac{\ma[m]}{\ra[m]}\Big)^{n-2k+l}\frac{\ns{\phi_\a}}{ \ma[m]\Pa}\\
    &\lesssim (1+|x|)^{2k-n-l} +\smallo(1),
    \end{aligned}
  \end{equation}
  where $\psi_\a$ is as defined in \eqref{def:psa} and where the last line follows from \eqref{tmp:defe0} and \eqref{eq:contpsa}. Lemma \ref{prop:ortpsa} also gives that, for $j=1,\ldots,d_m$,
  \begin{equation}\label{tmp:ortpsa}
    \inty[\hOa]{\big\langle \nabla^k   \psi_\a, \nabla^k Z^{mj} \big\rangle} = \smallo(1) \qquad \text{as } \a \to \infty.
  \end{equation}

\smallskip

  \paragraph{\textbf{Step 1: A limiting equation for $\psi_\a$.}} 
In the following if $f$ denotes a smooth function defined in $\Omega$ we will let, for $x\in \overline{\hOa}$, 
  \begin{equation*}
  \widehat{f}(x)  = \left \{
  \begin{aligned}
  & f (\xmm x) & \text{ if } \Va[m] \text{ is an interior bubble } \\
  & f \big(\sigma_{\xa[m]} ( \ma[m] x) \big) &  \text{ if } \Va[m] \text{ is a boundary bubble } 
  \end{aligned} 
  \right. . 
  \end{equation*}
Simple computations using \eqref{def:La} yield
  \[  \Dd^k \psi_\a(x) + \sum_{l=0}^{p}(-1)^{l} (\ma[m])^{2(k-l)}\nabla^{l} \big(\widehat{A_{l,\a}}(\nabla^l \psi_\a)\big)(x) = \frac{(\ma[m])^{\frac{n+2k}{2}}}{\ma[m]\Pa} \widehat{L_\a\phi_\a}(x)
  \]
for all $x\in \overline{\hOa}$. The latter, with \eqref{eq:lincrit}, shows that $\psi_\a$ satisfies
 \begin{equation} \label{tmp:eqpsa}
    \begin{aligned}
    \Dd^k \psi_\a &+ \sum_{l=0}^{p} (-1)^{l}(\ma[m])^{2(k-l)}\nabla^{l} \big(\widehat{A_{l,\a}}(\nabla^{l} \psi_\a)\big)\\ 
    &=(\crit-1)(\ma[m])^{2k}|\widehat{\Wa}|^{\crit-2}\psi_\a +
    \frac{(\ma[m])^{\frac{n+2k}{2}}}{\ma[m]\Pa}\widehat{R_\a}\\
   &  +\sumij \frac{\la (\ma[m])^{\frac{n+2k}{2}}}{\ma[m]\Pa}(-\D)^k \widehat{\Za}
  \end{aligned} \end{equation} 
  in $\hOa$. 
    With Lemma \ref{prop:ctltre} we have for all $y\in \overline{\hOa}$
  \begin{equation}\label{tmp:e1}
    (\ma[m])^{2k}|\widehat{\Wa}(y)|^{2^\sharp-2} \lesssim \Big[(\ma[m])^{\exc}\widehat{\Ba[m]}(y)\Big]^{2^\sharp-2}\lesssim 1.
  \end{equation}
We claim that the following holds:
    \begin{equation}\label{eq:estRa}
    \frac{(\ma[m])^{\frac{n+2k}{2}}}{\ma[m]\Pa} \Vert \widehat{R_\a} \Vert_{L^\infty(\hOa)} \lesssim \frac{\ea \nss{R_\a}}{\ma[m] \Pa} = \smallo(1)
  \end{equation}
as $\alpha \to + \infty$, where $\eta_\a$ is the explicit sequence given by \eqref{def:ea} and where we used \eqref{eq:contpsa}. We prove \eqref{eq:estRa}. First, by \eqref{def:star}, for any $y \in \overline{\O}$, 
$$|R_\a(y)| \le \nss{R_\a} \Big( \Psi_\a(y) + \eta_\a  \sumi[0]\Ba(y)^{\crit-1} \Big) .$$
On the one hand, by  Lemma \ref{prop:ctltre} we have 
  \begin{equation*}
    (\ma[m])^{\frac{n+2k}{2}}\big(\widehat{\Ba}(y)\big)^{\crit-1} \lesssim (1+|y|)^{-n-2k}
\end{equation*}  
for any $i=1,\ldots N$ and for any $y\in \overline{\hOa}$.  On the other hand, using the explicit expression of $\Psi_\a$ given by \eqref{def:PSI}, we have, for $y \in \overline{\hOa}$,
 $$\begin{aligned}
 & (\ma[m])^{\frac{n+2k}{2}} \widehat{\Psi_\a}(y) \\
 & = \sumi[0] (\ma[m])^{\frac{n+2k}{2}}\widehat{\ta}(y)^{2-2k}\widehat{\Ba}(y) \\
 & +  (\ma[m])^{\frac{n+2k}{2}} \Big[ \sum_{\substack{i,j = 0,\ldots, N\\ i\neq j}} (\widehat{\Ba})^{\crit-2}\widehat{\Ba[j]} \Big](y)  \\
 & \lesssim (\ma[m])^2 (1+|y|)^{-n+2}  + \eta_\a ,
 \end{aligned} $$
 where we again used Lemma \ref{prop:ctltre} to estimate the first sum and Lemma~\ref{lemme:interactions:arbre:bulles} (together with the explicit expression \eqref{def:ea} of $\eta_\a$) to estimate the second one. Combining these estimates proves \eqref{eq:estRa}. Using similar arguments, and thanks to \eqref{eq:Ilama} and \eqref{eq:IDDZa}, we obtain for $i=0,\ldots,N$, $j=1,\ldots,d_i$, that
  \begin{equation}\label{tmp:e2}
    \frac{|\la|(\ma[m])^{\frac{n+2k}{2}}}{\ma[m] \Pa} \left| \left| (-\D)\widehat{\Za} \right| \right|_{L^\infty(\hOa)} \lesssim \frac{\ea\nss{R_\a}}{\ma[m] \Pa} = \smallo(1).
  \end{equation}
  Combining \eqref{tmp:e1}, \eqref{eq:estRa}, and \eqref{tmp:e2} in  \eqref{tmp:eqpsa} shows that 
  \begin{equation}\label{eq:estDpsa}
    \left| \left| \Dd^k \psi_\a + \sum_{l=0}^{p}(-1)^{l}(\ma[m])^{2(k-l)}\nabla^{l}\big(\widehat{A_{l,\a}}(\nabla^l \psi_\a)\big)\right| \right|_{L^\infty(\hOa)} = \bigO(1) 
  \end{equation}
as $\alpha \to + \infty$. We let in what follows
  \[  \sS_m = \left\{\lim_{\a\to \infty}\frac{\xa[j]-\xa[m]}{\ma[m]}\,:\, j\in \sB_m \text{ and } |\xa[j]-\xa[m]| = \bigO(\ma[m])\right\},
  \]
where the limits are taken up to a subsequence. The set $\sS_m$ consists, after rescaling by a factor $\ma[m]$, of the limits of the centers of the highest bubbles in $\sB_m$ that remain at a distance from $\xa[m]$ of order $\ma[m]$. We also define 
 $$\O_\infty^m = \left\{ 
 \begin{aligned}
  & \R^n & \text{ if } \Va[m] \text{ is an interior bubble} \\
  &\R^n_+=\{x_1> 0\} & \text{ if } \Va[m] \text{ is a boundary bubble } \\
  \end{aligned} 
  \right. ,$$
  where we use the terminology of definitions \ref{def:bub:int} and \ref{def:bub:bdr}. Note that if $\Va[m]$ is a boundary bubble some points in $\sS_m$ may belong to $\partial \R^n_+ = \{ x_1 = 0 \}$. Since the operator 
$$\Dd^k + \sum_{l=0}^{p}(-1)^l(\ma[m])^{2(k-l)}\nabla^l \big(\widehat{A_{l,\a}}(\nabla^{l}\cdot\,)\big)$$ 
is elliptic with uniformly bounded coefficients in $\overline{\hOa}$ and since $\psi_\a$ satisfies Dirichlet boundary conditions in $\partial \hOa  \cap \partial \O$ it follows from \eqref{eq:Ipsa}, \eqref{eq:estDpsa} and standard elliptic theory that there exists $\psi_\infty \in D^{k,2}(\O_\infty^m) \cap C^{2k-1}(\overline{\O_\infty^m} \setminus \sS_m)$ such that $\psi_\a \to \psi_\infty$ in $C^{2k-1}_{loc}(\overline{\O_\infty^m} \setminus \sS_m)$ and $\psi_\a \rightharpoonup \psi_\infty$ in $D^{k,2}(\O_\infty^m)$ as $\a \to \infty$, up to a subsequence.\footnote{As a simple remark, note that if $\Va[m]$ is a boundary bubble, $\overline{\hOa}$ is the intersection of $\overline{\R^n_+}$ with an open set of $\R^n$, thus elliptic theory easily applies.} The bound \eqref{eq:Ipsa} then passes to the limit and $\psi_\infty$ satisfies, for $0 \le l \le 2k-1$ and for all $x\in \overline{\O_\infty^m} \setminus \sS_m$,
  \begin{equation}\label{eq:Ipsi}
    |\nabla^l \psi_\infty(x)| \leq C (1+|x|)^{2k-n-l}.
  \end{equation} 
   We now show that $\psi_\infty$ satisfies a limiting linearised equation with Dirichlet boundary conditions in the whole of $\O_\infty^m$. First, using \eqref{eq:estRa} and \eqref{tmp:e2} we have that 
  \begin{equation} \label{eq:passage:limite:1}
  \begin{aligned}
    &\frac{(\ma[m])^{\frac{n+2k}{2}}}{\ma[m]\Pa}\widehat{R_\a} \to 0\\
    &\frac{\la(\ma[m])^{\frac{n+2k}{2}}}{\ma[m] \Pa} (-\D)\widehat{\Za} \to 0 \qquad i=0,\ldots,N,\,j=1\ldots, d_i\\
  \end{aligned}
  \end{equation}
  in $C^0_{loc}(\overline{\O_\infty^m})$. By the definition of $\Wa$ in  \eqref{def:Wa} and using the same arguments that led to \eqref{eq:estRa}, we have 
    \begin{equation} \label{eq:passage:limite:2}
  (\ma[m])^{2k}|\widehat{\Wa}|^{\crit-2}\psi_\a \to |v^m|^{\crit-2}\psi_\infty \quad \text{in } C^0_{loc}(\overline{\O_\infty^m}\setminus\sS_m)
  \end{equation}
  as $\alpha \to + \infty$. Using \eqref{eq:passage:limite:1} and \eqref{eq:passage:limite:2} to pass \eqref{tmp:eqpsa} to the limit we obtain that $\psi_\infty \in C^{2k-1}(\overline{\O_\infty^m} \setminus\sS_m)$ satisfies 
  \begin{equation}\label{tmp:eqpsi}
    \Dd^k \psi_\infty = (\crit-1)|v^m|^{\crit-2}\psi_\infty 
  \end{equation}
in $\O_\infty^m \setminus\sS_m$ in the distributional sense.

\smallskip
  
    \paragraph{\textbf{Step 2: vanishing of $\psi_\infty$.}} 
We now prove that 
 \begin{equation} \label{eq:psi:ker:0}
 \psi_\infty \equiv 0 \quad \text{ in }  \O_\infty^m. 
 \end{equation}
We first claim that 
 \begin{equation} \label{eq:psi:ker:1}
 \psi_\infty \in \Ker_{v^m}.
 \end{equation}
  Assume first that $\O_\infty^m = \R^n$, i.e. that $\Va[m]$ is an interior bubble. Then, by \eqref{eq:Ipsi}, $\nabla^{l} \psi_\infty \in L^\infty(\R^n)$ for every $0 \le l \le 2k-1$. Together with a simple integration by parts argument this now shows that $\psi_\infty$ satisfies \eqref{tmp:eqpsi}, in a distributional sense, in $\O_\infty^m = \R^n$. Since $\psi_\infty \in D^{k,2}(\R^n)$ we therefore have $\psi_\infty \in \Ker_{v^m}$, where $\Ker_{v^m}$ is as in \eqref{ker:interior}. Assume now that $\O_\infty^m = \R^n_+$ and that $\Va[m]$ is a boundary bubble. As in the previous case, using again \eqref{tmp:eqpsi} and \eqref{eq:Ipsi} we obtain that $\psi_\infty$ satisfies \eqref{tmp:eqpsi} in $\R^n_+$. Moreover, $\psi_\infty \in D^{k,2}(\R^n_+)$ and hence 
$\psi_\infty \in \Ker_{v^m}$ as in \eqref{ker:boundary}. This proves \eqref{eq:psi:ker:1}. 
 
 We now claim that, for $j=1,\ldots, d_m$ we have
  \begin{equation}\label{tmp:ortpsi}
    \inty[\O_\infty^m]{\langle\nabla^{k} \psi_\infty,\nabla^{k} Z^{mj}\rangle} = 0.
  \end{equation}  
   To see this, fix $j\in \{1,\ldots,d_m\}$ and let $R>0$. By its definition in \eqref{def:psa} we have 
   $$\hOa \subseteq B\bigg(0, c \frac{\ra[m]}{\ma[m]} \bigg) \cap \O_\infty^m$$
   for some $c >0$. Using \eqref{eq:Ipsa} and Lemma \ref{prop:ctlsollin} we have
  \begin{align*}
    \abs{\inty[\hOa\setminus B(0,R)]{\langle\nabla^{k} \psi_\a,\nabla^{k} Z^{mj}\rangle}} 
    & \lesssim  \intM[y]{\R^n\setminus B(0,R)}{(1+|y|)^{-2n+2k}}  + \frac{\ns{\phi_\a}}{ \ma[m]\Pa} \left(\frac{\ma[m]}{\ra[m]}\right)^{n-k}\\
      &\lesssim R^{2k-n} + \smallo(1)  ,
  \end{align*}
where we used \eqref{eq:contpsa} for the last inequality. We also have  
  \[ \abs{ \inty[\big(\hOa\cap\bigcap_{x^i\in \sS_m}B(x^i,\frac{1}{R})\big)]{\langle\nabla^{k} \psi_\a,\nabla^{k} Z^{mj}\rangle} }  \lesssim  R^{-n} + \smallo(1) 
  \]
as $\a \to + \infty$, where the constants in $\lesssim$ do not depend on $R$. Combining these two estimates with   \eqref{eq:ortpsa}, and using   \eqref{tmp:defe0},  \eqref{eq:contpsa} therefore shows that 
$$ \abs{ \inty[\big(\hOa\setminus\bigcup_{x^i\in \sS_m}B(x^i,\frac{1}{R})\big)\cap B(0,R)]{\langle\nabla^{k} \psi_\a,\nabla^{k} Z^{mj}\rangle}} \lesssim  \smallo(1) + R^{-n}$$
for all $\a \ge 1$ and for $R \ge 1$. Since $\psi_\a$ converges to $\psi_\infty$ in $C^{2k-1}_{loc}( \overline{\O_\infty^m}\setminus\sS_m)$, the previous estimate passes to the limit as $\a \to + \infty$ and we get
$$    \abs{  \inty[\big(\O_\infty^m\setminus\bigcup_{x^i\in \sS_m}B(x^i,\frac{1}{R})\big)\cap B(0,R)]{\langle\nabla^{k} \psi_\infty,\nabla^{k} Z^{mj}\rangle}}  \lesssim R^{-n}.$$
We may now let $R \to + \infty$: using \eqref{eq:Ipsi} and dominated convergence we obtain \eqref{tmp:ortpsi}. Now, \eqref{tmp:ortpsi} shows that  $\psi_\infty \in \Ker_{v^m}^\perp$, where the orthogonal is taken with respect to the $D^{k,2}( \O_\infty^m)$ scalar product \eqref{norm:Dk2}: together with \eqref{eq:psi:ker:1} it thus shows that $    \psi_\infty \equiv 0$ in $\O_\infty^m$, which proves \eqref{eq:psi:ker:0}.

  \smallskip

  \paragraph{\textbf{Step 3: the contradiction.}} For each $\a \ge 1$, we let $z_\a \in \overline{\Oa}$ be a point where $y\mapsto \frac{|\phi_\a(y)|}{\ta[m](y)\Ba[m](y)}$ reaches its maximum in $\overline{\Oa}$, which is by \eqref{def:PHI} equal to $\Pa$. We apply \eqref{eq:I3phi} for $l=0$ at $z_\a$: together with Lemma~\ref{prop:ctltre} and by the contradiction assumption \eqref{eq:contpsa} we obtain that 
  \begin{equation}\label{tmp:e3}
    \Pa = \frac{|\phi_\a(z_\a)|}{\ta[m](z_\a)\Ba[m](z_\a)} \lesssim   \frac{\ma[m]}{\ta[m](z_\a)}\Pa + \frac{(\ma[m])^{\exc}}{(\ra[m])^{n-2k}}\frac{\ns{\phi_\a}}{\ta[m](z_\a) \Ba[m](z_\a)}.
  \end{equation}
 Since $z_\a \in \overline{\Oa[m]} \subseteq \overline{B(\xa[m],\ra[m])}$ we have $|z_\a - \xa[m]| \le \ra[m]$ and thus $\ta[m](z_\a)\Ba[m](z_\a) \gtrsim \frac{(\ma[m])^{\exc}}{(\ra[m])^{n-2k-1}}$. We thus obtain 
 $$ \frac{(\ma[m])^{\exc}}{(\ra[m])^{n-2k}}\frac{\ns{\phi_\a}}{\ta[m](z_\a)\Ba[m](z_\a)} \lesssim  \frac{\ns{\phi_\a}}{\ra[m]}  = \frac{\ma[m]}{\ra[m]} \frac{\ns{\phi_\a}}{\ma[m]} = \smallo(\Pa) $$
as $\a \to + \infty$, where in the last equality we again used \eqref{eq:contpsa}, together with \eqref{tmp:defe0}. With the latter \eqref{tmp:e3} becomes, after dividing by $\Pa$,
  \begin{equation} \label{tmp:zaxaO}
    1 \lesssim \frac{\ma[m]}{\ma[m] + |z_\a-\xa[m]|}+\smallo(1),
  \end{equation}
  which implies that $|z_a-\xa[m]| = \bigO(\ma[m])$ as $\a \to + \infty$. As a technical remark, we mention that the term $\ta[m]$ in the definition of $\Pa$ in \eqref{def:PHI} was precisely added to obtain decay in \eqref{tmp:zaxaO} as $|z_\a-\xa[m]|$ becomes larger than $\ma[m]$. There is no particular reason to choose precisely $\ta[m]$ however, and it is easily seen from \eqref{def:PHI} that we may choose any weight $(\ta[m])^{\kappa}$ in the definition of $\Pa$ provided $0 < \kappa < \min(2,2k,n-2k)$. Let now, for any $\a \ge 1$, 
  $$\hat{z}_\a = \left\{ 
  \begin{aligned}
&    \frac{z_\a-\xa[m]}{\ma[m]}  & \text{ if } \Va[m] \text{ is an interior bubble }  \\
&  \frac{1}{\ma[m]}\sigma_{\xa[m]}^{-1}(z_\a)  & \text{ if } \Va[m] \text{ is a boundary bubble } 
   \end{aligned} \right\} \in \overline{\hOa} .$$ 
 Then $|\hat{z}_\a|  = \bigO(1)$ and, up to a subsequence, we may  let $\hat{z}_\infty = \lim_{\a \to + \infty} \hat{z}_\a \in \O_\infty^m$. By \eqref{def:Ba}, \eqref{def:psa} and by definition of $\psi_\a$ we have 
  \[  1 = \frac{1}{\Pa}\frac{|\phi_\a(z_\a)|}{\ta[m](z_\a)\Ba[m](y)} \lesssim \abs{\psi_\a (\hat{z}_\a) }\big(1 + |\hat{z}_\a| \big)^{n-2k-1},
  \]
so that
\begin{equation} \label{eq:arg:final:contra}
\liminf_{\a\to \infty} \abs{\psi_\a (\hat{z}_\a)} > 0.
\end{equation}
 Assume first that $\hat{z}_\infty \in \overline{\O_\infty^m} \setminus \sS_m$: then, since $\psi_\a$ converges to $\psi_\infty$ in $C^{2k-1}_{loc}(\overline{\O_\infty^m} \setminus \sS_m)$, we get that $|\psi_\infty(\hat{z}_\infty)|>0$, which is a contradiction with \eqref{eq:psi:ker:0}. Assume now that $\hat{z}_\infty\in \sS_m$, that is $\frac{|z_\a-\xa[j]|}{\ma[m]}\to 0$ for some $j\in \sB_m$. By \eqref{eq:arg:final:contra}, and thanks to the bound on the derivatives of $\psi_\a$ in \eqref{eq:Ipsa}, we can let $(\hat{w}_\a)_\a$ be another sequence of points in $\overline{\hOa}$ which converges to some point $w_\infty \in \overline{\O_\infty^m} \setminus \sS_m$ arbitrarily close to $\hat{z}_\infty$ and which still satisfies $\liminf_{\a\to\infty}|\psi_\a(\hat{w}_\a)| >0$. Using again the $C^{2k-1}_{loc}(\overline{\hOa} \setminus \sS_m)$ convergence we again obtain $|\psi_\infty(w_\infty)| > 0$, which again contradicts \eqref{eq:psi:ker:0}. This concludes the proof of Proposition \ref{prop:estPHI}.
\end{proof}

As we already explained, Proposition~\ref{prop:estPHI} implies Proposition~\ref{prop:estIl} which in turn implies Theorem \ref{prop:ptinv}. The proof of Theorem \ref{prop:ptinv} is thus complete.

\section{Global pointwise estimates for the blow-up} \label{sec:nonlinear}

In this section we prove global pointwise estimates for blowing-up solutions of \eqref{eq:main} with bounded energy. We obtain a sharp asymptotic description of the blow-up which will be the main ingredient in the proof of Theorem~\ref{theo:compactness}. The main result of this section is as follows (we keep the notations from Sections \ref{sec:bulles}, \ref{sec:bulles2} and \ref{sec:lin}): 

\begin{theorem}\label{prop:main}
  Let $\O\sub \R^n$ be a smooth bounded domain and $k\geq 1$ an integer such that $n>2k$. Let $(L_\a)_\a$ be a sequence of polyharmonic operators as in \eqref{def:La} and let $(u_\a)_\a$ be a sequence of solutions of \eqref{eq:main} which satisfies $\limsup_{\a \to + \infty} \Vert u_\a \Vert_{H^k_0(\O)} < + \infty$. Up to passing to a subsequence, there exists $u_0 \in C^{2k}(\oO)$ solving \eqref{eq:lmmain}, there exists an integer $N \ge0$ and there exist $N$ bubbles $(\Va)_{1 \le i \le N}$, defined by triplets $[(\xa)_\a,(\ma)_\a, v^i]$ as in definitions \ref{def:bub:int} and \ref{def:bub:bdr} and satisfying \eqref{eq:struc}, such that for any $0 \le l  \le 2k-1$,
  \begin{equation}\label{eq:main:ctldua}    
    \lim_{\a \to + \infty} \left| \left| \bfrac{ \nabla^{l} \Big( u_\a -  u_0 - \sumi  \Va \Big)}{1 +\sumi (\ta)^{-l}  \Ba} \right| \right|_{C^0(\overline{\O})} = 0.
      \end{equation}
\end{theorem}
The functions $\Ba$ and $\ta$, for $0 \le i \le N$, are as in \eqref{def:Ba}. In the statement of Theorem~\ref{prop:main} $u_0$ is the weak limit of $(u_\a)_\a$ in $H^k_0(\O)$ and, if $N=0$, $u_\a$ strongly converges to $u_0$ in $C^{2k}(\oO)$. An equivalent formulation of \eqref{eq:main:ctldua} is that there is a sequence of positive numbers $(\ve_\a)$ with $\ve_\a \to 0$ as $\a \to + \infty$ such that for any $\a \ge 1$ 
$$ \Big| \nabla^{l} u_\a(x) -  \nabla^{l} u_0(x) - \sumi  \nabla^{l} \Va(x) \Big| \le \ve_\a \Big( 1 + \ta(x)^{-l}\Ba(x) \Big) \quad \text{ for any } x \in \overline{\O}. $$
Estimates such as \eqref{eq:main:ctldua} have been the subject of intense investigation in the last twenty years, especially because of their importance in the proof of compactness results for equations like \eqref{eq:main}. The novelty of Theorem \ref{prop:main} is that it  holds true in the general polyharmonic setting $k \ge 1$, for sign-changing solutions, and that the estimates in \eqref{eq:main:ctldua} are true uniformly up to $\partial \O$. This last fact is new even when $k=1$. Theorem \ref{prop:main} holds true under very general assumptions: we do not assume anything but \eqref{eq:struc} on the bubbles $(\Va)_{\{1 \le i \le N\}}$: in particular the profiles $v^i$ need not be non-degenerate in the sense of \cite{DuyckaertsKenigMerle}. When $k \ge 2$ we allow boundary bubbles to arise in \eqref{eq:main:ctldua} (although their existence is still unknown) and we allow interior concentration points $\xa$ to converge to $\partial \O$ as $\a \to + \infty$, at any speed. When $k=1$ estimates like \eqref{eq:main:ctldua} were first obtained for positive solutions of critical Yamabe-type equations on compact manifolds without boundary in \cite{LiZhu, DruetJDG, DruetYlowdim, Heb14} and for positive solutions of equations like \eqref{eq:main} in domains of $\R^n$ in \cite{DruetLaurain, LaurainKonig2, LaurainKonig1, LiSn1, LiSn2}. For sign-changing solutions the analogue of Theorem~\ref{prop:main}, still when $k=1$, was first obtained on closed manifolds in \cite{Pre24}, and in domains of $\R^n$ for solutions of least-energy in \cite{CAP1, CAP2}. For $k \ge 2$ and sign-changing solutions, an analogue of Theorem~\ref{prop:main} was obtained in \cite{Robert:localisation:bubbling, Robert:C0:ordre:k}, again on closed manifolds. All these results were applied to obtain compactness results, for positive solutions \cite{DruetJDG, LiSn1, LiSn2} and for sign-changing ones \cite{PremoselliRobert, PremoselliVetois3,  PremoselliVetois2, PremoselliVetois4}. 

\smallskip

We briefly describe the strategy of proof of Theorem \ref{prop:main}. Recall that by \cite[Theorem 1.2]{Maz17} (see \eqref{eq:decHk}) a sequence of solutions $u_\a$ as in Theorem~\ref{prop:main} splits as 
  \begin{equation*} 
   u_\a = u_0 + \sumi \Va + \smallo(1) \qquad \text{in } \Sob_0(\O).
  \end{equation*} 
Theorem~\ref{prop:main} therefore states that this energy decomposition can be improved into global pointwise estimates: it can be seen as a quantitative version of Struwe decomposition in strong spaces. To prove it we will show that $u_\a$ is uniquely obtained as a perturbation of the bubble-tree $\Wa$ \emph{via} a nonlinear fixed-point argument in $(C^{2k-1}(\overline{\O}), \Vert \cdot \Vert_{*})$. We prove this by linearising and using the sharp linear estimates of Theorem \ref{prop:ptinv}. We will then show that  $u_\a$ inherits the estimates of Theorem \ref{prop:ptinv}, which will yield \eqref{eq:main:ctldua}. Here again we follow the strategy of proof of \cite{Pre24}. 

\subsection{Nonlinear procedure} The first result that we state claims that any bubble-tree as in \eqref{def:Wa} can be perturbed in $H^k_0(\O)$ into a solution of \eqref{eq:main} up to an obstruction term in $\Ker_\a$:
\begin{proposition}\label{prop:nlinHk}
Let $u_0\in C^{2k}(\oO)\cap\Sob_0(\O)$ be a solution of \eqref{eq:lmmain}, let $N\geq 1$ be an integer and let $[(\xa)_\a,(\ma)_\a,v^i]$, $i=1,\ldots, N$ be $N$ bubbles $\Va$ defined as in Definitions \ref{def:bub:int} and \ref{def:bub:bdr}, satisfying \eqref{eq:struc}. For $0 \le i \le N$ and $1 \le j \le d_i$, where $d_i = \dim \Ker^i_\a$, let $(\nu^{ij}_\a)_\a$ be any sequence of real numbers such that $\nu^{ij}_\a \to 0$, and define $\Wa$ as in \eqref{def:Wa}. There exists $\delta>0$ such that, up to a subsequence, the following holds: there exists a unique function
  \[  \phi_\a \in \Ker_\a^\perp \cap \{\phi\in \Sob_0(\O)\,:\, \Vert \phi \Vert_{H^k_0(\Omega)}\leq \delta\}
  \] 
  that satisfies
  \begin{multline}\label{eq:nlinHk}  
    \Pi_{\Ker_\a^\perp}\Big(\Wa + \phi_\a + (-\D)^{-k} \Big[\sum_{l=0}^{p} (-1)^l\,\nabla^l \big(A_{l,\a}(\nabla^l (\Wa + \phi_\a))\big)\\ -\big|\Wa+\phi_\a\big|^{\crit-2}\big(\Wa+\phi_\a\big)\Big]\Big) = 0,
  \end{multline}
\end{proposition}   
Recall that $\Pi_{\Ker_\a^\perp}$ is the orthogonal projection onto $\Ker_\a^\perp$ in $\Sob_0(\O)$ for the scalar product \eqref{norm:Dk2}, that $\Ker_\a$ is defined in \eqref{def:Ka}, and  that $(-\D)^{-k}$ is as in Remark \ref{rk:Dinv}. The proof of Proposition \ref{prop:nlinHk} is standard and follows from a \emph{verbatim} adaptation of the arguments in  \cite{EspositoPistoiaVetois} and \cite{RobertVetois} for $k=1$ and \cite{CarJLMS} (with $N=1$) for $k\geq 2$. Proposition \ref{prop:nlinHk} provides us with a local uniqueness result in $H^k_0(\O)$ around $\Wa$ for solutions of \eqref{eq:main}. In the next result, which is the main one in this section, we prove that $\Wa$ can also be uniquely perturbed into a solution of \eqref{eq:main} in $(C^{2k-1}(\overline{\O}), \Vert \cdot \Vert_{*})$, where $\ns{\cdot\,}$ is as in \eqref{def:star}, up to an obstruction in $\Ker_\a$:

\begin{proposition}\label{prop:nlinpt}
 Let $u_0\in C^{2k}(\oO)\cap\Sob_0(\O)$ be a solution of \eqref{eq:lmmain}, let $N\geq 1$ be an integer and let $[(\xa)_\a,(\ma)_\a,v^i]$, $i=1,\ldots, N$ be $N$ bubbles $\Va$ defined as in Definitions \ref{def:bub:int} and \ref{def:bub:bdr}, satisfying \eqref{eq:struc}. For $0 \le i \le N$ and $1 \le j \le d_i$, where $d_i = \dim \Ker^i_\a$, let $(\nu^{ij}_\a)_\a$ be any sequence of real numbers such that $\nu^{ij}_\a \to 0$, and define $\Wa$ as in \eqref{def:Wa}. There exists a sequence $\eta_\a \to0$ such that, up to passing to a subsequence in $\a$, the following holds: there is a unique function 
  \[ \phi_\a \in \Ker_\a^\perp \cap \Big\{\phi\in C^{2k-1}(\oO) \,:\, \ns{\phi_\a} \leq \eta_\a\Big\},
  \]   
which satisfies 
  \begin{equation}\label{eq:nlin}
    L_\a\big(\Wa + \phi_\a\big) = \big|\Wa + \phi_\a\big|^{\crit-2}\big(\Wa+\phi_\a\big) + \sumij \la (-\D)^{k}\Za \quad \text{ in } \O
  \end{equation}
  for some $(\la)_\a$,$i=0,\ldots N$, $j=1,\ldots, d_i$, uniquely determined by $\Wa$. Moreover, this solution satisfies 
  \begin{equation}\label{eq:nlin:to0}
   \Vert \phi_\a \Vert_{H^k_0(\O)} \to 0 \qquad \text{as } \a \to \infty.
  \end{equation}
\end{proposition}
Since $\Ker_\a \subset H^k_0(\O)$, $\Wa + \vp_\a$ satisfies Dirichlet boundary conditions. We prove Proposition \ref{prop:nlinpt} by adapting the proofs in \cite[Proposition 4.2]{Pre24} (for $k=1$) and \cite[Proposition 4.11]{CarJLMS} (for a single bubble and $k \ge 2$).

\begin{proof}
  Let $\phi \in C^0(\oO)$. We define 
  \begin{equation*}
    \begin{aligned}
      E_\a & = L_\a \Wa  - |\Wa|^{\crit-2}\Wa \quad \text{ and } \\
      N_\a(\phi) & = \big|\Wa + \phi\big|^{\crit-2}\big(\Wa+\phi\big) - |\Wa|^{\crit-2}\Wa - (\crit-1)|\Wa|^{\crit-2}\phi.
    \end{aligned}
  \end{equation*}
  Since $u_0$ satisfies \eqref{eq:lmmain} we have $u_0\in C^{2k}(\oO)$  by standard elliptic theory. Straightforward computations using \eqref{def:La} then show that 
  \begin{equation}\label{eq:eta4}
  \begin{aligned}
    |E_\a|  & \lesssim \sumi (\ta)^{2-2k}\Ba + \sum_{\substack{i,j=0,\ldots,N\\i\neq j}}(\Ba[j])^{\crit-2}\Ba\\ 
    &+ \Big(\nu_\a + \sum_{l=0}^{p} \|A_{l,\a}-A_l\|_{C^{l}(\oO)}\Big) \sumi[0](\Ba)^{\crit-1}
\end{aligned}
  \end{equation}
   for all $y \in \oO$, where we have let $\nu_\a = \max_{\substack{i=0,\ldots, N\\j=1,\ldots, d_i}}|\nu_\a^{ij}|$. 
  Hence, there exists $C_1 >0$ independent of $\a$ such that 
  \begin{equation} \label{est:norme:Ea}
    \nss{E_\a}\leq C_1,
  \end{equation}
where $\ea$ is the sequence defined in \eqref{def:eta} and where $\nss{\cdot\,}$ is as in \eqref{def:star}. We let, for each $\a \ge 1$, 
  \[  \sC_\a = \Ker_\a^\perp \cap \Big\{\phi \in C^{2k-1}(\oO)\,:\, \ns{\phi}\leq 2C_0C_1\ea\Big\},
  \]
  where $C_0>0$ is given by Theorem \ref{prop:ptinv} and $C_1>0$ by \eqref{est:norme:Ea}. It is easily seen that, for all $\a$, $(\sC_\a,\ns{\cdot})$ is a non-empty complete metric space.
  For $\phi \in \sC_\a$ and $\a \ge 1$ we also define 
  \[  R_\a(\phi) = -E_\a + N_\a(\phi),
  \]
  and we let $T_\a(\phi)\in \Ker_\a^\perp \cap C^{2k-1}(\oO)$ be the unique solution, given by Proposition \ref{prop:lininv}, to
  \begin{equation}\label{eq:Tphi}
    L_\a \big( T_\a(\phi) \big) - (\crit-1)|\Wa|^{\crit-2}T_\a(\phi)= R_\a(\phi)+\sumij \la (-\D)^k \Za,
  \end{equation}
where the real numbers $(\la)_\a$, $i=0,\ldots,N$, $j=1,\ldots, d_i$ are uniquely given by Proposition \ref{prop:lininv}. We will now show that $T_\a : (\sC_\a,\ns{\cdot}) \to (\sC_\a,\ns{\cdot})$ is a contraction provided $\a$ is large enough.

\smallskip

  \paragraph{\textbf{Step 1: $T_\a$ stabilizes $\sC_\a$.}} Let $\phi\in\sC_\a$. We have, for all $y\in \oO$,
  \begin{equation*}
    |N_\a(\phi)(y)| \lesssim |\phi(y)|^{\crit-1}+ \sumi[0]\pBa^{\crit-2-\tau}|\phi(y)|^{1+\tau} 
  \end{equation*}
   for some $0<\tau<\min\{1,\crit-2\}$. By definition of $\Vert \cdot \Vert_*$, since $\phi\in\sC_\a$ and since $\eta_\a \to 0$ as $\a \to + \infty$, the previous inequality ensures that we have 
    \begin{align*}
    |N_\a(\phi)(y)| \lesssim \ea^{1+\tau}\sumi[0]\pBa^{\crit-1} 
  \end{align*}
for any $y \in \oO$, for $\a$ large enough. Together with \eqref{est:norme:Ea} this shows that there exists $C_2>0$ such that
  \begin{equation}\label{tmp:InssRa}
    \nss{R_\a(\phi)} \leq C_1 + C_2\ea^{\tau}.
  \end{equation}
We may now apply  Theorem \ref{prop:ptinv} to equation \eqref{eq:Tphi}: we obtain that, up to a subsequence, 
  \[  \ns{T_\a(\phi)} \leq C_0 \ea\big(C_1 + C_2\ea^{\tau}\big).
  \]
  Since $\ea\to 0$, we have $C_2\ea^{\tau} < C_1$ for $\a$ large enough, and hence $T_\a(\phi) \in  \sC_\a$.
  
  \smallskip
  
  \paragraph{\textbf{Step 2: $T_\a$ is a contraction in $\sC_\a$.}} We now let $\phi_1,\phi_2 \in \sC_\a$. By linearity, $T_\a(\phi_1)-T_\a(\phi_2)$ satisfies
  \begin{equation} \label{eq:contraction:Na}
  \begin{aligned}  
  L_\a \big(T_\a(\phi_1)-T_\a(\phi_2)\big) -(\crit-1)|\Wa|^{\crit-2}\big(T_\a(\phi_1)-T_\a(\phi_2)\big)\\
     = \big(N_\a(\phi_1)-N_\a(\phi_2)\big) +\sumij \Tilde{\lambda}_\a^{ij}(-\D)^{k}\Za
  \end{aligned}
  \end{equation}  
  for some real numbers $\Tilde{\lambda}_\a^{ij}$, $i=0,\ldots, N$, $j=1,\ldots, d_i$. Direct computations and the definition of $\sC_\a$ give 
  \begin{align*}
    \big|N_\a(\phi_1)-N_\a(\phi_2)\big|&\lesssim |\phi_1-\phi_2|\sumi[0](\Ba)^{\crit-2-\tau}\big(|\phi_1|^\tau+|\phi_2|^\tau\big) \\
      &\lesssim \ns{\phi_1-\phi_2} \ea^{\tau} \sumi[0](\Ba)^{\crit-1}
  \end{align*}
  in $\oO$. It thus follows that there exists $C_3>0$ such that $\nss{N_\a(\phi_1)-N_\a(\phi_2)} \leq C_3 \ea^{\tau-1} \ns{\phi_1-\phi_2}$, and therefore Theorem \ref{prop:ptinv} applied to \eqref{eq:contraction:Na} shows that 
  \[  \ns{T_\a(\phi_1)-T_\a(\phi_2)} \leq C_0C_3\ea^{\tau}\ns{\phi_1-\phi_2}.
  \]
 Since $\ea\to 0$, we have $C_0C_3\ea^{\tau}\leq \frac{1}{2}$ provided $\a$ is large enough, and $T_\a$ is a contraction in $\sC_\a$.
 
 \smallskip
 
  \paragraph{\textbf{Conclusion:}} We can now apply Banach's fixed point Theorem to $T_\a$ in $(\sC_\a, \ns{\cdot})$: we obtain that there exists a unique $\phi_\a \in \sC_\a$ such that $T_\a(\phi_\a)= \phi_\a$. With \eqref{eq:Tphi}, $\phi_\a$ then solves 
  \[  L_\a \phi_\a -(\crit-1)|\Wa|^{\crit-2}\phi_\a = R_\a(\phi_\a) + \sumij \la(-\D)^k \Za
  \]
  which is equivalent to \eqref{eq:nlin}. Finally, \eqref{eq:nlin:to0} simply follows from the condition $\Vert \vp_\a \Vert_* \lesssim \eta_\a$ and the fact that $\eta_\a \to 0$ as $\a \to + \infty$.  This concludes the proof of Proposition \ref{prop:nlinpt}. 
\end{proof}

\subsection{Proof of Theorem \ref{prop:main}} 

We are now in position to prove Theorem \ref{prop:main}: 
\begin{proof}[Proof of Theorem \ref{prop:main}]
Let $(L_\a)_\a$ be a sequence of polyharmonic operators as in \eqref{def:La} and let $(u_\a)_\a$ be a sequence of solutions \eqref{eq:main} which satisfies 
$$\limsup_{\a \to + \infty} \Vert u_\a \Vert_{H^k_0(\O)} < + \infty.$$
 We let $u_0$ be the weak limit of $u_\a$ in $H^k_0(\O)$, which  solves  \eqref{eq:lmmain}. Assume first that $\Vert u_\a \Vert_{L^\infty(\O)}$ is uniformly bounded as $\a \to + \infty$: in this case standard elliptic theory shows that $u_\a \to u_0$ in $C^{2k-1}(\oO)$, and \eqref{eq:main:ctldua} follows with $N=0$. 

We may thus assume that $u_\a$ blows up, that it that $\Vert u_\a \Vert_{L^\infty(\O)} \to + \infty$ up to a subsequence. By \eqref{def:La}, it is easily seen that $(u_\a)_\a$ is a Palais-Smale sequence for the functional $I$ on $\Sob_0(\O)$ defined in \eqref{def:funcI}. The result of \cite{Maz17} then shows that there exists an integer $N \ge 1$ and $N$ bubbles $\Va$ as in Definitions \ref{def:bub:int} or \ref{def:bub:bdr}, represented by triplets $[(\xa)_\a,(\ma)_\a, v^i]$ and satisfying \eqref{eq:struc}, such that 
  \[  u_\a = u_0 + \sumi \Va + \phi_\a,
  \]
  where $\phi_\a \to 0$ in $\Sob_0(\O)$. We define
  \begin{equation*}
    \phi_\a^\perp = \Pi_{\Ker_\a^\perp}(\phi_\a) \quad \text{and} \quad  \phi_\a^{\parallel} = \Pi_{\Ker_\a}(\phi_\a),
  \end{equation*}
where the projection is taken for the $H^k_0(\O)$ scalar product \eqref{norm:Dk2}, and where $\Ker_\a$ is as in \eqref{def:Ka}. By definition, there exist $(\nu_\a^{ij})_\a$, $i=0,\ldots, N$, $j=1,\ldots, d_i$ such that 
  \[  \phi_\a^\parallel = \sumij \nu_\a^{ij} \Za, \quad \text{ with } \sumij |\nu_\a^{ij}|\lesssim \Vert \phi_\a \Vert_{H^k_0(\O)} \to 0.
  \]
We may thus define a bubble-tree $\Wa$ associated to $u_\a$ as follows: 
  \[  \Wa = u_0 + \sumi \Va + \sumij \nu_\a^{ij} \Za,
  \]
  so that the above decomposition rewrites as 
  \[  u_\a = \Wa + \phi_\a^\perp \quad \text{ with } \quad \phi_\a^\perp \in \Ker_\a^{\perp}.
  \]
$\Wa$ as above satisfies the assumptions of Propositions \ref{prop:nlinHk} and \ref{prop:nlinpt}. In particular, Proposition \ref{prop:nlinpt} ensures, up to a subsequence, the existence of a unique function 
  \[ \hat{\phi}_\a \in \Ker_\a^\perp \cap \{\phi \in C^{2k-1}(\oO)\,:\, \ns{\phi_\a} \leq \eta_\a\}
  \]
that solves \eqref{eq:nlinHk} and satisfies $\Vert \hat{\phi}_\a \Vert_{H^k_0(\O)}\to 0$   for some sequence of positive numbers $\eta_\a \to 0$. By construction $u_\a = \Wa + \phi_\a^\perp$ with  $\phi_\a^\perp \in \Ker_\a^\perp$ and $\Vert \phi^\perp_\a \Vert_{H^k_0(\O)}\to 0$, and $\phi_\a^{\perp}$ satisfies \eqref{eq:nlinHk} since $\Wa + \phi_\a$ satisfies \eqref{eq:main}. We may thus apply the uniqueness result in Proposition \ref{prop:nlinHk} which shows, up to a subsequence, that $\phi_\a^\perp = \hat{\phi}_\a$. Thus our initial solution $u_\a$ actually satisfies $u_\a = \Wa + \hat{\phi}_\a$, and by definition of $\hat{\phi}_\a$ we have $\Vert u_\a - \Wa \Vert_* = \Vert \hat{\phi}_\a \Vert_*  \le \eta_\a \to 0$ as $\a \to + \infty$. Coming back to the definition of $\Wa$, using  \eqref{eq:IZa} and since $\nu_\a^{ij}  \to 0$ as $\a \to + \infty$ for every $i,j$ we finally get 
  \[  \left\ns{u_\a - u_0 - \sumi \Va\right} \to 0 \qquad \text{as }\a \to \infty
  \]
  up to a subsequence, which is exactly \eqref{eq:main:ctldua}. 
\end{proof}

\subsection{Generalisations of Theorem \ref{prop:main}}

The strategy of proof of Theorem \ref{prop:main} was initially developed in \cite{Pre24} for critical second-order equations on closed manifolds without boundary and we generalise it in this paper to polyharmonic critical equations on domains of $\R^n$ with Dirichlet boundary conditions. The strategy of proof that we follow here is flexible and allows for further generalisations. If one decides to complement the equation $Lu = |u|^{2^\sharp-2}u$ in $\Omega$ (or its perturbations $ L_\a u_\a= |u_\a|^{\crit-2} u_\a$) with boundary conditions that are not Dirichlet --- Navier or Steklov, for instance ---, one can still expect to prove global pointwise estimates on blowing-up solutions and to obtain an analogue of Theorem \ref{prop:main} to this new setting. The final argument in the proof of  Theorem \ref{prop:main} (the nonlinear procedure of Proposition \ref{prop:nlinpt}) is indeed very general and reduces the proof of Theorem \ref{prop:main} to proving the linear invertibility of Theorem \ref{prop:ptinv}. A careful inspection of the proof of Theorem \ref{prop:ptinv} shows that it relies on the following three fundamental ingredients: 
\begin{itemize}
\item The validity of a Struwe-type decomposition such as \eqref{eq:decHk} for energy-bounded sequences of solutions to the problem, measured with respect to a Hilbertian norm
\item Sharp pointwise decay for the interior and boundary bubbles that may arise in the analysis, that is for the finite-energy solutions of the limiting problems \eqref{eq:critRn} and \eqref{eq:crithf} (Proposition \ref{prop:ctlsol} in this paper)
\item Pointwise estimates on the Green's function of the operator $L- (2^\sharp-1) |u_0|^{2^\sharp-2}$ with the chosen boundary conditions (Dirichlet for us here, see e.g. the proof of \eqref{eq:I2phi})
\end{itemize} 
 Provided each of the three points above is satisfied for the new boundary condition considered, it is likely that a step-by-step adaptation of the proof of Theorem \ref{prop:ptinv} will yield a similar result in this new setting. Verifying each one of these three points will of course require an \emph{ad hoc} analysis. Similarly, since the proof of Theorem \ref{theo:compactness} relies on a Pohozaev argument, it is likely that it will adapt to other boundary conditions to yield new compactness results for different boundary conditions. 
Struwe-type decompositions, for instance, are known for biharmonic problems ($k=2$) with Navier boundary conditions by \cite[Lemma 7.74]{GazGruSw10}, and it is therefore likely that Theorem \ref{theo:compactness} possesses a close analogue to the Navier setting, at least when $k=2$. We do not pursue this further in this paper.

\section{Proof of Theorem \ref{theo:compactness}} \label{sec:pohozaev}

In this section we prove Theorem~\ref{theo:compactness}. We rely on the pointwise description given by Theorem~\ref{prop:main} to estimate the contribution of each bubble to the blow-up by means of suitable Pohozaev identities and to obtain a contradiction.

\begin{proof}[Proof of Theorem~\ref{theo:compactness}]
We let $0 \le p \le k-1$ be a fixed integer and $\bold{A}_* \in \mathcal{A}_p$ as in \eqref{def:A}, and we recall that $\bold{A}_* = (A_0, \cdots, A_p)$. We assume that $\bold{A}_*$ satisfies \eqref{cond:signchanging} and that $n > 6k-4p$. We define $L$ as in \eqref{def:L}. We first claim that there exist positive real numbers $\delta = \delta (n, \Lambda, \bold{A}_*)$ and $C = C(n,\Lambda, \bold{A}_*)$ such that 
\begin{equation} \label{contra:finale:0}
 \sup_{\Vert \bold{A}_* - \bold{A} \Vert_{p}< \delta} \sup_{u \in \mathcal{S}_{\bold{A}, \Lambda}} \Vert u \Vert_{C^{0}(\overline{\O})} \le C. 
 \end{equation}
We proceed by contradiction and we assume that there is a sequence $(\bold{A}_\a)_\a \in \mathcal{A}_p$ such that $\Vert \bold{A}_\a - \bold{A}_* \Vert_{\mathcal{A}_p} \to 0$ as $\a \to + \infty$ and that there is a sequence $(u_\a)_\a$ of solutions of \eqref{eq:main} with $L_\a$ given by \eqref{def:La} that satisfies $\Vert u_\a \Vert_{H^k_0(\O)} \le \Lambda$ such that, up to passing to a subsequence,
\begin{equation} \label{contra:finale} 
 \Vert u_\a \Vert_{L^\infty(\O)} \to + \infty \quad \text{ as } \a \to + \infty.
 \end{equation}
The result of \cite{Maz17} again shows that there exist a solution $u_0$ of \eqref{eq:lmmain}, an integer $N \ge 1$ and $N$ bubbles $\Va$ as in Definitions \ref{def:bub:int} or \ref{def:bub:bdr}, represented by triplets $[(\xa)_\a,(\ma)_\a, v^i]$ and satisfying \eqref{eq:struc}, such that \eqref{eq:decHk} holds. In particular, Theorem~\ref{prop:main} applies and shows that \eqref{eq:main:ctldua} holds true. Following \eqref{eq:renum}, up to passing to a subsequence and up to renumbering the bubbles, we  assume that 
\begin{equation}\label{eq:renum:comp}
    \ma[N] \leq \ldots\leq \ma[1]< 1
\end{equation}
for all $\a$. We will perform an asymptotic analysis near $\xa[1]$, that is the center of the less concentrated bubble.  Throughout this proof we will fix a radius $\rho >0$ chosen so that, up to a subsequence, the following holds:
\ben \label{def:Oa:comp}
\O_\a  = \O \cap B \big(\xa[1], \rho \sqrt{\ma[1]} \big)
\een
 is a Lipschitz domain for all $\a \ge 1$, and for any $i \in \{2, \cdots, N \}$ we have 
\ben  \label{def:rho:comp}
\begin{aligned}
 \text{ either } \quad   |\xa - \xa[1]| = \smallo (\sqrt{\ma[1]} ) \quad  \text{ or }  \quad  |\xa - \xa[1]| \ge 2 \rho \sqrt{\ma[1]} . \\
\end{aligned}
\een 
Such a radius $\rho$ is easily seen to exist. In this proof we will also denote by $(\xi_\a)_\a$ a sequence of points in $\R^n$, which will be chosen later (and may not belong to $\Omega$), and which satisfies 
\ben \label{def:xi:comp}
|\xi_\a - \xa[1]| \le 2 \rho \sqrt{\ma[1]}.
\een
We now apply the Pohozaev identity of Proposition ~\ref{prop.poho} below (equation  \eqref{poho.id0}), with $f \equiv 1, p = 2^\sharp$ and $\xi = \xi_\a$ in $\O_\a$: we obtain that 
\begin{equation} \label{poho.id0:comp}
\begin{aligned}
\mathcal{P}_{k}(\O_\a;u_\a) & =\int_{\O_\a}\left(\frac{n-2k}{2}u_\a+(x-\xi_\a)^{a}\partial_{a}u_\a\right)\mathcal{E}(u_\a)\,dx \\
& +\frac{n-2k}{2n}\int_{\partial \O_\a} (x-\xi_\a, \nu)|u_\a|^{2^\sharp}\,d\sigma ,\\
\end{aligned}
\end{equation}
where $\mathcal{E}(u_\a) =  (-\Delta)^k u_\a-|u_\a|^{2^\sharp-2}u_\a$ and where $\mathcal{P}_{k}(\O_\a;u_\a) $ is given by \eqref{poho.id01} below. In \eqref{poho.id0:comp} we used Einstein's convention according to which repeated and raised indices are summed over, so that $(x-\xi_\a)^{a}\partial_{a}u_\a = \sum_{a=1}^n (x- \xi_\a)_a \partial_a u_\a$. We estimate each term in \eqref{poho.id0:comp} separately to obtain a contradiction. We first prove the following technical result: 

\begin{lemma} \label{lem:calculs:integrales}
\textbullet Assume that $k \ge 1$ and let $0 \le l \le k-1$. Then, as $\a \to + \infty$
  $$ \bal
  \int_{\O_\a}   |\nabla^l u_\a|^2 \,dx= \bigO \big( (\ma[1])^{2(k-l)} \big) + \smallo \big( (\ma[1])^{\frac{n-2k}{2}}\big). 
 \eal $$ 
\textbullet Assume that $k=1$. Then as $\a \to + \infty$
  $$ \bal
  \int_{\O_\a}  |x - \xi_\a| |u_\a| |\nabla u_\a| \,dx= \bigO( (\ma[1])^{\frac32}). 
 \eal $$ 
\end{lemma}
\begin{proof}
Let us first observe that using \eqref{eq:main:ctldua} we have, for any $x \in \overline{\O}$ and $0 \le l \le 2k-1$,
\ben \label{est:int:comp:5}
 |\nabla^l u_\a(x)| \lesssim 1 + \sum_{i=1}^N \ta(x)^{-l} \Ba(x). 
 \een
Assume first that $k \ge 1$ and $0 \le l \le p-1$. As a consequence of \eqref{def:Oa:comp}, \eqref{est:int:comp:5} and with \eqref{def:Ba} we have
$$\int_{\O_\a}   |\nabla^l u_\a|^2 \,dx \lesssim (\ma[1])^{\frac{n}{2}} + \sum_{i=1}^N \int_{\O_\a}  \Big( \ta(x)^{-l} \Ba(x) \Big)^2 dx. $$
 Let $i \in \{2, \cdots, N \}$ be fixed. We use \eqref{def:rho:comp}: if $|\xa - \xa[1]| \ge 2 \rho \sqrt{\ma[1]} $ then $\ta(x) \gtrsim \sqrt{\ma[1]}$ for any $x \in \O_\a$, and thus 
 \ben \label{calc:bulle:1}
 \bal 
\int_{\O_\a}  \Big( \ta(x)^{-l} \Ba(x) \Big)^2 dx &\lesssim  (\ma)^{n-2k} (\ma[1])^{-\frac{n}{2} + 2k - l} & \lesssim (\ma[1])^{\frac{n}{2}-l }, 
 \eal 
 \een
  where we used \eqref{eq:renum:comp} for the last inequality. If now  $|\xa - \xa[1]| = \smallo (\sqrt{\ma[1]} )$, by \eqref{def:Oa:comp} we have  $\O_\a \subset \O \cap B \big(\xa, 2\rho \sqrt{\ma[1]} \big)$ for $\a$ large enough and direct computations using \eqref{eq:renum:comp} show that 
   \ben \label{calc:bulle:2}
\int_{\O_\a}  \Big( \ta(x)^{-l} \Ba(x) \Big)^2 dx\lesssim
 \left \{
 \bal 
 & (\ma)^{2k - 2l} & \text{ if } n > 4k - 2l \\
 & (\ma)^{2k - 2l} \log \frac{1}{\ma} & \text{ if } n = 4k - 2l \\
& (\ma)^{n-2k} (\ma[1])^{-\frac{n-2k}{2} + k- l} & \text{ if } n < 4k - 2l \\
 \eal \right \}. 
 \een
Thus we obtain, using again \eqref{eq:renum:comp}, that 
  \begin{equation*}   
  \bal \int_{\O_\a}   |\nabla^l u_\a|^2 \,dx & \lesssim
 \left \{
 \bal 
 & (\ma[1])^{2k - 2l} & \text{ if } n > 4k - 2l \\
 & (\ma[1])^{2k - 2l} \log \frac{1}{\ma[1]} & \text{ if } n = 4k - 2l \\
& (\ma[1])^{\frac{n}{2}- l} & \text{ if } n < 4k - 2l \\
 \eal \right \} +  (\ma[1])^{\frac{n}{2}-l} \\
 &  = \bigO \big( (\ma[1])^{2(k-l)} \big) + \smallo\big( (\ma[1])^{\frac{n-2k}{2}}\big),
 \eal 
 \end{equation*}
where in the last equality we observed that, since $n >2k$ and $0 \le l \le k-1$, we have $\frac{n}{2}-l = \frac{n}{2}-k + k-l > \frac{n-2k}{2}$. This proves the first part of Lemma~\ref{lem:calculs:integrales}.

We now prove the second part. Arguing as in \eqref{est:int:comp:5} we find that 
$$ u_\a(x) |\nabla u_\a(x)| \lesssim 1 + \sum_{i=1}^N \ta(x)^{-1} \Ba(x) \quad \text{ for } x \in \O. $$ 
The proof  then follows from similar computations than in \eqref{calc:bulle:1} and \eqref{calc:bulle:2} and from the observation that, by \eqref{def:xi:comp}, we have $ |x - \xi_\a| \lesssim \sqrt{\ma[1]}$ for every $x \in \O_\a$. 
 \end{proof}

We next estimate the bulk integral in \eqref{poho.id0:comp}: 

\begin{lemma}
We have, as $\a \to + \infty$:
 \ben \label{est:int:comp:6}
 \bal
& \int_{\O_\a}\left(\frac{n-2k}{2}u_\a+(x-\xi_\a)^{a}\partial_{a}u_\a\right)\mathcal{E}(u_\a)\,dx \\ &=  \int_{\O_\a} A_{p} (\xa[1])(\nabla^{p} u_\a, \nabla^{p} u_\a ) dx  
+  \smallo\big( (\ma[1])^{2(k-p)}\big)+ \smallo \big( (\ma[1])^{\frac{n-2k}{2}} \big),
\eal
 \een
 where we recall that $p$ is as in \eqref{def:L}. 
\end{lemma}

\begin{proof}
Using \eqref{def:La} and \eqref{eq:main} we have 
$$ \mathcal{E}(u_\a) = - \sum_{l=0}^{p}(-1)^l \nabla^l \big(A_{l,\a}(\nabla^l u_\a \,)\big) \quad \text{ in } \O, $$
so that 
\ben \label{est:int:comp}
 \bal
 &\int_{\O_\a}\left(\frac{n-2k}{2}u_\a+(x-\xi_\a)^{a}\partial_{a}u_\a\right)\mathcal{E}(u_\a)\,dx \\
 & =  \sum_{l=0}^{p}\int_{\O_\a} (-1)^{l+1} \left(\frac{n-2k}{2}u_\a+(x-\xi_\a)^{a}\partial_{a}u_\a\right)\nabla^{l} \big(A_{l,\a}(\nabla^{l} u_\a \,)\big). \\
 \eal 
 \een

Let $l \in \{0, \cdots, p\}$ be fixed.  Assume first that $k \ge 2$ and $0 \le l \le k-1$. We have 
 $$ \nabla^{l} \big( (x-\xi_\a)^a \partial_a u_\a \big) = (x-\xi_\a)^a \partial_a \nabla^{l} u_\a + l \nabla^{l} u_\a. $$
We now use the latter to integrate the right-hand side of \eqref{est:int:comp} by parts $l$ times. By definition of $\O_\a$ in \eqref{def:Oa:comp}, when integrating by parts two boundary integrals arise: one over $\partial \O \cap B \big(\xa[1], \rho \sqrt{\ma[1]} \big)$ and one over $\O \cap \partial B \big(\xa[1], \rho \sqrt{\ma[1]} \big)$. The boundary integral in $\partial \O \cap B \big(\xa[1], \rho \sqrt{\ma[1]} \big)$ is zero since $u_\a$ satisfies Dirichlet boundary conditions and hence $u_\a$ and $|\nabla u_\a|$ vanish in $\partial \O$ (since $k \ge 2$). To estimate the boundary integrals in $\O \cap \partial B \big(\xa[1], \rho \sqrt{\ma[1]} \big)$ we observe that by \eqref{def:rho:comp}, for any $x \in \O \cap \partial B(\xa[1], \rho \sqrt{\ma[1]})$ and any bubble $i \in \{2, \cdots, N\}$ we have $|\xa - x| \gtrsim \sqrt{\ma[1]}$. The latter with \eqref{eq:renum:comp} and \eqref{est:int:comp:5} then shows that, for $0 \le l \le 2k-1$,
\ben \label{est:int:comp:2}
 |\nabla^l u_\a(x)| \lesssim (\ma[1])^{- \frac{l}{2}} \quad \text{ for any } x \in \O \cap \partial B(\xa[1], \rho \sqrt{\ma[1]}).
 \een
 With these considerations, and since by \eqref{def:xi:comp} we have $|\xi_\a - x| \lesssim \sqrt{\ma[1]}$ for any $x \in \O \cap \partial B(\xa[1], \rho \sqrt{\ma[1]})$, we obtain 
 \ben \label{est:int:comp:3}
 \bal
& \int_{\O_\a} \left(\frac{n-2k}{2}u_\a+(x-\xi_\a)^{a}\partial_{a}u_\a\right)\nabla^{l} \big(A_{l,\a}(\nabla^{l} u_\a \,)\big) dx \\
 & = \int_{\O_\a} (-1)^{l} A_{l,\a}\left( \Big(\frac{n}{2} - k + l \Big)\nabla^l u_\a+(x-\xi_\a)^{a}\partial_{a} \nabla^l u_\a, \nabla^{l} u_\a \right) dx \\
 & + \smallo( (\ma[1])^{\frac{n-2k}{2}}),
 \eal 
 \een
 where we again used that $(\ma[1])^{\frac{n}{2}-l} = \smallo \big( (\ma[1])^{\frac{n-2k}{2}}\big)$ since $l \le k-1$.  Remark also that expression \eqref{est:int:comp:3} remains true even in the case $k=1, l=0$, where it is tautological. Recall that $A_{l,\a}$ is a symmetric continuous $(2l,0)$-tensor in $\overline{\O}$. As a consequence,
 $$ \bal
 A_{l,\a}\big( (x-\xi_\a)^{a}\partial_{a} \nabla^l u_\a, \nabla^{l} u_\a \big) & =  (x-\xi_\a)^{a}A_{l,\a}\big( \partial_{a} \nabla^l u_\a, \nabla^{l} u_\a \big) \\
 &= (x-\xi_\a)^{a}A_{l,\a}\big( \nabla^l u_\a, \partial_{a} \nabla^l u_\a \big).
 \eal$$ 
 Assume first that $l \ge 1$ (this only occurs if $p \ge 1$ and hence $k \ge 2$). Since $A_{l, \a}$ is of class $C^1$ in $\overline{\O}$ we may write that 
 $$ \bal 
 \partial_{a} \Big[A_{l,\a}\big( \nabla^l u_\a , \nabla^{l} u_\a \big)\Big]  &= 2 A_{l,\a}\big( \partial_{a} \nabla^l u_\a, \nabla^{l} u_\a \big) + \big( \partial_a A_{l,\a} \big)\big(\nabla^l u_\a , \nabla^{l} u_\a \big)
\eal  $$
everywhere in $\O$. Integrating the latter expressions by parts, using \eqref{def:xi:comp} and  Lemma ~\ref{lem:calculs:integrales}
thus shows that, for any $l \ge 1$,
\ben \label{est:int:comp:4}
\bal
& \int_{\O_\a} A_{l,\a}\left( (x-\xi_\a)^{a}\partial_{a} \nabla^l u_\a, \nabla^{l} u_\a \right) dx \\
& = -\frac{n}{2} \int_{\O_\a} A_{l,\a}\big( \nabla^l u_\a , \nabla^{l} u_\a \big)  dx + \smallo \big( (\ma[1])^{2(k-p)} \big) + \smallo\big( (\ma[1])^{\frac{n-2k}{2}}\big).\\
\eal 
\een
Remark that to integrate by parts in the first equality of \eqref{est:int:comp:4} we again used \eqref{est:int:comp:2}, the Dirichlet boundary conditions in $\O$ and the assumption $l \le k-1$. 
 
  Assume now that $l=0$. Since $A_{0,\a}$ is uniformly bounded in $C^1(\overline{\O})$ we may write that in $\O_\a$ we have
 $$ \bal
 A_{0,\a}\big( u_\a, \partial_{a} u_\a \big) & =  A_{0,\a}(\xa[1])\big( u_\a, \partial_{a} u_\a \big) + \smallo \big( \sqrt{\ma[1]} u_\a |\nabla u_\a| \big) \\
 & = \frac12 \partial_{a} \Big[A_{0,\a}(\xa[1])\big(u_\a ,  u_\a \big)\Big] + \smallo \big( \sqrt{\ma[1]}u_\a |\nabla u_\a| \big). 
 \eal $$ 
Here again, integrating the latter by parts, using the strong convergence of $A_{0,\alpha}$ towards $A_0$ and using Lemma ~\ref{lem:calculs:integrales} yields 
\ben \label{est:int:comp:4bis}
\bal
& \int_{\O_\a} A_{0,\a}\left( (x-\xi_\a)^{a}\partial_{a} u_\a, u_\a \right) dx \\
& = -\frac{n}{2} \int_{\O_\a} A_{0, \a}(\xa[1])\big( u_\a ,  u_\a \big)  dx  + \smallo( (\ma[1])^2) + \smallo( (\ma[1])^{\frac{n-2k}{2}}).\\
\eal 
\een
If $0 \le l \le k-1$, independently, we have by the strong convergence of $A_{l,\alpha}$ towards $A_l$ and Lemma ~\ref{lem:calculs:integrales}:
\begin{equation} \label{est:int:comp:4bisbis}
\begin{aligned}
\int_{\O_\a} A_{l,\a}\left(\nabla^l u_\a, \nabla^{l} u_\a \right) dx & = \int_{\O_\a} A_{l}(\xa[1])\left(\nabla^l u_\a, \nabla^{l} u_\a \right) dx \\
&+ \smallo \big( (\ma[1])^{2(k-l)} \big)+ \smallo( (\ma[1])^{\frac{n-2k}{2}}).
\end{aligned}
\end{equation}
Combining \eqref{est:int:comp}, \eqref{est:int:comp:3}, \eqref{est:int:comp:4}, \eqref{est:int:comp:4bis} and \eqref{est:int:comp:4bisbis} finally proves \eqref{est:int:comp:6}.
\end{proof}
We now estimate the main integral term appearing in \eqref{est:int:comp:6}: 

\begin{lemma}
There exists a nonzero constant $C_0$ such that, as $\a \to + \infty$,
\ben \label{est:int:comp:8}
\int_{\O_\a} A_{p}(\xa[1]) (\nabla^{p} u_\a, \nabla^{p} u_\a ) dx =\big( C_0 + \smallo(1) \big) (\ma[1])^{2(k-p)}.
\een
\end{lemma}

\begin{proof}
We use \eqref{eq:main:ctldua} to write that, pointwise in $\overline{\O}$,
$$ \bal
 A_{p}(\xa[1]) (\nabla^{p} u_\a, \nabla^{p} u_\a )  &= A_{p}(\xa[1]) \Big(\sum_{i=1}^N \nabla^{p} \Va, \sum_{i=1}^N \nabla^{p} \Va \Big)\\
 & + \bigO(1) + \smallo \left( \sum_{i=1}^N \big( (\ta)^{-p} \Ba \big)^2 \right). 
 \eal $$
Integrating the latter and using \eqref{calc:bulle:1} and \eqref{calc:bulle:2} thus gives
\ben \label{int:L2:bulle:0}
 \bal 
  &\int_{\O_\a} A_{p}(\xa[1]) (\nabla^{p} u_\a, \nabla^{p} u_\a ) dx  \\
 & = \sum_{i,j=1}^N  \int_{\O_\a} A_{p}(\xa[1]) (\nabla^{p} \Va, \nabla^{p} \Va[j] ) dx +  \smallo \big( (\ma[1])^{2(k-p)} \big),
 \eal \een
 where we used that $n>6k - 4p > 4k-2p$ by assumption.  Using \eqref{est:bubble} and Lemma \ref{prop:lemBiBj} below we have, for $i \neq j$, 
$$\int_{\O_\a} A_{p} (\xa[1])(\nabla^{p} \Va, \nabla^{p} \Va[j] ) dx = \smallo \big( (\ma \ma[j])^{k-p} \big) =   \smallo( (\ma[1])^{2(k-p)}).$$
Let now $i \in \{1, \cdots, N\}$ be fixed. If $|\xa - \xa[1]| \ge 2 \rho \sqrt{\ma[1]} $ then \eqref{calc:bulle:1} shows that 
\ben \label{int:L2:bulle} \int_{\O_\a} A_{p} (\xa[1])(\nabla^{p} \Va, \nabla^{p} \Va[i] ) dx  = \bigO \left((\ma[1])^{\frac{n}{2}-l } \right) =  \smallo( (\ma[1])^{2(k-p)}). 
\een
If now  $|\xa - \xa[1]| = \smallo (\sqrt{\ma[1]} )$, and since $\ma \le \ma[1]$, by \eqref{def:Oa:comp} we have  $\O \cap B \big(\xa, \frac12 \rho \sqrt{\ma} \big) \subset \O_\a \subset \O \cap B \big(\xa, 2 \rho \sqrt{\ma[1]} \big)$ for $\a$ large enough. Let $R >0$ be fixed. Using the definition of $\Va$ (see Definitions  \ref{def:bub:int} or \ref{def:bub:bdr}) and \eqref{est:bubble}, straightforward computations with a change of variables give, since $n \ge 6k-4p > 4k-2p$,
$$ \bal 
&  \int_{\O_\a} A_{p} (\xa[1])(\nabla^{p} \Va, \nabla^{p} \Va[i] ) dx  \\
& = (\ma)^{2(k-p)} \Bigg[ \int_{B(0,R) \cap \hat{\O}_\a^i} A_{p}(\xa[1]) (\nabla^{p} v^i, \nabla^{p} v^i ) dx  + \bigO(R^{2k+2-n}) + \smallo(1) \Bigg] 
\eal $$
where $\hat{\O}_\a^i$ is as in \eqref{def:psa} or \eqref{def:psa:bdr}, depending on whether $\Va$ is an interior or a boundary bubble. Passing to a subsequence as $\a \to + \infty$ we have thus proven that for any $i \in \{1, \cdots, N\}$ satisfying $|\xa - \xa[1]| = \smallo (\sqrt{\ma[1]} )$ we have 
\begin{equation} \label{int:L2:bulle:2}
 \int_{\O_\a} A_{p}(\xa[1]) (\nabla^{p} \Va, \nabla^{p} \Va[i] ) dx = C_i (\ma)^{2(k-p)} +  \smallo( (\ma)^{2(k-p)}), 
\end{equation}
where we have let 
\ben \label{comp:def:Ci} C_i = \left \{ 
\bal 
& \int_{\R^n} A_{p}(x^1_\infty) (\nabla^{p} v^i, \nabla^{p} v^i ) dx & \text{ if } \Va \text{ is an interior bubble }, \\
& \int_{\R^n_+} A_{p}(x^1_\infty) (\nabla^{p} v^i, \nabla^{p} v^i ) dx & \text{ if } \Va \text{ is a boundary bubble }, \\
\eal \right. \een
and where, up to a subsequence, $x^1_\infty = \lim_{\a \to + \infty} \xa[1] \in \overline{\O}$. Note that the second case only occurs if $x_1 \in \partial \O$. Up to a subsequence we now let 
$$\mathcal{S} = \Big \{ i \in \{1, \cdots, N\}, |\xa-\xa[1]| = \smallo( \sqrt{\ma[1]}) \quad \text{ and } \quad \lim_{\a \to + \infty} \frac{\ma}{\ma[1]} >0\Big \},$$
and for each $i \in \mathcal{S}$ we let $\mu_i = \lim_{\a \to + \infty} \frac{\ma}{\ma[1]} >0$. Combining \eqref{int:L2:bulle} and \eqref{int:L2:bulle:2} in \eqref{int:L2:bulle:0} finally proves \eqref{est:int:comp:8} with 
\ben \label{comp:def:C0}
 C_0 =   \sum_{i \in \mathcal{S}} C_i (\mu_i)^{2(k-p)}.
\een 
It follows from our main assumption \eqref{cond:signchanging} that $C_i \neq 0$ for every $i \in \{1, \cdots, N\}$, and that all the constants $C_i$ have the same sign. Since $1 \in \mathcal{S}$ we obtain that $C_0 \neq 0$.
\end{proof}
 
 To conclude the proof of Theorem \ref{theo:compactness} we now estimate the boundary terms in \eqref{poho.id0:comp}. First, using \eqref{def:xi:comp}, \eqref{est:int:comp:2},  and since $u_\a$ satisfies Dirichlet boundary conditions in $\partial \O$ we have 
 \ben \label{est:bord:comp:0}
 \bal
  & \frac{n-2k}{2n}\int_{\partial \O_\a} (x-\xi_\a, \nu)|u_\a|^{2^\sharp}\,d\sigma \\
  & =  \frac{n-2k}{2n}\int_{\O \cap \partial B \big(\xa[1], \rho \sqrt{\ma[1]} \big)} (x-\xi_\a, \nu)|u_\a|^{2^\sharp} 
   \,d\sigma
  =  \smallo \big( (\ma[1])^{2(k-p)} \big) . 
\eal 
\een
Combining \eqref{est:int:comp:6}, \eqref{est:int:comp:8} and \eqref{est:bord:comp:0} in \eqref{poho.id0:comp} shows that 
 \ben \label{est:int:comp:8bis}
\mathcal{P}_{k}(\O_\a;u_\a)= \big( C_0+ \smallo(1) \big) (\ma[1])^{2(k-p)} +  \smallo \big( (\ma[1])^{\frac{n-2k}{2}}\big)
\een
as $\a \to + \infty$, where $C_0$ is as in \eqref{comp:def:C0} and is nonzero. We now estimate $\mathcal{P}_{k}(\O_\a;u_\a)$, which is given by \eqref{poho.id01} below. We again have 
 $$\partial \O_\a = \Big(\partial \O \cap B \big(\xa[1], \rho \sqrt{\ma[1]} \big) \Big) \bigcup \Big(\O \cap \partial B \big(\xa[1], \rho \sqrt{\ma[1]} \big)\Big).$$
On the one hand, the contribution over $\O \cap \partial B \big(\xa[1], \rho \sqrt{\ma[1]} \big)$ in $\mathcal{P}_{k}(\O_\a;u_\a)$ is again computed using  \eqref{def:xi:comp} and \eqref{est:int:comp:2} and contributes as $\bigO \big( (\ma[1])^{\frac{n-2k}{2}}\big)$. On the other hand, since $u_\a$ satisfies Dirichlet boundary conditions it is easily seen using \eqref{poho.id01} and \eqref{poho:k:odd} that the contribution over $\partial \O \cap B \big(\xa[1], \rho \sqrt{\ma[1]} \big)$ in $\mathcal{P}_{k}(\O_\a;u_\a)$ reduces to a single integral that depends on the parity of $k$ and which is given by 
\ben \label{est:bord:comp:11}
\bal
\mathcal{P}_{k}(\O_\a;u_\a)  &=  \frac{(-1)^k}{2}\int_{\partial \O \cap B \big(\xa[1], \rho \sqrt{\ma[1]} \big) } (x-\xi_\a, \nu) \big| (-\Delta)^{k/2}u_\a \big|^{2}d\sigma  \\
& + \bigO \big( (\ma[1])^{\frac{n-2k}{2}}\big).\\
\eal
\een
If $\dist{\xa[1], \partial \O} > \rho \sqrt{\ma[1]}$ we have $\partial \O \cap B \big(\xa[1], \rho \sqrt{\ma[1]} \big)  = \emptyset$ and thus 
\ben \label{est:bord:comp:1}
\mathcal{P}_{k}(\O_\a;u_\a) = \bigO \left( (\ma[1])^{\frac{n-2k}{2}} \right)
\een
as $\a \to + \infty$.  We may thus assume from now on that $\dist{\xa[1], \partial \O} \le \rho \sqrt{\ma[1]}$. We claim that the following holds: 
\begin{lemma}
Assume that $\dist{\xa[1], \partial \O} \le \rho \sqrt{\ma[1]}$. We may choose $\xi_\a$ in \eqref{def:xi:comp} so that for every $x \in \partial \O \cap B \big(\xa[1], \rho \sqrt{\ma[1]} \big)$, 
\ben  \label{est:bord:comp:2}
(-1)^k C_0 \big(x-\xi_\a, \nu(x) \big)\le 0.
\een
\end{lemma}

\begin{proof}
Since $\dist{\xa[1], \partial \O} \to 0$ as $\a \to + \infty$ we may let $\pi(\xa[1])$ be the unique closest point to $\xa[1]$ in $\partial \O$. Up to a rotation we may assume that the tangent space to $\partial \O$ at $\pi(\xa[1])$ is $ \{0\}  \times \R^{n-1}$ and that there exists $\delta >0$ such that 
$$
\bal 
\partial \O \cap B(\pi(\xa[1]), \delta) &= \pi(\xa[1]) +\big \{  x_1 = h_\a(x') \big \} \cap B(\pi(\xa[1]), \delta) \quad \text{ and } \\
 \O \cap B(\pi(\xa[1]), \delta) &= \pi(\xa[1]) +\big \{  x_1 > h_\a(x') \big \} \cap B(\pi(\xa[1]), \delta), 
\eal  $$
where we have let $x' = (x_2, \cdots, x_{n})$ and where $h_\a$ is a smooth function defined in a neighbourhood of $0$ in $\R^{n-1}$, which satisfies 
$$ h_\a(x') = \sum_{i=1}^{n-1} \kappa_{i,\a} (x'_i)^2 + \bigO( |x'|^3) \quad \text{ as } x' \to 0. $$
The numbers $\kappa_{i,\a}$ are the principal curvatures of $\partial \O$ at $\pi(\xa[1])$ and the constant in the $\bigO( |x'|^3)$ does not depend on $\a$ since $\partial \O$ is compact. 
The exterior unit normal is given by 
$$ \nu \big( h_\a(x'), x' \big) = \frac{1}{\sqrt{1 + |\nabla h_\a(x')|^2}} \big(  -1, \nabla h_\a(x')\big), $$
and straightforward computations show that for $x = (h_\a(x'), x') \in \partial \O$ with $|x - \pi(\xa[1])| \le \delta$ we have 
$$ \big(x-\pi(\xa[1]), \nu(x) \big) =  \bigO \big( |x'|^2\big). $$
For any $x \in \partial \O \cap B \big(\xa[1], \rho \sqrt{\ma[1]} \big)$ we have $|x - \xa[1]| \lesssim \sqrt{\ma[1]}$ and $| \nu(x) - \nu(\pi(\xa[1]))| \lesssim |x'| \lesssim \sqrt{\ma[1]}$. By \eqref{def:xi:comp} and since $\dist{\xa[1], \partial \O} \le \rho \sqrt{\ma[1]}$  we have in addition $|\xi_\a - \pi(\xa[1])| \lesssim \sqrt{\ma[1]}$. Combining the latter estimates we thus have
\ben \label{est:bord:comp:3} \begin{aligned}
\big(x-\xi_\a, \nu(x) \big) & =  \big(x - \pi(\xa[1]) , \nu(x) \big) +  \big(\pi(\xa[1])-\xi_\a, \nu(\pi(\xa[1])) \big) \\
&+  \big(\pi(\xa[1])-\xi_\a, \nu(x) - \nu(\pi(\xa[1])) \big)  \\
& =  \big(\pi(\xa[1])-\xi_\a, \nu(\pi(\xa[1])) \big) +  \bigO \left( \ma[1] \right).
\end{aligned} 
\een
We may therefore choose 
$$ \xi_\a = \pi(\xa[1]) + \varepsilon \rho \sqrt{\ma[1]} \nu \big( \pi(\xa[1]) \big), $$
where $\varepsilon = (-1)^k \text{sign}(C_0)$. It satisfies \eqref{def:xi:comp}, and \eqref{est:bord:comp:2} now follows from \eqref{est:bord:comp:3}. 
\end{proof}

We are now in position to conclude the proof of Theorem~\ref{theo:compactness}. Multiplying \eqref{est:int:comp:8bis} by $C_0$ and using \eqref{est:bord:comp:11}, \eqref{est:bord:comp:1} and \eqref{est:bord:comp:2} shows that, for $\a$ large enough,
$$ C_0^2 (\ma[1])^{2(k-p)} +  \smallo \big( (\ma[1])^{2(k-p)} \big) \lesssim (\ma[1])^{\frac{n-2k}{2}}.$$
Since $n > 6k-4p$ we have $\frac{n-2k}{2} > 2(k-p)$, and the latter becomes 
$$ C_0^2 (\ma[1])^{2(k-p)} +  \smallo \big( (\ma[1])^{2(k-p)} \big) \le 0, $$
as $\a \to + \infty$. Since $C_0 \neq 0$ this is an obvious contradiction and shows that \eqref{contra:finale} cannot hold. 

We have thus proven that \eqref{contra:finale:0} holds true. Standard elliptic theory then improves \eqref{contra:finale:0} into global $C^{2k}$ bounds on $\overline{\Omega}$ and concludes the proof of Theorem~\ref{theo:compactness}. 
\end{proof}

For simplicity we stated Theorem \ref{theo:compactness} by requiring $C^1(\overline{\O})$ bounds for $A_0$ in \eqref{def:A}. The only point in the proof where this appears is \eqref{est:int:comp:4bis} and it is easily seen that $C^{0,\eta}(\overline{\O})$ bounds for some $\frac12 < \eta <1$ would be enough.

\appendix

\section{Polyharmonic Pohozaev identities}

We state in this section a polyharmonic version of the celebrated Pohozaev identity \cite{poho}. The following result is proven in \cite[Proposition A.1]{MazumdarPremoselli} (see also \cite[Proposition 13.1]{Robert:localisation:bubbling}) and generalises similar results in \cite[Proposition 2.2]{LiXiong} (in the case $k=2$):
\begin{proposition}[Proposition A.1 in \cite{MazumdarPremoselli}] \label{prop.poho}
Let $\O$ be a bounded smooth domain of $\R^n, n >2k$ and $p\geq2$. Let $u\in C^{2k}(\overline{\O})$, $f\in C^{1}(\overline{\O})$ and let $\mathcal{E}(u) = (-\Delta)^k u-f|u|^{p-2}u$. Let $\xi \in \R^n$ be fixed. Then we have
\begin{equation} \label{poho.id0}
\begin{aligned}
&\mathcal{P}_{k}(\O;u) =\int_{\O}\left(\frac{n-2k}{2}u+(x-\xi)^{i}\partial_{i}u\right)\mathcal{E}(u)\,dx +\frac{1}{p}\int_{\partial \O}(x-\xi, \nu)f|u|^{p}\,d\sigma \\
&~+\left(\frac{n-2k}{2}-\frac{n}{p}\right) \int_{\O}f\,|u|^{p}\,dx-\frac{1}{p}\int_{\O}(x-\xi)^{i}\partial_{i}f\,|u|^{p}\,dx.\\
\end{aligned}
\end{equation}
Here $\mathcal{P}_{k}(\O;u)$ denotes boundary terms whose expression is given by: 
\begin{equation} \label{poho.id01}
\begin{aligned} 
&\mathcal{P}_{k}(\O;u)  = \mathcal{R}_{k}(\O;u) \\
&+\frac{n-2k}{2}\sum_{i=0}^{[k/2]-1} \int_{\partial \O}\bigg(\partial_{\nu}((-\Delta)^{i}u)\,(-\Delta)^{k-i-1}u\,-(-\Delta)^{i}u~\partial_{\nu}((-\Delta)^{k-i-1}u)\bigg)\,d\sigma \\
&+\sum_{i=0}^{[k/2]-1} \int_{\partial \O}\bigg[\partial_{\nu}((-\Delta)^{i}((x-\xi)^{a}\partial_{a}u))\,(-\Delta)^{k-i-1}u\, \\
& -(-\Delta)^{i}((x-\xi)^{a}\partial_{a}u)\,\partial_{\nu}((-\Delta)^{k-i-1}u)\bigg]\,d\sigma, \\
\end{aligned}
\end{equation}
where we have let
\begin{equation*}
\mathcal{R}_{k}(\O;u)=  \frac{1}{2}\int_{\partial \O} (x-\xi, \nu) \big| (-\Delta)^{\frac{k}{2}}u \big|^{2}d\sigma  \quad  \text{ if } k \text{ is even } \\
\end{equation*}
and 
\begin{equation*} 
\begin{aligned}
  \mathcal{R}_{k}(\O;u) & =  \frac{1}{2}\int_{\partial \O} (x-\xi, \nu)(-\Delta)^{\frac{k+1}{2}}u (-\Delta)^{\frac{k-1}{2}}u\,d\sigma \\
& + \frac{1}{2}\int_{\partial \O}\bigg[(-\Delta)^{\frac{k-1}{2}}u\,\partial_{\nu}((x-\xi)^{a}\partial_{a}((-\Delta)^{\frac{k-1}{2}}u))\,\\
&\quad \quad -\big((x-\xi)^{a}\partial_{a}((-\Delta)^{\frac{k-1}{2}}u)\big)\partial_{\nu} \big((-\Delta)^{\frac{k-1}{2}}u\big)\bigg]\,d\sigma \\
& \quad\quad\quad \quad\quad\quad\quad\quad\quad\quad \quad\quad\quad\quad\quad\quad\quad\quad\text{ if } k \text{ is odd. } 
\end{aligned}  
\end{equation*}
In the previous expressions, $[x]$ denotes the integer part of $x \in \mathbb{R}$.
\end{proposition}
A simple observation is that, under the assumptions of Proposition~\ref{prop.poho}, identity \eqref{poho.id0} remains true when $\O$ is replaced by any subdomain $U \subseteq \O$ with piecewise smooth boundary. This is e.g. the case when $U$ is the transverse intersection of $\O$ and of an open ball contained in $\O$. If $u$ in the statement of Proposition \ref{prop.poho} satisfies in addition Dirichlet boundary conditions in $\partial \O$, i.e. if $|\nabla^l u | = 0$ in $\partial \O$ for every $0 \le l \le k-1$, then $ \mathcal{R}_{k}(\O;u)$ greatly simplifies in the case where $k$ is odd. Indeed,  $(-\Delta)^{\frac{k-1}{2}}u \equiv 0$ in $\partial \O$, and as a consequence, $(-\Delta)^{\frac{k}{2}}u  = \nabla (-\Delta)^{\frac{k-1}{2}}u = \partial_\nu (-\Delta)^{\frac{k-1}{2}}u  \cdot \nu$ in $\partial \O$. We then obtain, for any $k$, 
\ben \label{poho:k:odd}
\mathcal{P}_{k}(\O;u) =  \mathcal{R}_{k}(\O;u) =  \frac{(-1)^k}{2} \int_{\partial \O} (x-\xi, \nu) \big|  (-\Delta)^{\frac{k}{2}}u \big|^{2}d\sigma . 
 \een
 A similar formula was proven in \cite[Theorem 7.27]{GazGruSw10}, but in this reference the $(-1)^k$ is missing.

\section{A proof of \eqref{eq:struc}  } \label{structurerelation}

Let $(u_\a)_{\a}$ be a Palais-Smale sequence for the functional $I$ defined in \eqref{def:funcI}. The global compactness result of \cite{Maz17} states that there exist $u_0 \in \Sob_0(\O)$ a solution to \eqref{eq:lmmain}, an integer $N\geq 0$ and $N$ bubbles $\Va$ defined as in Definitions \ref{def:bub:int} and \ref{def:bub:bdr} by their triples $[(\xa)_\a,(\ma)_\a,v^i]$, $i=1,\ldots, N$ such that up to a subsequence 
  \begin{equation} \label{eq:decHk2} 
   u_\a = u_0 + \sumi \Va + \smallo(1) \qquad \text{in } \Sob_0(\O)
  \end{equation} 
    and 
  \begin{equation}\label{eq:decHk:bis2} 
  I(u_\a) = I(u_0) + \sumi I_0(v^i) + \smallo(1)
  \end{equation} 
as $\a \to + \infty$, where $I_0$ is as in \eqref{def:I0}. It is stated in \cite[Chapter III, Remark 3.2]{StruweVariationalMethods} (for $k=1$) and in \cite{Maz17} (for $k \ge 2$) that the bubbles $[(\xa)_\a,(\ma)_\a,v^i]$, $i=1,\ldots, N$ can in addition be assumed to satisfy the following structure relation: 
 \begin{equation}\label{eq:struc2}
    \eps_\a^{ij} = \frac{|\xa-\xa[j]|^2}{\ma\ma[j]} + \frac{\ma}{\ma[j]} + \frac{\ma[j]}{\ma}\to \infty \quad \text{as } \a \to \infty,
  \end{equation} 
  for all $i\neq j$ in $\{1,\ldots,N\}$. Relation \eqref{eq:struc2} states the the bubbles $\Va$ are weakly-interacting in $H^k_0(\O)$ and is proven when $u_\a$ is nonnegative, at least when $k=1$, e.g. in \cite{BahriCoron} and \cite[Chapter 3]{Heb14}. But, to the best of our knowledge, no proof for sign-changing solutions is available in the litterature. In this section we provide a simple proof that \eqref{eq:decHk2} and \eqref{eq:struc2} can be assumed to hold for any Palais-Smale sequence $(u_\a)_\a$ for $I$. The difficulty in proving \eqref{eq:struc2} for sign-changing sequences $(\ua)_\a$ comes from the much richer nature of the profiles $v^i$ that may occur in the bubbling. We illustrate this with the following simple computation. 
  Let $\Va$ and $\Va[j]$ be two bubbles in the sense of Definition \ref{def:bub:int} (we assume that they are both interior for simplicity).  A straightforward computation shows that 
$$
 \text{ If } \eps_\a^{ij} \to + \infty \quad \text{ then } \quad  \langle \Va, \Va[j] \rangle_{H^k_0(\O)} \to 0 
 $$
as $\a \to + \infty$, where we consider the scalar product defined by \eqref{norm:Dk2}. Conversely, assume that $\eps_\a^{ij} \le C$ for all $\alpha$: by the definition of $\eps_\a^{ij} $ in \eqref{eq:struc2}, and up to passing to a subsequence, we can thus assume that
$ \ma = \ma[j](\mu_{j,i} + \smallo(1))$ and $\frac{\xa[j] - \xa}{\ma[j]} = z_{j,i} + \smallo(1)$ for some $\mu_{j,i} > 0$ and $z_{j,i} \in \R^n$. A change of variables then shows that
\begin{equation} \label{struwe5}
\langle \Va, \Va[j] \rangle_{H^k_0(\O)} = (\mu_{j,i})^{\frac{n-2k}{2}} \int_{\R^n} |v^i(x)|^{2^\sharp-2} v^i(x) v^j\big( \mu_{j,i} x + z_{j,i}) dx + \smallo(1). 
 \end{equation}
If $v^i$ and $v^j$ were both positive and equal to $B$ given by \eqref{def:eucBub} the right-hand side of \eqref{struwe5} would have a positive limit, and this would contradict the energy quantification \eqref{eq:decHk:bis2}. But if $v^i$ or $v^j$ are allowed to change sign, and due to abundance of sign-changing solutions of \eqref{eq:critRn} (we refer to subsection \ref{subsec:condition} above for the relevant citations), the existence of $\mu_{j,i}$ and $z_{j,i}$ for which the right-hand side of \eqref{struwe5} vanishes cannot be ruled out \emph{a priori}. When $(\ua)_\alpha$ changes sign we therefore cannot expect to recover \eqref{eq:struc2} solely from the quantification of the energy \eqref{eq:decHk:bis2}.

 \medskip
 
\begin{proof}[Proof that \eqref{eq:struc2} holds.]
We will prove that for any Palais-Smale sequence $\ua$ of $I$, there exists a bubble-tree $u_0, (\Va)_{\{1 \le i \le N \}}$ satisfying \eqref{eq:decHk2} and \eqref{eq:decHk:bis2} for which \eqref{eq:struc2} also holds. This bubble-tree will be chosen to contain a minimal number $N$ of bubbles in \eqref{eq:decHk2}. All the convergences in the proof will take place up to a subsequence. Let $(\ua)_\alpha$ be a Palais-Smale sequence for $I_\alpha$ that weakly converges towards $u_0 \in H^k_0(\O)$. If $\ua \to u_0$ in $H^k_0(\O)$ there is nothing to prove. We thus assume that $\Vert \ua - u_0 \Vert_{H^k_0(\O)} \not \to 0$ as $\alpha \to + \infty$ and we let 
\[ \begin{aligned}
\mathcal{N} = \Big \{& N \ge 1 \textrm{ such that there exist } N \textrm{ bubbles } \Va[1], \dots, \Va[N]  \\
& \textrm{ such that }  \quad  \left \Vert \ua - u_0 - \sum_{i=1}^N \Va \right \Vert_{H^k_0(\O)} \to 0  \Big \}.
\end{aligned}   \]
We allow both interior and boundary bubbles as in Definitions \ref{def:bub:int} and \ref{def:bub:bdr}. By \cite[Theorem 1.2]{Maz17} $\mathcal{N}$ is not empty. Letting $N = \min \mathcal{N}$ we obtain that there exist $N$ bubbles $[(\xa)_\a,(\ma)_\a,v^i]_{\{1 \le i \le N \}}$  satisfying 
\ben \label{struwe5bis} 
 \ua - u_0 - \sum_{i=1}^N  \Va \to 0 
 \een
 in $H^k_0(\O)$. We let $f(s) = |s|^{2^\sharp-2}s$ for all $s \in \R$. Let $i \in \{1, \dots, N\}$ be fixed and let 
\[ E(i) = \Big\{ j \in \{1, \dots, N\} \textrm{ such that }\eps_\a^{ij} = \bigO(1) \Big \} .\]
Clearly $i \in E(i)$. A first simple observation is that $E(i)$ only contains indices of bubbles of the same nature than $\Va$: that is, if $\Va$ is an interior bubble, all the indices in $E(i)$ correspond to other interior bubbles, and if $\Va$ is a boundary bubble all the indices in $E(i)$ correspond to other boundary bubbles. This simply follows from the observation that $\eps_\a^{ij} \to + \infty$ whenever $\Va$ is an interior bubble and $\Va[j]$ is a boundary one, which is a consequence of Definitions \ref{def:bub:int} and \ref{def:bub:bdr}. Indeed, if $\ma = \smallo(\ma[j])$ or $\ma[j] = \smallo(\ma)$ we obviously have $\eps_\a^{ij} \to + \infty$ by definition. And if $\ma \lesssim \ma[j] \lesssim \ma$, and since $\xa[j] \in \partial \O$, we have $|\xa - \xa[j]| \ge \dist{\xa, \partial \O} >> \ma$, hence again $\eps_\a^{ij} \to + \infty$. 

We will prove that $E(i) = \{i\}$ for all $i \in \{1, \dots, N\}$, which will prove that \eqref{eq:struc2} holds. If $j \in E(i)$ we let $\mu_{j,i}$ and $z_{j,i}$ be as in \eqref{struwe5}. In the case where $\Va$ and $\Va[j]$ are both boundary bubbles we have $z_{j,i} \in \partial \R^n_+$. We let 
$$v_{\mu_{j,i}, z_{j,i}}(x) = (\mu_{j,i})^{\frac{n-2k}{2}} v^j\big( \mu_{j,i} x + z_{j,i}),$$
which is defined for $x \in \R^n$ or $\R^n_+$ depending on the nature of $\Va$. We then have the following result:

\begin{lemma} \label{lemme:unebulledemoins}
We let
$$U_i  = \sum_{j \in E(i)} v_{\mu_{j,i}, z_{j,i}} ,$$
which is defined in $\R^n$ or $\overline{\R^n_+}$ depending on the nature of $\Va$. 
Then: 
\begin{itemize}
\item If $\Va$ is an interior bubble, $U_i$ is a solution of \eqref{eq:critRn}.
\item If $\Va$ is a boundary bubble, $U_i$ is a solution of \eqref{eq:crithf}.
\end{itemize}
\end{lemma}

\begin{proof}[Proof of Lemma \ref{lemme:unebulledemoins}]
Assume first that $\Va$ is an interior bubble. Let $\vp \in C^\infty_c(\R^n)$ and let, for all $x \in \O$,
\[ \vp_{i, \alpha}(x) =  (\ma)^{- \frac{n-2k}{2}} \vp \Big( \frac{x - \xa}{\ma}  \Big). \]
It is easily seen that $(\vp_{i, \alpha})_\alpha$ is bounded in $H^k_0(\O)$ and that $\vp_{i,\alpha} \rightharpoonup 0$ in $H^k_0(\O)$. Since $u_\a$ is a Palais-Smale sequence for $I$ we have 
\begin{equation*} 
 (-\Delta)^k \ua + \sum_{l=0}^{p} (-1)^l \nabla^l\big(A_{l}(\nabla^l\, \ua \,)\big) - |\ua|^{2^{\sharp}-2} \ua = \smallo(1) \quad \textrm{ in } H^{-k}(\O), 
 \end{equation*}
and integrating the latter against $\vp_{i,\alpha}$ gives, together with \eqref{struwe5bis},
\ben \label{struwe6}
 \int_\O \Big \langle \sum_{j=1}^N (- \Delta)^{\frac{k}{2}} \Va[j]  , (- \Delta)^{\frac{k}{2}} \vp_{i, \alpha} \Big \rangle dy = \int_\O f \Big( \sum_{j=1}^N\Va[j] \Big) \vp_{i, \alpha} dy + \smallo(1). 
\een
First, direct computations show that if $j \not \in E(i)$ then 
\begin{equation*} 
  \int_{\O} \langle (- \Delta)^{\frac{k}{2}}  \Va[j]  , (- \Delta)^{\frac{k}{2}}  \vp_{i, \alpha} \rangle dv_g = \smallo(1).
 \end{equation*}
Then, it is easily checked that $(\ma)^{\frac{n-2k}{2}} \Va[j](\xa + \ma \cdot) \to v_{\mu_{j,i}, z_{j,i}}$ in $D^{k,2}(\R^n)$ as $\alpha \to + \infty$. Combining these two elements yields, by a simple change of variables,
 \ben \label{struwe6:bis}
  \bal  &\int_\O  \Big \langle \sum_{j=1}^N (- \Delta)^{\frac{k}{2}} \Va[j]  , (- \Delta)^{\frac{k}{2}} \vp_{i, \alpha} \Big \rangle dy  \\
  & = \int_\O  \Big \langle \sum_{j\in E(i) } (- \Delta)^{\frac{k}{2}} \Va[j]  , (- \Delta)^{\frac{k}{2}} \vp_{i,\alpha} \Big \rangle dy + \smallo(1) \\
& = \int_{\R^n}  \Big \langle (- \Delta)^{\frac{k}{2}} U_i, (- \Delta)^{\frac{k}{2}} \vp \Big \rangle dy + \smallo(1)
\eal 
\een
as $\a \to + \infty$. Independently we have, using \eqref{est:bubble},
\[\bal & \Big| f \Big(\sum_{j=1}^N \Va[j] \Big) - f \Big(\sum_{j \in E(i)} \Va[j] \Big)  \Big|   \lesssim \Big( \sum_{j \not \in E(i)} \Ba[j] \Big)^{2^\sharp-1} + \Big( \sum_{j \in E(i)} \Ba[j] \Big)^{2^\sharp-2}\Big( \sum_{j \not \in E(i)} \Ba[j] \Big). \eal \]
When $j \not \in E(i)$, direct computations show that 
$$ \begin{aligned} 
& \int_\O (\Ba[j])^{2^\sharp-1} |\vp_{i, \alpha}| dv_g  = \smallo(1) \quad \text{ and } \\
& \int_\O \Big( \sum_{k \in E(i)} \Ba[k] \Big)^{2^\sharp-2} (\Ba[j]) |\vp_{i, \alpha}| dv_g  = \smallo(1)  \\
\end{aligned} $$ 
as $\a \to + \infty$, and H\"older's inequality and a change of variables thus show that 
\ben \label{struwe6:ter}  \bal \int_\O f \Big( \sum_{j=1}^N\Va[j] \Big) \vp_{i, \alpha} dy & = \int_\O  f \Big(\sum_{j \in E(i)} \Va[j] \Big) \vp_{i,\alpha} dy + \smallo(1) \\
& =  \int_{\R^n} f (U_i) \vp dy + \smallo(1). \eal 
\een
Combining \eqref{struwe6:bis} and \eqref{struwe6:ter} in \eqref{struwe6} and passing to the limit shows that $U_i$ is a weak, and hence a strong, solution of \eqref{eq:critRn}. This proves Lemma~\ref{lemme:unebulledemoins} if $\Va$ is an interior bubble. The proof when $\Va$ is a boundary bubble is identical by defining $ \vp_{i,\a}(x) =  \mu_\a^{-\exc} \vp\Big(\frac{\sigma_{x_\a}^{-1}(x)}{\mu_\a}\Big)$ and since $U_i \in D^{k,2}(\R^n_+)$. 
\end{proof}
We may now conclude the proof of \eqref{eq:struc2}. Let $x \in \overline{\O}$. If $\Va$ is an interior bubble we let 
\[ U_{i, \alpha}(x) = \chi\Bigg(\frac{x-\xa}{\dist{\xa,\dO}}\Bigg) (\ma)^{-\exc} U_i\Big(\frac{x-\xa}{\ma}\Big),\]
while if $\Va$ is a boundary bubble we let 
\[ U_{i, \alpha}(x) =  \chi\big(\sigma_{\xa}^{-1}(x)\big) (\ma)^{-\exc} U_i\Big(\frac{\sigma_{\xa}^{-1}(x)}{\ma}\Big).\]
It is easily seen that $U_{i,\alpha}  = \sum_{j \in E(i)} \Va[j](x) + \smallo(1)$ in $H^k_0(\O)$, and the initial Struwe decomposition \eqref{struwe5bis} thus becomes
\[ \ua = u_0 + U_{i,\alpha} + \sum_{j \not \in E(i)} \Va[j] + \smallo(1) \quad \textrm{ in } H^k_0(\O). \]
Since $U_i$ is a solution of \eqref{eq:critRn}, this contradicts the minimality of $N$ in the case where $\# E(i) \ge 2$, regardless of whether $U_i \equiv 0$ or not. Hence $E(i) = \{i\}$. This holds true for any $i \in \{1, \dots, N\}$ and proves \eqref{eq:struc2}. 
\end{proof}

\section{Technical results} \label{app:technicalresults}

\subsection{Yet another version of Giraud's Lemma}

We state a version of Giraud's lemma that we frequently used in this paper: 

\begin{lemma}[Giraud's Lemma]\label{prop:giraud}
 Let $\oO$ be a bounded open subset of $\R^n, n \ge 3$. Let $\gamma\in \R$ and $\beta >0$ satisfying $\beta+\gamma<n$ and let $0<\mu<1$. Let $X \in C^0(\oO \times \oO)$, $Y\in C^0(\oO \times \oO \setminus \{(x,x)\,:\,x\in \oO\})$ be such that
  \begin{align*}
    |X(x,y)| \leq (\mu + |x-y|)^{\gamma-n}  \quad \text{ and } \quad 
    |Y(x,y)| \leq |x-y|^{\beta-n} 
  \end{align*}
for all $x \neq y \in \oO$.  Define $Z(x,y) = \intM[z]{\O}{X(x,z)Y(z,y)}$. Then $Z\in C^0(\oO \times \oO)$ and there exists $C>0$ such that 
  \[  |Z(x,y)| \leq C \begin{cases}
    \mu^\gamma (\mu+|x-y|)^{\beta-n} & \text{if } \gamma < 0\\
    (\mu+|x-y|)^{\beta-n} \left(1+ \abs{\log\Big(\frac{\mu+|x-y|}{\mu}\Big)}\right) & \text{if } \gamma = 0\\
    (\mu+|x-y|)^{\beta+\gamma-n} & \text{if } \gamma >0
  \end{cases}.
  \]
\end{lemma}
The proof of Lemma \ref{prop:giraud} easily follows from an adaptation of the proof of classical Giraud-type inequalities. See for instance the arguments in \cite[Lemma 7.5]{Heb14} or \cite[Lemma 7.1]{CarRob25}.

\subsection{Some integral estimates} In this subsection we use the notations of Sections \ref{sec:bulles} and  \ref{sec:lin}. We state several bubble-tree estimates. They are all adaptations of well-known results in the case $k=1$ but we gather them here for clarity.

\begin{lemma} \label{lemme:ordre:deux}
Let $i \in \{1, \ldots, N\}$. We have, for any $0 \le l \le 2k-1$,
\begin{equation*} 
  \inty{|x-y|^{2k-n-l} \ta(y) \Ba(y)^{\crit-1}} \lesssim \ma \ta(x)^{-l} \Ba(x) 
\end{equation*}
as $\alpha \to + \infty$, uniformly in $x \in \oO$.
\end{lemma}

\begin{proof}
This is a direct application of Lemma~\ref{prop:giraud}.
\end{proof}

\begin{lemma} \label{prop:lemtrou:0}
Let $i \in \{1, \ldots, N\}$. We have, for any $0 \le l \le 2k-1$,
\begin{equation}\label{eq:IIa2}
  \inty{|x-y|^{2k-n-l} \Ba(y)^{\crit-2}} = \smallo \big(1+\ta(x)^{-l}\Ba(x) \big)
\end{equation}
as $\alpha \to + \infty$, uniformly in $x \in \oO$.
\end{lemma}

\begin{proof}
A direct computation using Lemma \ref{prop:giraud} shows that for all $x\in \oO$
\begin{equation*}  
  \inty{|x-y|^{2k-n-l} \Ba(y)^{\crit-2}} \lesssim\begin{cases}
  (\ma)^{\exc} \ta(x)^{-l}\Ba(x) & \text{when $n<4k$}\\
  (\ma)^{2k-1} \ta(x)^{1-2k-l} & \text{when $n=4k$}\\
  (\ma)^{2k} \ta(x)^{-2k-l} & \text{when $n>4k$}
\end{cases} .
\end{equation*}
When $n < 4k$ \eqref{eq:IIa2} thus follows. Assume now that $n > 4k$. We have 
\begin{equation*} 
 \begin{aligned}
 (\ma)^{2k} \ta(x)^{-2k-l} & \lesssim (\ma)^{3k - \frac{n}{2}} \ta(x)^{n-4k} \ta(x)^{-l}\Ba(x) \\
 & \lesssim  \left \{ \begin{aligned} & (\ma)^{k} \ta(x)^{-l}\Ba(x) & \text{ if } \ta(x) \le (\ma)^{\frac12} \\ & 
 (\ma)^{k - \frac{l}{2}}   & \text{ if } \ta(x) \ge (\ma)^{\frac12} \end{aligned} \right\} \\
 & \lesssim (\ma)^{\frac12} (1 + \ta(x)^{-l}\Ba(x) ),  
\end{aligned} 
\end{equation*}
which proves \eqref{eq:IIa2}. The case $n=4k$ is similar. 
\end{proof}

\begin{lemma}\label{prop:lemtrou}
Let $i \in \{1, \ldots, N\}$ and $0 \le l \le 2k-1$. Let $(M_\a)_\a$ be a sequence of positive numbers such that $M_\a \geq 1$. Then  for all $x\in \oO$, 
  \[ \inty[\O\setminus B(\xa,M_\a\ma)]{|x-y|^{2k-n-l} \pBa^{\crit -1}} \lesssim \frac{1}{M_\a^{2k}} \ta(x)^{-l} \Ba(x). 
  \] 
\end{lemma}
\begin{proof}
The proof is an adaptation of \cite[Lemma 7.6]{Heb14} for the case $k=1$.
\end{proof}

\begin{lemma}\label{prop:lemBiBj}
  Let $i\neq j\in \{1,\ldots N\}$. There exists a sequence $\eps_\a \to 0$ as in \eqref{def:ea} such that 
  \[    (\ma)^{\exc}(\ma[j])^{\exc} \intM{\O}{\ta(y)^{k-n}\ta[j](y)^{k-n}}  \leq \eps_\a. \]
Let $p \in \{0, \cdots, k-1\}$ and assume that $n > 4k-2p$. Then 
  \[ (\ma)^{\exc}(\ma[j])^{\exc} \intM{\O}{\ta(y)^{2k-p-n}\ta[j](y)^{2k-p-n}}   \leq \eps_\a  \big(\ma \ma[j]\big)^{k-p}. \]
\end{lemma}
\begin{proof}
The first inequality follows from an adaptation of the proof of estimate (A.6) in \cite[Proposition A.1]{Pre24} for the case $k=1$. We prove the second one. Without loss of generality we may assume that $\ma[j] \le \ma$. Since $\ta[j](y) \ge |\xa[j]-y|$ and since $n >2(2k-p)$ we may apply Lemma \ref{prop:giraud}
to get 
$$ \begin{aligned}
 & (\ma)^{\exc}(\ma[j])^{\exc} \intM{\O}{\ta(y)^{2k-p-n}\ta[j](y)^{2k-p-n}} \\
 &\lesssim \frac{ (\ma)^{\exc}(\ma[j])^{\exc}}{(\ma + |\xa - \xa[j]|)^{n-4k+2p}  } \\
 & \lesssim \left(\frac{\ma \ma[j]}{(\ma)^2 + |\xa - \xa[j]|^2}\right)^{\frac{n-4k+2p}{2}} \big(\ma \ma[j]\big)^{k-p} \lesssim  \big( \eps_\a^{ij} \big)^{ -\frac{n-4k+2p}{2}} \big(\ma \ma[j]\big)^{k-p},
 \end{aligned} $$
and the result follows from \eqref{eq:struc}. 
\end{proof}

\subsection{Proof of Lemma~\ref{prop:lem2}}

We finally give a proof of Lemma \ref{prop:lem2}:

\begin{proof}[Proof of Lemma \ref{prop:lem2}] 
By \eqref{def:Ba} we have $\Ba[0] \equiv 1$. Using \eqref{def:PSI} we thus have
\begin{equation*}
\bal
  \Psi_{\a}(y) & =  \sumi \ta(y)^{2-2k}\Ba(y)+ \sum_{1 \le j \le N} \Big( \Ba[j](y)^{\crit-2} + \Ba[j](y) \Big) \\
  & +  \sum_{\substack{i,j = 1,\ldots, N\\i\neq j}} \Ba[j](y)^{\crit-2} \Ba(y).
\eal
\end{equation*}
We let $0 \le l \le 2k-1$ be fixed and integrate each term separately against $|x-y|^{2k-n-l}$. Let $i\in \{1,\ldots N\}$. First, a direct application of Lemma \ref{prop:giraud} shows that
\[  \begin{aligned} 
\inty{|x-y|^{2k-n-l}\ta(y)^{2-2k}\Ba(y)} &\lesssim 
\big(\ta(x)^{2}\log \ta(x) \big)\ta(x)^{-l} \Ba(x) \\
&= \smallo\big(1+\ta(x)^{-l}\Ba(x)\big)
\end{aligned} 
\]
uniformly for $x\in\oO$. Since $\ta(x)^{2}\log \ta(x) \lesssim \ta(x)$, the last equality follows again from the following observation: 
\begin{equation*} 
 \begin{aligned}
 \ta(x) \ta(x)^{-l} \Ba(x)
 & \lesssim \left \{ \begin{aligned} & (\ma)^{\frac{n-2k}{2(n-2k+l)}} \ta(x)^{-l}\Ba(x) & \text{ if } \ta(x) \le (\ma)^{\frac{n-2k}{2(n-2k+l)}} \\ & (\ma)^{\frac{n-2k}{2(n-2k+l)}}   & \text{ if } \ta(x) \ge (\ma)^{\frac{n-2k}{2(n-2k+l)}}  \end{aligned} \right\} \\
 & \lesssim (\ma)^{\frac{n-2k}{2(n-2k+l)}}  (1 + \ta(x)^{-l}\Ba(x) ).  
\end{aligned} 
\end{equation*}
 Similarly, arguing as in the proof of Lemma~\ref{prop:lemtrou:0} and using Lemma \ref{prop:giraud} we get that 
 \[  \inty{|x-y|^{2k-n-l}\Ba(y)} = \smallo\big(1+\ta(x)^{-l}\Ba(x)\big)
\]
on $\oO$, and the term involving $ \sum_{1 \le j \le N} (\Ba[j](y))^{\crit-2}$ in $\Psi_\a$ is estimated using \eqref{eq:IIa2}. Thus we only need to estimate 
\[  \inty{|x-y|^{2k-n-l} \Ba[j](y)^{\crit-2}\Ba(y)}
\]
for all $i \neq j$ in $\{1,\ldots, N\}$. The conclusion then follows from considering separately every possible configuration of $\ma,\ma[j], |\xa-\xa[j]|$ satisfying \eqref{eq:struc}, by a straightforward adaptation of the proof for the case $k=1$ in \cite[Proposition A.1]{Pre24}. 
\end{proof}

\bibliographystyle{alpha}
\bibliography{reference}

\end{document}